\documentclass[reqno,11pt]{amsart}
\usepackage{amsmath,amsfonts,amssymb,amsxtra,latexsym,amscd,enumerate,amsthm,verbatim}

\usepackage[notcite,notref]{}
\usepackage{hyperref}
\usepackage{graphicx}
\usepackage{xcolor}
\usepackage{comment}
\usepackage{bbm}

\usepackage[margin=1.40in]{geometry}
\numberwithin{equation}{section}

\newcommand{\R}{\mathbb{R}}

\newcommand{\A}{\mathcal{A}}

\newcommand{\G}{\mathcal{G}}
\newcommand{\T}{\mathbb{T}}

\newcommand{\Z}{\mathbb{Z}}
\newcommand{\eps}{\epsilon}
\newcommand{\ib}{b^{-1}}

\numberwithin{equation}{section} 

\newcommand{\supp}{\text{supp}}
\newcommand{\la}{\langle}
\newcommand{\ra}{\rangle}

\newtheorem{theorem}{Theorem}[section]
\newtheorem{lemma}[theorem]{Lemma}
\newtheorem{proposition}[theorem]{Proposition}

\newtheorem{definition}[theorem]{Definition}
\newtheorem{remark}[theorem]{Remark}

\DeclareMathOperator{\sign}{sign}
\newcommand{\norm}[1]{\left\lVert#1\right\rVert}

\begin{document}
	
	\title{Singularity structures of linear inviscid damping in a channel}
	
	
	
	\author{Wenjie Lu}
	\address{University of Minnesota}
	\email{lu000005@umn.edu}
	
	\begin{abstract}
		{\small}
		This paper studies singularity structures of the linear inviscid damping of two-dimensional Euler equations in a finite periodic channel. We introduce a recursive definition of singularity structures which characterize the singularities of the spectrum density function from different sources: the free part and the boundary part of the Green function. As an application, we demonstrate that the stream function exhibits smoothness away from the channel's boundary, yet it presents singularities in close proximity to the boundary. The singularities arise due to the interaction of boundary and interior singularities of the spectrum density function. We also show that the behavior of the initial data and background flow have an impact on the regularity of different components of the stream function.
	\end{abstract}
	
	\maketitle
	
	\setcounter{tocdepth}{1}
	
	
	\section{Introduction}
\subsection{Main equations} 
	Consider the two dimensional Euler equation linearized around a shear flow $(b(y),0)$, in the periodic channel $ (x,y)\in \mathbb{T}\times[0,1]$:
	\begin{equation}\label{Main1}
		\begin{split}
			&\partial_t\omega+b(y)\partial_x\omega-b''(y)u^y=0,\\
			&{\rm div}\,u=0\qquad{\rm and}\qquad \omega=-\partial_yu^x+\partial_xu^y,
		\end{split}
	\end{equation}
	with the natural non-penetration boundary condition $u^y|_{y=0,1}=0$. 
	
	For the linearized flow, 
	$\int\limits_{\mathbb{T}\times[0,\,1]}u^x(x,y,t)\,dxdy$ and $ \int\limits_{\mathbb{T}\times[0,\,1]}\omega(x,y,t)\,dxdy$
	are conserved quantities. In this paper, we will assume that
 \begin{equation}\label{EQ:132}
     \int_{\mathbb{T}\times[0,1]}u_0^x(x,y)\,dxdy=\int_{\mathbb{T}\times[0,1]}\omega_0\,dx dy=0.
 \end{equation}
	These assumptions can be dropped by adjusting $b(y)$ with a linear shear flow $C_0y+C_1$.
	Then one can see from the divergence free condition on $u$ that 
	there exists a stream function $\psi(t,x,y)$ with $\psi(t,x,0)=\psi(t,x,1)\equiv 0$, such that 
	\begin{equation}\label{eqS1}
		u^x=-\partial_y\psi,\,\,u^y=\partial_x\psi.
	\end{equation}
	The stream function $\psi$ can be solved through
	\begin{equation}\label{eq:equationStream}
		\Delta\psi=\omega, \qquad \psi|_{y=0,1}=0.
	\end{equation}
	We summarize our equations as follows
	\begin{equation}\label{main}
		\left\{\begin{array}{ll}
			\partial_t\omega+b(y)\partial_x\omega-b''(y)\partial_x\psi=0,&\\
			\Delta \psi(t,x,y)=\omega(t,x,y),\qquad \psi(t,x,0)=\psi(t,x,1)=0,&\\
			(u^x,u^y)=(-\partial_y\psi,\partial_x\psi),&\\
			\omega(0, x, y) = \omega_0(x,y).&
		\end{array}\right.
	\end{equation}
	for $ t\ge0, (x,y)\in\mathbb{T}\times[0,1]$. 
	
	An important property for (\ref{main}) is that the evolution for each mode is de-coupled. Indeed, define for $k \in \mathbb{Z}$, $y \in [0, 1]$ and $t\ge0$ the Fourier modes $\omega_k, \psi_k$ for the vorticity and stream functions $\omega$ and $\psi$ as
	\begin{equation}
		\begin{split}
			\omega_k(t, y) &:= \int_{\T}e^{-ikx}\omega(t, x, y)\,dx, \quad \omega^k_0(y) := \int_{\T}e^{-ikx}\omega_0(x, y)\,dx, \\ \quad \psi_k(t, y) &:= \int_{\T}e^{-ikx}\psi(t, x, y)\,dx.
		\end{split}
	\end{equation}
	Then the equation \eqref{main} can be written as
	\begin{equation}\label{maink}
		\left\{\begin{array}{ll}
			\partial_t\omega_k+ikb(y)\omega_k-ikb''(y)\psi_k = 0,&\\
			(\partial_y^2 - k^2) \psi_k(t,y)=\omega_k(t,y),\qquad \psi_k(t,0)=\psi(t,1)=0,&\\
			\omega_k(0,y) = \omega_0^k(y),&
		\end{array}\right.
	\end{equation}
	for $k\in\Z, y\in[0,1], t\ge0$. For each $k \in \mathbb{Z}\setminus\{0\}$, we set for any $g \in L^2(0,1),$
	\begin{equation} \label{F3.1}
		L_kg(y) = b(y)g(y) + b''(y)\int_0^1G_k(y, z)g(z)dz,
	\end{equation}
	where $G_k$ is the Green's function for the operator $k^2-\frac{d^2}{dy^2}$ on $(0, 1)$ with zero Dirichlet boundary condition. Our main assumption is that the linearized operator $L_k$ does not have embedded eigenvalues. Our goal is to obtain a precise asymptotic decomposition of the stream function $\psi_k(t,y), y\in [0,1]$, as $t\to+\infty$.

	\subsection{Background on hydrodynamic stability and inviscid damping}\label{intro to inviscid}

	Hydrodynamical stability is a classical topic in mathematical analysis of fluid flows, pioneered by prominent figures such as Rayleigh \cite{rayleigh1895stability}, Kelvin \cite{kelvin1887stability}, Orr \cite{orr1907stability}, among many others. The main focus was to study stability of important physically relevant flows, such as shear flows and vortices.
	
	In this paper we consider shear flows. There are extensive works on the linear stability property of these flows. In particular, Rayleigh \cite{rayleigh1895stability} proved that shear flows with no inflection points are spectrally stable. Orr \cite{orr1907stability} in 1907 observed the $t^{-1}$ decay rate of the velocity when the shear flow is Couette (linear shear), and Case \cite{case1960stability} provided a formal proof in the case of a finite channel. See also Lin and Zeng \cite{lin2011inviscid} for a sharp version with optimal dependence on the regularity of the initial data. 
	
	The observation of Orr can be described roughly as follows. Consider the linearized equation near Couette flow:
	$$\partial_t\omega+y\partial_x\omega=0, \,\,(x,y)\in\mathbb{T}\times \R.$$
	
	One can solve this equation explicitly and it follows that
	$\omega(t,x,y)=\omega_0(x-yt,y).$
	The equation for the stream function becomes
	$\Delta\psi(t,x,y)=\omega(t,x,y)=\omega_0(x-yt,y)$
	for $(x,y)\in \mathbb{T}\times \R$ and therefore
	\begin{equation}\label{decaycostderivative}
		\widetilde{\psi}(t,k,\xi)=-\frac{\widetilde{\omega}(t,k,\xi)}{k^2+|\xi|^2}=-\frac{\widetilde{\omega_0}(k,\xi+kt)}{k^2+|\xi|^2}.
	\end{equation}
	In the above, $\widetilde{h}$ denotes the Fourier transform of $h$ in $x,\,y$.
	Assume that $\omega_0$ is smooth, so $\widetilde{\omega}_0(k,\xi)$ decays fast in $k,\,\xi$. Then we can view $\xi$ as
	$$\xi=-kt+O(1),$$ 
	and hence $\widetilde{\psi}(t,k,\xi)$ decays like $|k|^{-2}\langle t\rangle^{-2}$ for each $k\neq 0$. Similarly, using the relations $u^x=-\partial_y\psi$ and $u^y=\partial_x\psi,$
	we conclude that $\widetilde{u^x}$ decays like $|k|^{-1}\langle t\rangle^{-1}$ and $\widetilde{u^y}$ decays like $|k|^{-1}\langle t\rangle^{-2}$ for all $k\neq0$. Hence, the velocity field decays to another shear flow $(u_{\infty}(y),0)$. 
	

	For general monotone shear flows, the linearized operator becomes more complicated due to the extra term $b''(y)\partial_x\psi$, see \eqref{main}, which can not be treated as perturbations. Therefore spectral analysis of the linearized operator is required to understand the dynamical properties of the associated flow. For results on the general spectral property of the linearized operator, we refer to Faddeev \cite{faddeev1971theory} and Lin \cite{lin2003instability}. In the direction of inviscid damping, Stepin \cite{stepin1995nonself} proved $t^{-\nu}$ decay of the stream function associated with the continuous spectrum, Rosencrans and Sattinger \cite{rosencrans1966spectrum} proved $t^{-1}$ decay for analytic monotone shear flows. 
	
	Recently, inspired by the remarkable work of Bedrossian and Masmoudi \cite{bedrossian2015inviscid} on the nonlinear asymptotic stability of shear flows close to the Couette flow in $\mathbb{T}\times\R$ (see also an extension \cite{ionescu2020inviscid} to $\mathbb{T}\times[0,1]$), optimal decay estimates for the linear problem received much attention, see e.g.  Zillinger \cite{zillinger2016linear,zillinger2017linear} and references therein for shear flows close to Couette.  In an important work, Wei, Zhang and Zhao \cite{wei2018linear} obtained the optimal decay estimates for the linearized problem around monotone shear flows, under very general conditions. In \cite{jia2020linear} the first author identified the main term in the asymptotics of the stream function. 
		
	We also refer the reader to important developments for the linear inviscid damping in the case of non-monotone shear flows \cite{bouchet2010large,wei2019linear, wei2020linear,ionescu2022linear} and circular flows \cite{bedrossian2019vortex,coti2019degenerate}. See also Grenier et al \cite{grenier2020linear} for an approach using methods from the study of Schr\"odinger operators. 
	
       Lastly, we mention the result of Zillinger \cite{zillinger2016linear} which addresses the boundary effect in the dynamics of \eqref{maink}, and is perhaps the closest to our results below. In particular, he showed that it is in general not possible to obtain a uniformly smooth ``profile" in high Sobolev spaces over $t\in[0,\infty)$. The main assumption in \cite{zillinger2016linear} is that the background shear flow is close to the Couette flow. In our case, we need to perform detailed spectral analysis if such an assumption is dropped, using and refining methods introduced in \cite{jia2020linear}.

	\subsection{ Main results}\label{main result 1}
	Without loss of generality, we assume the mode $k \ge 1$. In fact, the equations are trivial for $k = 0$, and the case $k > 0$ and $k < 0$ are equivalent up to complex conjugation.
Our main goal is to understand the more refined asymptotics of the linearized flow \eqref{maink} as $t\to+\infty$, and in particular, to capture precisely the singularity structure of the so-called spectral density functions when the spectral parameter collapses into the continuous spectrum. As a consequence, we obtain an expansion of the stream function $\psi_k(t,y)$, $y\in [0,1], t\ge0$ to all orders, as $t\to\infty$, see Theorem \ref{interior regularity}- Theorem \ref{regularity near b(0)}. As we shall see, the boundary $y\in\{0,1\}$ plays an important role in the long time behavior of $\psi_k(t,y), y\in [0,1], t\ge0$, and it is the primary difficulty in the study of the refined dynamics of \eqref{maink}. 
	
	In order to capture the optimal dependence on the frequency $k$ in different estimates, we define for integer $n\ge 1$ the weighted Sobolev norm $\norm{h}_{H^n_k(\mathbb{R})}$ for $h\in H^n(\R)$ as 
	\begin{equation}\label{weighted norm}
		\norm{h}_{H^n_k(\mathbb{R})} := \sum_{\alpha = 0}^n |k|^{n - \alpha}\norm{\partial^\alpha_v h}_{L^2(\mathbb{R})}.
	\end{equation}
	Direct computation shows that
	\begin{equation}
		\norm{h}_{H^n_k(\mathbb{R})} \lesssim |k|^{-m}\norm{h}_{H^{n + m}_k(\mathbb{R})}
	\end{equation}
	for all $m\ge 0$.

	To state main theorems, we first define the smooth cut-off functions $$\chi^{in}(y), \chi^{b0}(y), \chi^{b1}(y) \in C^\infty_c(\R)$$ such that for $y \in \R$,
	\begin{equation}\label{in cutoff}
		\chi^{in}(y) = \begin{cases}
			&1, ~~~ \text{if}~~~ y \in [1/8, 7/8]\\
			&0, ~~~ \text{if}~~~ y \notin [1/16, 15/16],
		\end{cases}
	\end{equation}
	\begin{equation}\label{b0 cutoff}
		\chi^{b0}(y) = \begin{cases}
			&1, ~~~\text{if}~~~ y\in [-1/8, 1/8]\\
			&0, ~~~\text{if}~~~ y\notin [-1/4, 1/4],
		\end{cases}
	\end{equation}
	and
	\begin{equation}\label{b1 cutoff}
		\chi^{b1}(y) = \chi^{b0}(1-y).
	\end{equation}
	
	Denote $I:= [0, 1]$. Our first main result is the following characterization of the stream function in the interior of the channel. 
	\begin{theorem}\label{interior regularity}
		For $y\in I$, $k \ge 1$ and $N\in\Z\cap[3,\infty)$, assume that $\omega_{0}^k\in H^N(I)$. Then there exist functions $\alpha^{in}(t, \cdot, k)$, $\beta^{in}(t, \cdot, k), \gamma^{in}(t, \cdot, k)\in C([1,\infty), H^{N-3}(I))$, such that the unique solution $\psi_k\in C([0,\infty), H^{N+2}(I))$ satisfies for for $y \in I$ and $t \ge 1$,
		\begin{equation}\label{decomp of interior part}
			\psi_k(t, y)\chi^{in}(y) = \frac{e^{-ikb(y)t}}{k^2 t^2}\alpha^{in}(t, y, k) + \frac{e^{-ikb(0)t}}{k^2t^2}\beta^{in}(t, y, k)+ \frac{e^{-ikb(1)t}}{k^2t^2}\gamma^{in}(t, y, k),
		\end{equation}
		and for every integer $m\in\Z\cap[0,N-3]$ and $t\ge1$,
		\begin{equation}\label{control 116}
			\norm{\alpha^{in}(t, \cdot, k)}_{H^m_k(I)} + \norm{\beta^{in}(t, \cdot, k)}_{H^m_k(I)} +\norm{\gamma^{in}(t, \cdot, k)}_{H^m_k(I)}\lesssim_{m} \norm{\omega_0^k}_{H^{m + 3}_k(I)}.
		\end{equation}
	\end{theorem}
	
	\begin{remark}\label{remarkIR}
	Theorem \ref{interior regularity} implies the existence of not one, but {\it three profiles}. The profile $\alpha^{in}$ is the ``main profile" when no boundary effect is present, see for example \cite{jia2020linear}. The profiles $\beta^{in}, \gamma^{in}$ come from the boundary contribution. They depend on boundary values of derivatives of the initial vorticity $\omega_0^k$ and background $b''$. In general,  $\beta^{in}, \gamma^{in}$ do not decay over time with a faster rate, if $\omega_{0}^k$ and $b''$ do not vanish at the boundary. If $\omega_0^k$ and $b''$ vanish to some order at the boundary, we can show that we can get more decay in time $\frac{1}{t^{2 + p}}$ for the $\beta^{in}$ and $\gamma^{in}$ components, and they also have higher regularity $H^{m+p}_k(I)$ in the estimate (\ref{control 116}). Here the positive integer $p$ is determined by to which order $\omega_0^k$ and $b''$ are vanishing. 
	\end{remark}

	Next we turn to the behavior of $\psi_k(t,y), y\in I, t\ge1$ near $y=b(0)$, the case of $y=b(1)$ being similar. 
	\begin{theorem}\label{regularity near b(0)}
		For any integers $N\ge 3$, $k \ge 1$, and $y \in [0, 1]$, assume $\omega_0^k(y) \in H^N_k(I)$. Then there exist functions $\alpha^{b0, N}(t, y, k)$, $\beta^{b0, N}(t, y, k), R^{b0, N}(t, y, k)\in C([1,\infty), L^2(I))$ such that the unique solution $\psi_k\in C([0,\infty), H^{N+2}(I))$ satisfies for for $y \in I$ and $t \ge 1$,
		\begin{equation}\label{expb1}
			\psi_k(t, y)\chi^{b0}(y) = \frac{e^{-ikb(y)t}}{k^2 t^2}\alpha^{b0, N}(t, y, k) + \frac{e^{-ikb(0)t}}{k^2t^2}\beta^{b0, N}(t, y, k) + \frac{R^{b0, N}(t, y, k)}{(kt)^{N-1}}
		\end{equation}
		with the following properties.
		\begin{enumerate}
			\item $R^{b0, N}(t, y, k)$ satisfies the bounds for $t\ge1$,
			\begin{equation}
				\norm{R^{b0, N}(t, \cdot, k)}_{H^1_k(I)} \lesssim_N (1 + \log^{2(N-3)}\la t\ra) \norm{\omega_0^k}_{H^N_k(I)}.
			\end{equation}
			
			\item There exist functions $\alpha_j(t, y, k)\in C([1,\infty), L^2(I))$ for $1 \le j \le N-2$, satisfying the bounds for $t \ge 1$ and $j\in[2,N-2]$,
			\begin{equation}\label{expb3}
				\norm{\alpha_1(t, \cdot, k)}_{H^1_k(I)} \lesssim \norm{\omega_0^k}_{H^3_k(I)}
			\end{equation}
   and 
   \begin{equation}\label{expb2}
       \norm{\alpha_j(t, \cdot, k)}_{L^2(I)} \lesssim (1 + \log^{2(j-1)}\la t\ra)\norm{\omega_0^k}_{H^{j+2}_k(I)}
   \end{equation}
                          such that for $y \in I$ and $t \ge 1$, 
			                         \begin{equation}
				\alpha^{b0, N}(t, y, k) = \sum_{j = 1} ^ {N-2} \frac{\alpha_j(t, y, k)}{(kt)^{j-1}}.
			\end{equation}\label{expb4}
      Furthermore, if $\omega_0^k(0) = 0$, we have $\alpha_1(t, y, k) \in H^2_k(I)$ and $\alpha_2(t, y, k) \in H^1_k(I)$.
	                \item There exist functions $\beta_j(t, y, k)\in C([1,\infty), L^2(I))$ for $1 \le j \le N-2$, such that for $y \in I$ and $t \ge 1$,
	                satisfying the bounds for $t \ge 1$ and $j\in[2,N-2]$,
			\begin{equation}
				\norm{\beta_1(t, \cdot, k)}_{H^1_k(I)} \lesssim \norm{\omega_0^k}_{H^3_k(I)}
			\end{equation}
   and
   \begin{equation}\label{expb5}
       \norm{\beta_j(t, \cdot, k)}_{L^2(I)} \lesssim (1 + \log^{2(j-1)}\la t\ra)\norm{\omega_0^k}_{H^{j+2}_k(I)}
   \end{equation}
			                          such that for $y \in I$ and $t \ge 1$,
			                          \begin{equation}
				\beta^{b0, N}(t, y, k) = \sum_{j = 1} ^ {N-2} \frac{\beta_j(t, y, k)}{(kt)^{j-1}}.
			\end{equation}
			
			Furthermore, if $\omega_0^k(0) = 0$, we have $\beta_1(t, y, k) \equiv 0 $ for $y\in I$, and $\beta_2(t, \cdot, k) \in H^1_k(I)$.
			
		\end{enumerate}
	\end{theorem}
	
	\begin{remark}\label{remarkb0}
		We note that the expansion \eqref{expb1}-\eqref{expb5} allows us to identify the asymptotics of the stream function to all orders. The regularity bounds \eqref{expb3} and \eqref{expb5} on the coefficient functions $\alpha_j, \beta_j$ are sharp in general, unless we assume more vanishing conditions on the initial vorticity $\omega_0^k$ on the boundary. For example, one of the contributing factor to the term $\beta_2(t,y,k)$ is given by the integral for $y\in I, t\ge1$ with $v = b(y)$, 
		\begin{equation}\label{remarkb01}
		\lim_{\epsilon\to0+}\int_{\R} e^{-ik(w-b(0))t}\varphi_k(v,w)\frac{\log(v+i\epsilon)}{b(0)-w+i\epsilon}\,dw,
		\end{equation}
		where $\varphi\in C_0^\infty(\R^2)$. It follows that it is not possible to control \eqref{remarkb01} in $H^1$. Similar considerations apply for $\alpha_2$. On the other hand, if we assume that $\omega_0^k$ and $b''$ vanish to a high order, then the control we have on $\alpha_j, \beta_j$ can be improved significantly. 
		
				\end{remark}
				
The proof of theorems is based on the study of the singularity structure of the spectral density functions, defined in section \ref{secSDF}. The stream function can be written as an oscillatory integral of the spectral density function, using standard spectral analysis. To obtain the precise singularity structure of the spectral density functions, we make a change of coordinate and extend the problem to the whole space, to simplify the specific form of the singularities both in the interior and near the boundary. The desired characterization of the singularity structures, see section \ref{secSSSDF}, follows from the limiting absorption principle and a refined analysis of the singularities of Rayleigh equations. 
 \section{Preliminary}
 \subsection{Spectral density function}\label{secSDF}
With the linearized operator $L_k$ defined in \eqref{F3.1},  the equation (\ref{maink}) can be formulated as
	\begin{equation}\label{F3.2}
		\partial_t\omega_k+ikL_k\omega_k = 0, \quad \omega_k(0, y) = \omega_0^k(y).
	\end{equation}
 By standard theory of spectral projection, we then have
	\begin{equation}\label{F4}
		\begin{split}
			\omega_k(t,y)&=\frac{1}{2\pi i}\lim_{\epsilon\to0+}\int_{\R}e^{i\lambda t}\left[(\lambda+kL_k-i\epsilon)^{-1}-(\lambda+kL_k+i\epsilon)^{-1}\right]\omega_0^k\,d\lambda\\
			&=\frac{1}{2\pi i}\lim_{\epsilon\to0+}\int_{0}^1e^{-ikb(y_0) t}|b'(y_0)|\times \bigg[(-b(y_0)+L_k-i\epsilon)^{-1} \\
   &\qquad \qquad \qquad -(-b(y_0)+L_k+i\epsilon)^{-1}\bigg]\omega_0^k\,dy_0.
		\end{split}
	\end{equation}
	We then obtain
	\begin{equation}\label{F5}
		\begin{split}
			\psi_k(t,y)&=-\frac{1}{2\pi i}\lim_{\epsilon\to0+}\int_{0}^1e^{-ikb(y_0) t}|b'(y_0)|\int_0^1G_k(y,z)\\
			&\hspace{1in}\times\bigg\{\Big[(-b(y_0)+L_k-i\epsilon)^{-1}-(-b(y_0)+L_k+i\epsilon)^{-1}\Big]\omega_0^k\bigg\}(z)\,dz dy_0\\
			&=-\frac{1}{2\pi i}\lim_{\epsilon\to0+}\int_{0}^1e^{-ikb(y_0) t}|b'(y_0)|\left[\psi_{k,\epsilon}^{-}(y,y_0)-\psi_{k,\epsilon}^{+}(y,y_0)\right]dy_0.
		\end{split}
	\end{equation}
	In the above, 
	\begin{equation}\label{F6}
		\begin{split}
			&\psi_{k,\epsilon}^{+}(y,y_0):=\int_0^1G_k(y,z)\Big[(-b(y_0)+L_k+i\epsilon)^{-1}\omega_0^k\Big](z)\,dz,\\
			&\psi_{k,\epsilon}^{-}(y,y_0):=\int_0^1G_k(y,z)\Big[(-b(y_0)+L_k-i\epsilon)^{-1}\omega_0^k\Big](z)\,dz.
		\end{split}
	\end{equation}
	We note that $\psi_{k,\epsilon}^{+}(y,y_0), \psi_{k,\epsilon}^{-}(y,y_0)$ satisfy for $\iota\in\{+,-\}$
	\begin{equation}\label{F7}
		\big(-k^2 + \frac{d^2}{dy^2}\big)\psi_{k,\epsilon}^{\iota}(y,y_0)-\frac{b''(y)}{b(y)-b(y_0)+i\iota\epsilon}\psi_{k,\epsilon}^{\iota}(y,y_0)=\frac{-\omega_0^k(y)}{b(y)-b(y_0)+i\iota\epsilon}
	\end{equation}
	which can be reformulated as
	\begin{equation} \label{F8}
        \begin{split}
            &\psi^\iota_{k,\epsilon}(y, y_0) + \int_0^1G_k(y,z)\frac{b''(z)\psi^\iota_{k, \epsilon}(z, y_0)}{b(z) - b(y_0) + i\iota\eps}dz \\
            = &\int_0^1G_k(y,z)\frac{\omega_0^k(z)}{b(z) - b(y_0) + i\iota\eps}dz.
        \end{split}
	\end{equation}

We summarize the above computation as the following proposition.
\begin{proposition}
    For $k \ge 1$ and $(t, y) \in [0, \infty) \times [0, 1]$, assume $\omega_k(t,y)$ and $\psi_k(t,y)$ solve the equation \eqref{maink} with the initial data $\omega_0^k(y)$. For $y, y_0 \in [0, 1]$, $\eps \in (0, 1)$ and $\iota \in \{+, -\}$, let $\psi_{k,\eps}^\iota(y, y_0)$ be the solution to 
    \begin{equation}
        \begin{split}
            &\psi^\iota_{k,\epsilon}(y, y_0) + \int_0^1G_k(y,z)\frac{b''(z)\psi^\iota_{k, \epsilon}(z, y_0)}{b(z) - b(y_0) + i\iota\eps}dz \\
            = &\int_0^1G_k(y,z)\frac{\omega_0^k(z)}{b(z) - b(y_0) + i\iota\eps}dz.
        \end{split}
	\end{equation}
 Then we have for $(t, y) \in [0, \infty) \times [0, 1]$
 \begin{equation}
     \psi_k(t, y) = -\frac{1}{2\pi i}\lim_{\epsilon\to0+}\int_{0}^1e^{-ikb(y_0) t}|b'(y_0)|\left[\psi_{k,\epsilon}^{-}(y,y_0)-\psi_{k,\epsilon}^{+}(y,y_0)\right]dy_0.
 \end{equation}
\end{proposition}

	\subsection{Extension of the Green's function}
	For integers $k \ge 1$, recall that the Green's function $G_k(y,z)$ solves for $y, z \in [0, 1]$
	\begin{equation}\label{eq:Helmoltz}
		-\frac{\partial^2}{\partial y^2}G_k(y,z)+k^2G_k(y,z)=\delta(y-z),
	\end{equation}
	with Dirichlet boundary conditions $G_k(0,z)=G_k(1,z)=0$. $G_k$ has the explicit formula 
	\begin{equation}\label{eq:GreenFunction}
		G_k(y,z)=\frac{1}{k\sinh k}
		\begin{cases}
			\sinh(k(1-z))\sinh (ky)\qquad&\text{ if }y\leq z,\\
			\sinh (kz)\sinh(k(1-y))\qquad&\text{ if }y\geq z.
		\end{cases}
	\end{equation}

 We first make zero extension of $G_k(y,z)$ in the $z$ direction as follows. For $y \in [0, 1], z\in \R$, we set
	\begin{equation} \label{extension1}
		G_k(y,z) := 
		\begin{cases}
			G_k(y, z)\qquad&\text{ if }z\in [0,1], \\
			0\qquad&\text{ if }z\not\in [0, 1].
		\end{cases}
	\end{equation}
	After this extension, $G_k(y,z)$ solves 
	\begin{equation}\label{eq:extend1}
		(-\frac{d^2}{dz^2}+k^2)G_k(y,z) = \delta(z-y) -\frac{\sinh(k(1-y))}{\sinh k}\delta(z) - \frac{\sinh (ky)}{\sinh k}\delta(z-1)
	\end{equation}
	for $y\in [0, 1], z\in \R$. In order to obtain optimal control of Sobolev norms, we choose a smooth cutoff function $\Psi_k(y)$ satisfying the following assumptions:
	\begin{enumerate}
		\item For some small $\delta_0 > 0$, $\Psi_k \equiv 1 ~\text{on}~ [0, 1], \supp \Psi_k \Subset (-\delta_0/k, 1+\delta_0/k)$ and $ 0 \le \Psi_k \le 1$.
		\item $\Psi_k'(y) > 0$ for $y \in (-\delta_0/k, 0)$, and $\Psi_k'(y) < 0$ for $y \in (1, 1+\delta_0/k)$.
	\end{enumerate}
	These assumptions are helpful for establishing the limiting absorption principle. We can then extend $G_k(y,z)$ to $(y,z) \in \R\times\R$ by solving 
	\begin{equation}\label{eq:extend2}
		(-\frac{d^2}{dz^2}+k^2)G_k(y,z) = \Psi_k(y)\bigg[\delta(z-y) -\frac{\sinh(k(1-y))}{\sinh k}\delta(z) - \frac{\sinh (ky)}{\sinh k}\delta(z-1)\bigg].
	\end{equation}

	\begin{lemma}
		The extended Green function (also denoted as $G_k(y, z)$) has the explicit expression: for $y, z \in \R$,
		\begin{equation}\label{extendexpression}
			G_k(y,z) = \Psi_k(y)\bigg[\frac{1}{k}e^{-k|z-y|} - \frac{\sinh (k(1-y))}{k\sinh k}e^{-k|z|} - \frac{\sinh (ky)}{k\sinh k}e^{-k|z-1|}\bigg].
		\end{equation}
	\end{lemma}
	The proof of this lemma is straightforward. One can directly verify that expression (\ref{extendexpression}) agrees with (\ref{eq:GreenFunction}) for $(y,z) \in [0,1]\times[0,1]$.

	We extend the background flow $b(y)$ in a way such that $b(y)$ is smooth, $b'(y) > c $ for $y \in \R$ and some constant $c > 0$. In addition, $b''(y) = 0$ for $y \notin [-2, 2]$. We extend the initial data $\omega_0^k(y)$ to be a smooth function defined for all $y \in \mathbb{R}$ and supported on the interval $[-2, 2]$. We can then extend the spectral density function $\psi^\iota_{k, \eps}(y, y_0)$ to $y \in \R$ by solving for any $y_0 \in [0, 1]$
	\begin{equation} \label{extenddensity}
		\psi_{k,\epsilon}^\iota(y,y_0) + \int_\R G_k(y,z) \frac{b''(z)\psi_{k,\epsilon}^\iota(z,y_0)}{b(z)-b(y_0)+i\epsilon}dz = \int_\R G_k(y,z)\frac{\omega^k_0(z)}{b(z) -b(y_0) +i\epsilon}dz.
	\end{equation}
	Note that for $y \in [0, 1]$, the extended Green function $G_k(y, z)$, see (\ref{extendexpression}), does not vanish only when $z \in [0, 1]$. Therefore, the solution to (\ref{extenddensity}) agrees with the original spectral density function when $y\in[0, 1]$ and $y_0 \in [0, 1]$. 
	
	\subsection{Change of variables} For $y, y_0, z \in \mathbb{R}$, define the following new variables $v := b(y), v' := b(z), w := b(y_0)$. Define the following functions with new variables: for $y, y_0 \in \R$,
	\begin{equation} \label{change of variables}
 \begin{split}
     &f_0^k(v) := \omega_0^k(y), \quad\G_k(v,v') := G_k(y,z), \\
     &\phi_{k,\epsilon}^\iota(v,w) := \psi_{k,\epsilon}^\iota(y,y_0), \quad B(v) := b'(y).
 \end{split}
	\end{equation}
	Let $\Theta^\iota_{k,\epsilon}(v,w) = \phi_{k,\epsilon}^\iota(v+w,w)$. Then for $v, w \in \R$, $\Theta^\iota_{k,\epsilon}(v, w)$ solves the following equation: for $v, w \in \R$,
	\begin{equation}\label{mainequation}
		\begin{split}
			&\Theta^\iota_{k,\epsilon}(v,w) + \int_\R\G_k(v+w, v'+w)\frac{\partial_{v'}B(v'+w)\Theta^\iota_{k,\epsilon}(v',w)}{v'+i\iota\epsilon}dv' \\
			= &\int_\R\G_k(v+w, v'+w)\frac{1}{B(v'+w)}\frac{f^k_0(v'+w)}{v'+i\iota\epsilon}dv'.
		\end{split}
	\end{equation}
 \subsection{Weighted Sobolev norms} 
 For the weighted Sobolev norm $H^n_k(\R)$ defined in (\ref{weighted norm}), we have the following lemma.
 \begin{lemma}\label{lemma 22}
     For any positive integer $k$ we define the weighted Sobolev space $H_k^n(\mathbb{R})$ as in \eqref{weighted norm}. Let $r$ be a real number such that $0 \le r \le 1$. Let $H^r(\mathbb{R})$ be the usual Sobolev space. Then for any $u \in H^1_k(\mathbb{R})$, we have
     \begin{equation}
         \norm{u}_{H^r(\mathbb{R})} \lesssim_{r, n} |k|^{r-1}\norm{u}_{H^1_k(\mathbb{R})}.
     \end{equation}
     In addition, for any $2 \le p < \infty$, we have
     \begin{equation}
         \norm{u}_{L^p{\mathbb(\mathbb{R})}} \lesssim_{p} k^{-(\frac{1}{2} + \frac{1}{p})}\norm{u}_{H^1_k(\mathbb{R})}
     \end{equation}
 \end{lemma}
 \begin{proof}
     The proof follows directly from the Gagliardo-Nirenberg interpolation inequality and the definition of the $H^1_k(\mathbb{R})$ in (\ref{weighted norm}).
 \end{proof}

\section{The limiting absorption principle}
Define the following operator for $v, w \in \R, \eps \in [-\frac{1}{4}, \frac{1}{4}]\setminus \{0\}$ and $h \in H^{1}_k(\mathbb{R})$
\begin{equation}\label{Operator T}
	T_{k, \eps}h(v, w) := \int_\R \G_k(v+w, v'+w)\partial_{v'}B(v'+w)\frac{h(v')}{v'+i\eps}dv'.
\end{equation}
We have the following lemma.
\begin{lemma} \label{lemma 31}
Assume $k$ is a positive integer and $\eps \in [-\frac{1}{4}, \frac{1}{4}]\setminus \{0\}$. For any $h \in H^{1}_k(\mathbb{R})$ and $w \in \R$, we have
	\begin{equation} \label{ineq 33}
		\|T_{k, \eps}h(\cdot, w)\|_{H^1_k(\mathbb{R})} \lesssim \frac{k^{-\frac{1}{4}}}{1 + |w|}\|h\|_{H^1_k(\mathbb{R})}.
	\end{equation}
	In addition, we have for $w \in \R$,
	\begin{equation} \label{ineq 3.4}
		\norm{\partial_v T_{k, \eps}h(v, w)+ 2\Psi_k(v+w)\frac{\partial_{v}B(v+w)}{B(v+w)}h(v)\log(v+i\eps)}_{W^{1, 1}(v\in\mathbb{R})} \lesssim_k \|h\|_{H^1_k(\mathbb{R})}.
	\end{equation}
\end{lemma}
\begin{proof}
	For $w \in [b(-10), b(10)]$, using the identity that $$\partial_{v'}\log(v'+i\eps) = \frac{1}{v'+i\eps}$$
 and taking integration by parts, we get
	\begin{equation}
		\begin{split}
			T_{k, \eps}h(v, w) 
			=& -\int_\R \partial_{v'}\G_k(v+w, v'+w)\partial_{v'}B(v'+w)h(v')\log(v'+i\eps)dv' \\
			&- \int_\R \G_k(v+w, v'+w) \partial_{v'}\big[\partial_{v'}B(v'+w)h(v')\big]\log(v'+i\eps) dv' \\
			:=& T_1 + T_2.
		\end{split}
	\end{equation}
 
 {\bf Step 1: estimates of $T_1$}. Using the expression of $G_k(y, z)$ in (\ref{extendexpression}), we have
	\begin{equation}
		\begin{split}
			T_1 =& ~\Psi_k(v+w)\int_\R \bigg[e^{-k|b^{-1}(v+w) - b^{-1}(v'+w)|}\sign(v-v')\frac{\partial_{v'}B(v'+w)}{B(v'+w)}h(v')\\
   &\qquad \qquad \qquad \times\log(v'+i\eps)\bigg]dv' \\
			&- \Psi_k(v+w)\frac{\sinh k(1-b^{-1}(v+w))}{\sinh k}\int_\R\bigg[ e^{-k|b^{-1}(v'+w)|}\sign(v'+w - b(0))\\
			&\qquad\qquad \qquad \qquad \qquad \qquad \qquad \qquad \times\frac{\partial_{v'}B(v'+w)}{B(v'+w)}h(v')\log(v'+i\eps)\bigg]dv' \\
			&- \Psi_k(v+w)\frac{\sinh k(b^{-1}(v+w))}{\sinh k}\int_\R\bigg[ e^{-k|1-b^{-1}(v'+w)|}\sign(v'+w - b(1))\\
			&\qquad\qquad \qquad \qquad \qquad \qquad \qquad \times\frac{\partial_{v'}B(v'+w)}{B(v'+w)}h(v')\log(v'+i\eps)\bigg]dv'\\
			:=& T_{11} + T_{12} + T_{13}.
		\end{split}
	\end{equation}
	
	For $w \in [b(-10), b(10)]$, we have the following bounds
	\begin{equation} \label{eq 36}
		\begin{split}
			\|T_{11}(\cdot, w)\|_{L^\infty(\mathbb{R})} &\lesssim \|h\|_{L^2(\R)}\bigg(\int_{\R}e^{-2k|\ib(v+w) - \ib(v'+w)|} |\log (v'+i\eps)|^2dv'\bigg)^{\frac{1}{2}} \\
			&\lesssim \frac{1 + \log \la k\ra}{k^{\frac{1}{2}}}\|h\|_{L^2(\R)} \lesssim k^{-\frac{1}{4}}\|h\|_{L^2(\R)}.
		\end{split}
	\end{equation}
 Since $\Psi_k$ is compactly supported, (\ref{eq 36}) implies that
 \begin{equation}
     \|T_{11}(\cdot, w)\|_{L^2(\mathbb{R})} \lesssim k^{-\frac{1}{4}}\|h\|_{L^2(\R)}.
 \end{equation}
	 Note that
	\begin{equation} \label{calc 2.8}
		\begin{split}
			\partial_vT_{11}(v, w) = &- \frac{k\Psi_k(v+w)}{B(v+w)}\int_{\R} \bigg[e^{-k|b^{-1}(v+w) - b^{-1}(v'+w)|}\frac{\partial_{v'}B(v'+w)}{B(v'+w)}h(v')\\
   &\qquad \qquad \qquad \quad \times \log(v'+i\eps)\bigg]dv' \\
			& + 2\Psi_k(v+w)\frac{\partial_{v}B(v+w)}{B(v+w)}h(v)\log(v+i\eps)  \\
			& + \partial_v\Psi_k(v+w)\int_\R \bigg[ e^{-k|b^{-1}(v+w) - b^{-1}(v'+w)|}\sign(v-v')\\
			& \qquad \qquad \qquad \qquad \times \frac{\partial_{v'}B(v'+w)}{B(v'+w)}h(v')\log(v'+i\eps)\bigg]dv' \\
			:=& T_{111} + T_{112} + T_{113}.
		\end{split}
	\end{equation}
	By similar computation in (\ref{eq 36}),  $T_{111}$ and $T_{113}$ are bounded in $v$ for $w \in [b(-10), b(10)]$ with the following estimates
	\begin{equation}
		\begin{split}
			\|T_{111}(\cdot, w)\|_{L^2(\R)} +  \|T_{113}(\cdot, w)\|_{L^2(\R)}\lesssim k^{\frac{3}{4}} \|h\|_{L^2(\R)}\lesssim k^{-\frac{1}{4}} \|h\|_{H^1_k(\R)}. 
		\end{split}
	\end{equation}
	For $T_{112}$, using H\"{o}lder inequality and Lemma \ref{lemma 22} we have
 \begin{equation}
     \norm{T_{112}(\cdot, w)}_{L^2(\R)} \lesssim \norm{h}_{L^4(\R)} \lesssim k^{-\frac{3}{4}}\norm{h}_{H^1_k(\R)}.
 \end{equation}
 Hence by definition of the space $H^1_k(\R)$, we have
	\begin{equation}
		\norm{T_{11}(\cdot, w)}_{H^1_k(\R)} \lesssim |k|^{-\frac{1}{4}}\norm{h}_{H^1_k(\R)}.
	\end{equation}
 
	$T_{12}$ and $T_{13}$ can be treated in the same way. Note that by definition of the cutoff function $\Psi_k(v+w)$, $$\quad \Psi_k(v+w)\frac{\sinh k(1-\ib(v+w))}{\sinh k} \quad \text{and} \quad\Psi_k(v+w)\frac{\sinh k(\ib(v+w))}{\sinh k}$$ are bounded uniformly for $k\ge 1$ on the support of $\Psi_k(v+w)$. Hence the same estimate holds for $T_{12}$ and $T_{13}$. Therefore, we have for $w \in [b(-10), b(10)]$
	\begin{equation}\label{eq 310}
		\begin{split}
			\|T_{1}(\cdot, w)\|_{H^1_k(\R)} \lesssim k^{-\frac{1}{4}}\norm{h}_{H^1_k(\R)}.
		\end{split}
	\end{equation}

 {\bf Step2: estimates of $T_2$.} 
	Using the H\"{o}lder inequality, we have 
 \begin{equation}
 \begin{split}
          |T_2(v, w)| &\lesssim \norm{h}_{H^1(\R)}\bigg\{\int_\R \bigg(\G_k(v+w, v'+w)\log(v'+i\eps)\bigg)^2dv'\bigg\}^{\frac{1}{2}} \\
          &\lesssim k^{-\frac{5}{4}} \norm{h}_{H^1_k(\R)}.
 \end{split}
 \end{equation}
 Hence we have 
 \begin{equation}
     \norm{T_2(\cdot, w)}_{L^2(\R)} \lesssim k^{-\frac{5}{4}} \norm{h}_{H^1_k(\R)}.
 \end{equation}
	Taking one derivative in $v$ for $T_2$, we find that the derivative only acts on the Green function $\G_k$ and the cutoff function $\Psi_k$ which leads to a factor of $k$. We have
	\begin{equation}
		\begin{split}
			\|\partial_vT_2\|_{L^2(\R)} &\lesssim k^{-\frac{1}{4}} \norm{h}_{H^1_k(\R)}.
		\end{split}
	\end{equation}
	Therefore,
	\begin{equation} \label{eq 315}
		\norm{T_2}_{H^1_k(\R)} \lesssim k^{-\frac{1}{4}} \norm{h}_{H^1_k(\R)}.
	\end{equation}
	Combining (\ref{eq 315}) and (\ref{eq 310}), we get for $w \in [b(-10), b(10)]$, 
 \begin{equation}\label{eqq: 317}
     \norm{T_{k,\eps}h (\cdot, w)}_{H^1_k(\R)} \lesssim k^{-\frac{1}{4}} \norm{h}_{H^1_k(\R)}.
 \end{equation}

 {\bf Step 3: the case $w \notin [b(-10), b(10)]$}. Assume now $w \notin [b(-10), b(10)]$. Since $\partial_v'B(v'+w)$ is supported in $[b(-2), b(2)]$, we have $|v'| \gtrsim |w|$ on the support of $\partial_{v'}B(v'+w)$. Direct computation shows that 
 \begin{equation}\label{eqq: 318}
     \norm{T_{k,\eps}h (\cdot, w)}_{H^1_k(\R)} \lesssim \frac{1}{k|w|} \norm{h}_{H^1_k(\R)}.
 \end{equation}
Combining (\ref{eqq: 318}) and(\ref{eqq: 317}) we get for $w \in \R$,
\begin{equation}
     \norm{T_{k,\eps}h (\cdot, w)}_{H^1_k(\R)} \lesssim \frac{k^{-\frac{1}{4}}}{1 + |w|} \norm{h}_{H^1_k(\R)}.
 \end{equation}

 In view of (\ref{calc 2.8}), the most singular component of $\partial_v T_{k, \eps}h$ is $$2\Psi_k(v+w)\frac{\partial_{v}B(v+w)}{B(v+w)}h(v)\log(v+i\eps)$$. (\ref{ineq 3.4}) follows from a similar computation.
\end{proof}

\begin{proposition}[Limiting absorption principle]\label{lap}
	There exists $\sigma_0 > 0$ that is sufficiently small such that for $\eps \in [-\sigma_0, \sigma_0]\setminus\{0\}$, $k \ge 1$ and $f \in H^1_k(\R)$, we have for each $w \in \R$
	\begin{equation} \label{eq 314}
		\norm{f + \int_\R \G_k(v+w, v'+w)\partial_{v'}B(v'+w)\frac{f(v')}{v'+i\eps}dv'}_{H^1_k(\R)}  \ge \sigma_0 \norm{f}_{H^1_k(\R)}.
	\end{equation}
\end{proposition}

\begin{proof}
		Assume (\ref{eq 314}) does not hold, we can then find $\eps_j$, $w_j \in \R$, $k_j$ and $f_j$ such that $\eps_j \to 0+$, $\|f_j\|_{H^1_{k_j}(\R)} = 1$ and 
	\begin{equation}
		\lim_{j\to \infty}\norm{f_j + \int_\R \G_{k_j}(v+w_j, v'+w_j)\partial_{v'}B(v'+w_j)\frac{f_j(v')}{v'+i\eps_j}dv'}_{H^1_{k_j}(\R)} = 0.
	\end{equation}
	In view of Lemma \ref{lemma 31}, we can assume that $k_j$ and $w_j$ are bounded and thus we can replace $k_j$ simply by some fixed $k \in \mathbb{Z}^+$. Assume $w_j \to w_0$ for some $w_0 \in \R$. By the estimate (\ref{ineq 3.4}) and the fact that $\Psi_k$ is compactly supported, we can choose a subsequence of $\{f_j\}$ (also denoted by $\{f_j\})$ such that $f_j \to f$ in $H^1_k(\R)$ and $\|f\|_{H^1_k(\R)} = 1.$
	Therefore, we have
	\begin{equation}
		f(v) + \lim_{j\to \infty}\int_\R \G_k(v+w_0, v'+w_0)\partial_{v'}B(v'+w_0)\frac{f(v')}{v'+i\eps_j}dv' = 0.
	\end{equation}
 
	Let $y = b^{-1}(v), y_0 = b^{-1}(w_0)$ and set
	\begin{equation}
		h(y) = f(v).
	\end{equation}
	We have
	\begin{equation} \label{eq 318}
		h(y) + \lim_{j\to\infty}\int_\R G_k(y, z)b''(z)\frac{h(z)}{b(z) - b(y_0) + i\eps_j}dz = 0.
	\end{equation}
	
	Recall that $G_k(y, z)$ is given by
	\begin{equation} \label{eqq 328}
		\begin{split}
			G_k(y,z) &= \Psi_k(y)\bigg[\frac{1}{k}e^{-k|z-y|} - \frac{\sinh (k(1-y))}{k\sinh k}e^{-k|z|} - \frac{\sinh (ky)}{k\sinh k}e^{-k|z-1|}\bigg] \\
			&:= \Psi_k(y)\widetilde{G}_k(y, z),
		\end{split}
	\end{equation}
	and $\widetilde{G}_k(y, z)$ satisfies for $y, z \in \R$,
	\begin{equation}
		(k^2 - \frac{d^2}{dy^2})\widetilde{G}_k(y, z) = \delta(y-z).
	\end{equation}
	It follows from (\ref{eq 318}) that we can write $h(y)$ as 
	\begin{equation}
		h(y) = \Psi_k(y)g(y)
	\end{equation}
	   where $g(y)$ is not identically zero and has the same support as $\Psi_k(y)$. Furthermore, $g(y)$ satisfies for $y \in \R$
	\begin{equation}\label{eq for g}
		g(y) + \lim_{j\to \infty} \int_\R\widetilde{G}_k(y, z)b''(z)\frac{\Psi_k(z)g(z)}{b(z)-b(y_0)+i\eps_j}dz = 0.
	\end{equation}
We apply $k^2 - \frac{d^2}{dy^2}$ to (\ref{eq for g}) and get for $y\in \R$ 
	\begin{equation}
		k^2g(y) - \partial^2_y g(y) + \lim_{j\to \infty}\frac{b''(y)\Psi_k(y)g(y)}{b(y)-b(y_0) + i\eps_j} = 0.
	\end{equation}
	Therefore, for $y \in \R$, in the sense of distributions
	\begin{equation} \label{eq 320}
 \begin{split}
		&k^2g(y) - \partial^2_y g(y) + \lim_{j\to \infty}\frac{(b(y) - b(y_0))b''(y)\Psi_k(y)g(y)}{(b(y)-b(y_0))^2 + \eps_j^2} \\
  &+ iC(y_0)\delta(y-y_0)b''(y_0)\Psi_k(y_0)g(y_0) = 0
  \end{split}
	\end{equation}
	for some real number $C(y_0)\neq 0$. Multiplying (\ref{eq 320}) by $\overline{g(y)}$, integrating over $\R$ and taking the imaginary part, we get
	\begin{equation} \label{eq 321}
		b''(y_0)\Psi_k(y_0)g(y_0) = b''(y_0)h(y_0) = 0.
	\end{equation}
 In view of (\ref{eq 321}), \eqref{eq 320} is reduced to 
 \begin{equation}\label{eq: 332}
     k^2g(y) - \partial^2_y g(y) + \frac{b''(y)\Psi_k(y)g(y)}{b(y)-b(y_0)} = 0
 \end{equation}
We next show that $g(y) \equiv 0$ for $y \in \R$. This contradicts with the fact $\norm{f}_{H^1_k(\R)} = 1$. We consider two cases: $y_0 \in (0, 1)$ and $y_0 \notin (0, 1)$.

{\bf Case I: $y_0 \in (0, 1)$}. We first show that $g(y) \equiv 0$ for $y \in [0, 1]$. Set
\begin{equation}
    H(y) = \frac{b''(y)g(y)}{b(y)-b(y_0)}.
\end{equation}
By (\ref{eq 321}), $H(y) \in L^2(\R)$. Assume first $\norm{H}_{L^2(0, 1)}>0$. Note that for $y \in [0, 1]$, $G_k(y,z) = 0$ for $z \notin [0, 1]$ and $\Psi_k(y) = 1$. It follows from \eqref{eq: 332} that for $y \in [0, 1]$
\begin{equation}
    g(y) + \int_0^1 G_k(y, z)H(z) dz = 0,
\end{equation}
which leads to 
\begin{equation}
    (b(y) - b(y_0))H(y) + b''(y)\int_0^1 G_k(y, z)H(z) dz = 0.
\end{equation}
This contradicts with the assumption that the linearized operator $L_k$ does not have embedded eigenvalues. Hence $\norm{H}_{L^2(0, 1)} = 0$. Equation \eqref{eq: 332} and the boundary condition $g(0) = g(1) = 0$ give that $g(y) \equiv 0$ for $y \in [0, 1]$.

Next, we show that $g(y) \equiv 0$ for $y \ge 0$ and $y \le 1$. The fact $g(y) \equiv 0$ on $[0, 1]$ and \eqref{eq: 332} implies that $g \in H^2(\R)$ and $g'(0) = g'(1) = 0$. Multiplying (\ref{eq: 332}) by $g(y)$ and integrating for $y \ge 1$, we get
\begin{equation} \label{eq 326}
	\begin{split}
		0 = \int_1^\infty&\left[k^2|g(y)|^2+|\partial_yg(y)|^2 + \frac{(b'(y))^2\Psi_k(y)g^2(y)}{(b(y) - b(y_0))^2} \right.\\
		&-\left. \frac{b'(y)\Psi'_k(y)g^2(y)}{b(y) - b(y_0)} - \frac{2b'(y)\Psi_k(y)g_k(y)\partial_yg(y)}{b(y) - b(y_0)}\right]dy.
	\end{split}
\end{equation}
By assumption, $0\le \Psi_k(y) \le 1$, we have
\begin{equation}
	\begin{split}
		&\int_{1}^\infty\left[|\partial_yg(y)|^2 + \frac{(b'(y))^2\Psi_k(y)g^2(y)}{(b(y) - b(y_0))^2}- \frac{2b'(y)\Psi_k(y)g(y)\partial_yg(y)}{b(y) - b(y_0)}\right]dy \\
		\ge & \int_{1}^\infty\left[|\partial_yg(y)|^2 + \frac{(b'(y))^2\Psi^2_k(y)g(y)}{(b(y) - b(y_0))^2}- \frac{2b'(y)\Psi_k(y)g(y)\partial_yg(y)}{b(y) - b(y_0)}\right]dy \\
		= & \int_{1}^\infty\left[\partial_yg(y) - \frac{b'(y)\Psi_k(y)g(y)}{b(y)-b(y_0)}\right]^2dy \ge 0.
	\end{split}
\end{equation}
Since $b(y)$ is strictly increasing and we assumed that $\Psi'_k(y) \le 0$ for $y \ge 1$, we have 
\begin{equation} \label{eqq 335}
	-\frac{b'(y)\Psi'_k(y)g^2(y)}{b(y) - b(y_0)} \ge 0.
\end{equation}
Therefore, it follows from (\ref{eq 326}) that $g(y) \equiv 0$ for $y \ge 1$. $g(y) \equiv 0$ for $y \le 0$ follows from a similar argument.

{\bf Case II: $y_0 \notin (0, 1)$.} In this case $g(y) \equiv 0$ for $y \in [0, 1]$ still holds since the linearized operator $L_k$ is assumed to have no embedded eigenvalues. Note that $g(y)$ is supported on $(-\delta_0 /k, 1 + \delta_0/k)$. Multiplying (\ref{eq: 332}) by $g(y)$ and integraing over the interval $J: = (1, 1 + \delta_0 / k)$, we get
\begin{equation}\label{eq: 339}
    k^2\norm{g}^2_{L^2(J)} + \norm{\partial_y g}^2_{L^2(J)} + \int_1^{1+\delta_0/k} \frac{b''(y)\Psi_k(y)(g(y))^2}{b(y) - b(y_0)} dy = 0.
\end{equation}
Assume first $y_0 \in J$ and $g(y_0) \neq 0$. By assumption of $\Psi_k(y)$, $\Psi_k(y_0)\neq 0$. Hence from (\ref{eq 321}) we get $b''(y_0) = 0$. Therefore, 
\begin{equation}\label{eq: 340}
    \begin{split}
        \lvert \int_1^{1+\delta_0/k} \frac{b''(y)\Psi_k(y)(g(y))^2}{b(y) - b(y_0)} dy \rvert &\le \norm{g}^2_{L^\infty(J)} \lvert \int_1^{1+\delta_0/k} \frac{b''(y) - b''(y_0)}{b(y) - b(y_0)} dy \rvert \\
        &\le \frac{\delta_0}{k} \norm{b'' / b'}_{L^\infty(J)}\norm{g}^2_{L^\infty(J)}.
    \end{split}
\end{equation}
Since $g(1) = 0$, we have for $y \in J$
\begin{equation}\label{eq: 341}
    \begin{split}
        |g(y)| &= |g(y) - g(1)| \\
        &= \lvert \int_1^{ y} \partial_tg(t)dt \rvert \\
        &\le \big(\frac{\delta_0}{k}\big)^{1/2}\norm{\partial_yg}_{L^2(J)}.
    \end{split}
\end{equation}
It follows from \eqref{eq: 339}, \eqref{eq: 340} and \eqref{eq: 341} that
\begin{equation}\label{eq: 342}
 k^2\norm{g}^2_{L^2(J)} + \norm{\partial_y g}^2_{L^2(J)}  \le \big(\frac{\delta_0}{k}\big)^2\norm{b'' / b'}_{L^\infty(J)} \norm{\partial_y g}^2_{L^2(J)}.
\end{equation}
By assumption of the background flow $b(y)$, $b''(y) / b'(y)$ is bounded for all $y \in \R$. Therefore, there exists a $\delta_0 > 0$ sufficiently small such that 
\begin{equation}
    \big(\frac{\delta_0}{k}\big)^2\norm{b'' / b'}_{L^\infty(J)} < 1,
\end{equation}
which implies $g(y) \equiv 0$ on $J$.

Next, we assume that $y_0 \in J$ and $g(y_0) = 0$. We have 
\begin{equation}\label{eq: 344}
\begin{split}
        &\lvert \int_1^{1+\delta_0/k} \frac{b''(y)\Psi_k(y)(g(y))^2}{b(y) - b(y_0)} dy \rvert \\
        \le &\norm{b''}_{L^\infty(J)}\int_1^{1+\delta_0/k}\frac{\big[g(y) - g(y_0)\big]^2}{|b(y) - b(y_0)|}dy \\
        = &\norm{b''}_{L^\infty(J)} \int_1^{1+\delta_0/k}\frac{\big[\int_{y_0}^y \partial_sg(s)ds\big]^2}{|b(y) - b(y_0)|}dy \\
        \le &\norm{b''}_{L^\infty(J)} \norm{\partial_yg}^2_{L^2(J)}\int_1^{1+\delta_0/k} \frac{|y - y_0|}{|b(y) - b(y_0)|}dy \\
        \le &\frac{\delta_0}{k}\norm{b''}_{L^\infty(J)} \norm{1/b'}_{L^\infty(J)}\norm{\partial_yg}^2_{L^2(J)}.
\end{split}
\end{equation}
There exists a $\delta_0$ sufficiently small that 
\begin{equation}
    \frac{\delta_0}{k}\norm{b''}_{L^\infty(J)} \norm{1/b'}_{L^\infty(J)} < 1,
\end{equation}
which also implies that $g(y) \equiv 0$ on $J$.

Lastly, we assume that $y_0 \notin J$. In this case 
\begin{equation}
\begin{split}
    \lvert \int_1^{1+\delta_0/k} \frac{b''(y)\Psi_k(y)(g(y))^2}{b(y) - b(y_0)} dy \rvert &\le \norm{b''}_{L^\infty(J)}\int_1^{1+\delta_0/k}\frac{\big[g(y)\big]^2}{|b(y) - b(y_0)|}dy \\
    &\le \norm{b''}_{L^\infty(J)}\int_1^{1+\delta_0/k}\frac{\big[g(y)\big]^2}{|b(y) - b(1 + \delta_0 /k)|}dy.
\end{split}
\end{equation}
Using the fact that $g(1+\delta_0/k) = 0$ and similar arguments in \eqref{eq: 344}, we can show that $g(y) \equiv 0$ on $J$ if $\delta_0 > 0$ is sufficiently small. 

In conclusion, we have showed that $g(y) \equiv 0$ on $J = (1, 1 + \delta_0 / k)$ if $y_0 \notin [0, 1]$ and $\delta_0 > 0$ is sufficiently small. Using the same method we can show that $g(y)$ also vanishes on $(-\delta_0/k, 0)$ for some $\delta_0 > 0$. Therefore, $g(y)$ vanishes for all $y \in \R$ which finishes the proof of the limiting absorption principle.
\end{proof}

For later application, we also need the following result.
\begin{proposition} \label{solve Phi}
	Let $\sigma_0 > 0$ be the constant given by Proposition \ref{lap}. For $\eps \in (-\sigma_0, \sigma_0) \setminus \{0\}$, $k \ge 1$, let $T_{k,\eps}$ be the operator defined in \eqref{Operator T}. For $w \in \R$, there exist two functions $\Phi^{0}_{k,\eps}(v, w), \Phi^{1}_{k, \eps}(v, w) \in H^1_k(\R)$ such that
	\begin{equation} \label{eq 329}
		\Phi^{0}_{k, \eps}(v, w) + T_{k,\eps}\Phi^{0}_{k, \eps}(v, w) = \Psi_k(v+w)\frac{\sinh k(1 - b^{-1}(v+w))}{\sinh k}
	\end{equation}
	and
	\begin{equation} \label{eq 330}
		\Phi^{1}_{k,\eps}(v, w) + T_{k,\eps}\Phi^{1}_{k, \eps}(v, w) = \Psi_k(v+w)\frac{\sinh( kb^{-1}(v+w))}{\sinh k}.
	\end{equation}
\end{proposition}
\begin{proof}
	The above two equations are similar so we only need to focus on (\ref{eq 329}). By Lemma \ref{lemma 31}, 
	\begin{equation}
		- T_{k,\eps}\left(\Psi_k(v+w)\frac{\sinh k(1 - b^{-1}(v+w))}{\sinh k}\right) \in H^1_k(\R).
	\end{equation}
	It follows from the limiting absorption principle Proposition \ref{lap} that there exists a unique function $\widetilde{\Psi}_{k, \eps}(v, w) \in H^1_k(\R)$ such that \begin{equation}
		\widetilde{\Psi}_{k, \eps} + T_{k, \eps}\left(\widetilde{\Psi}_{k, \eps}\right) = - T_{k,\eps}\left(\Psi_k(v+w)\frac{\sinh k(1 - b^{-1}(v+w))}{\sinh k}\right).
	\end{equation}
	Set 
	\begin{equation}
		\Phi^{0}_{k, \eps}(v, w) = \widetilde{\Psi}_{k, \eps}(v, w) + \Psi_k(v+w)\frac{\sinh k(1 - b^{-1}(v+w))}{\sinh k}.
	\end{equation}
	One can verify that $\Phi^{0}_{k, \eps}(v, w)$ solves (\ref{eq 329}). 
\end{proof}
 
\section{Singularity structures}
The analysis of the singularity structure of the spectrum density function $\Theta^\iota_{k, \eps}(v, w)$ is the most crucial part to understand the long time behavior of the stream function. Recall that for $v\in \R, w \in [b(0), b(1)]$, $\Theta^\iota_{k, \eps}(v, w)$ solves the following equation
\begin{equation}\label{mainequation1}
		\begin{split}
			&\Theta^\iota_{k,\epsilon}(v,w) + \int_\R\G_k(v+w, v'+w)\frac{\partial_{v'}B(v'+w)\Theta^\iota_{k,\epsilon}(v',w)}{v'+i\iota\epsilon}dv' \\
			= &\int_\R\G_k(v+w, v'+w)\frac{1}{B(v'+w)}\frac{f^k_0(v'+w)}{v'+i\iota\epsilon}dv'.
		\end{split}
	\end{equation}

	\subsection{Singularity structure of right-hand side of (\ref{mainequation1})}
	To illustrate the main idea, we start with the analysis of singularities of 
	\begin{equation}
		A_k(v, w) := \int_\R\G_k(v+w, v'+w)\frac{h(v'+w)}{v'+i\epsilon}dv',
	\end{equation} 
	where $h(\cdot)$ is compactly supported smooth function on the real line.  For $v, v' \in \R$, let
	\begin{equation}\label{free part}
		\G^{fr}_k(v, v') := \Psi_k(\ib(v))\frac{1}{k}e^{-k|\ib(v)-\ib(v')|}
	\end{equation}
	be the \textit{free part} of the Green function, and
	\begin{equation}\label{boundary part}
		\G^b_k(v, v') := \G_k(v,v') - \G^{fr}_k(v,v')
	\end{equation}
	be the \textit{boundary part} of the Green function. The boundary part $\G^b_k$ can be further splitted into two parts $\G_k^b(v, v') = \G_k^{b_0}(v, v') + \G_k^{b_1}(v, v')$ where
	\begin{equation} \label{eq 222}
        \begin{split}
            \G_k^{b_0}(v, v') &:= -\Psi_k(\ib(v))\frac{\sinh(k(1-\ib(v)))}{k\sinh k}e^{-k|\ib(v')|} \\
            &:= \Phi^{b0}_k(v) \frac{e^{-k|\ib (v')|}}{k}
        \end{split}
	\end{equation}
	captures the boundary effect from the side $b(0)$, and 
	\begin{equation}
 \begin{split}
     		\G_k^{b_1}(v, v') &:=  -\Psi_k(\ib(v))\frac{\sinh(k\ib(v))}{k\sinh k}e^{-k|\ib(v')-1|} \\
       &:= \Phi^{b1}_k(v) \frac{e^{-k|\ib (v')-1|}}{k}
 \end{split}
	\end{equation}
	captures the boundary effect from the side $b(1)$. Then we can decompose for $v, w \in \R$
	\begin{equation} \label{general integral}
		\begin{split}
			&A_k(v, w) \\ = &~\int_\R\G_k^{fr}(v+w, v'+w)\frac{h(v'+w)}{v'+i\epsilon}dv'+ \int_\R\G_k^b(v+w, v'+w)\frac{h(v'+w)}{v'+i\epsilon}dv'\\
			:= &~A_k^{fr}(v, w) + A_k^b(v, w) .
		\end{split}
	\end{equation}
	We study the free term $A^{fr}_k(v, w)$ first, which can be written as
	\begin{equation} \label{toyintegral}
		A^{fr}_k(v, w) = \Psi_k(v+w)\int_\R \frac{1}{k}e^{-k|\ib(v+w)-\ib(v'+w)|}\frac{h(v'+w)}{v'+i\epsilon}dv'.
	\end{equation}
	Taking one derivative in $v$, we get for $v, w \in \R$
	\begin{equation}\label{computation1} 
		\begin{split}
			\partial_vA^{fr}_k(v, w) &= -\Psi_k(v+w)\int_\R e^{-k|\ib(v+w)-\ib(v'+w)|}\frac{\sign(v-v')}{B(v+w)}\frac{h(v'+w)}{v'+i\eps}dv'  \\
   &\quad +\partial_v \Psi_k(v+w)\int_\R \frac{1}{k}e^{-k|\ib(v+w)-\ib(v'+w)|}\frac{h(v'+w)}{v'+i\epsilon}dv' \\
   &:= T_1 + T_2.
		\end{split}
	\end{equation}
 Taking integration by parts, we get
 \begin{equation}\label{comp 411}
     \begin{split}
         T_1 = &-\Psi_k(v+w)\int_\R \bigg[e^{-k|\ib(v+w)-\ib(v'+w)|}\frac{\sign(v-v')}{B(v+w)}h(v'+w) \\ &\qquad \qquad \qquad \times\partial_{v'}\log(v'+i\eps)\bigg]dv'\\
         = &- \Psi_k(v+w)\int_\R \bigg[ke^{-k|\ib(v+w)-\ib(v'+w)|}\frac{h(v'+w)}{B(v+w)B(v'+w)}\\ &\qquad \qquad \qquad \times\log(v'+i\eps)\bigg]dv' \\
			&+\Psi_k(v+w)\int_\R \bigg[e^{-k|\ib(v+w)-\ib(v'+w)|}\sign(v'-v)\frac{\partial_{v'}h(v'+w)}{B(v+w)}\\  &\qquad \qquad \qquad\times\log(v'+i\eps)\bigg]dv' \\
   &+\frac{2\Psi_k(v+w)}{B(v+w)}h(v+w)\log(v+i\eps)
     \end{split}
\end{equation}
	Here in the last equation we used the fact that for $v, v' \in \R$ $$\frac{d}{dv'}\sign(v'-v) = 2\delta(v'-v).$$ 
 Therefore, we can rewrite for $v, w \in \R$
 \begin{equation}\label{comp: 411}
     \partial_v \A_k^{fr}(v, w) = \frac{2\Psi_k(v+w)}{B(v+w)}h(v+w)\log(v+i\eps) + \mathcal{R}_{k, \eps}^1(v, w),
 \end{equation}
 where the remainder term $\mathcal{R}_{k, \eps}^1(v, w)$ is defined as
 \begin{equation}\label{comp 413}
 \begin{split}
          \mathcal{R}_{k, \eps}^1(v, w)
          :=& - \Psi_k(v+w)\int_\R \bigg[ke^{-k|\ib(v+w)-\ib(v'+w)|}\frac{h(v'+w)}{B(v+w)B(v'+w)}\\ &\qquad \qquad \qquad \times\log(v'+i\eps)\bigg]dv' \\
			& +\Psi_k(v+w)\int_\R \bigg[e^{-k|\ib(v+w)-\ib(v'+w)|}\sign(v'-v)\\&\qquad \qquad \qquad \times\frac{\partial_{v'}h(v'+w)}{B(v+w)}\log(v'+i\eps)\bigg]dv' \\
   &- \partial_v \Psi_k(v+w)\int_\R \bigg[\partial_{v'}\big[\frac{1}{k}e^{-k|\ib(v+w)-\ib(v'+w)|}h(v'+w)\big]\\ &\qquad \qquad \qquad \times\log(v'+i\eps)\bigg]dv'.
 \end{split}
 \end{equation}
 
We observe from the above calculation that when we differentiate $A^{fr}_k(v, w)$ in the variable $v$, we get a term with $\log(v+i\eps)$. This term introduces singularity as $\eps \to 0+$. For the remainder term, we also have similar singularity structure. We formulate this rigorously as the following lemma.
\begin{lemma}\label{lemma 41 fr}
    	Let $N \ge 1$, $k \ge 1$, $\eps \in (-1, 1)\setminus\{0\}$. For any $h \in H^N_k(\R)$, define for $v, w \in \R$
		\begin{equation}
			\mathcal{R}^0_{k,\eps}(v,w) = \int_\R \G^{fr}(v+w, v'+w) \frac{h(v'+w)}{v'+i\eps}dv'.
		\end{equation}
		Then for $0 \le j \le N - 1$ there exist functions $A_k^{j}(v)$ and $\mathcal{R}^{j}_{k,\eps}(v,w)$ satisfying the following bounds (setting $A_k^0(v, w) = 0$)
  \begin{equation}\label{est 415}
				\norm{\partial_v^mA^j_{k}(v)}_{H^1_{k}} + \norm{\partial_w^m\mathcal{R}^j_{k,\eps}(v,w)}_{L^2_w H^1_{k,v}}\lesssim_{j,m} \norm{h}_{H^{j+1+m}_k}
			\end{equation}
  such that for $0 \le j \le N - 2$,
			\begin{equation}\label{comp 415}
				\partial_v \mathcal{R}^j_{k,\eps}(v,w) = A_k^{j+1}(v+w)\log(v+i\eps) + \mathcal{R}^{j+1}_{k,\eps}(v,w)
			\end{equation}
			
\end{lemma}	
\begin{proof}
We claim that for $j \ge 1$,
there exist smooth functions $\alpha^j_p(v, v', k)$ and $\beta^j_p(v, v', k)$ ($p \in \{0, 1\}$) such that 
for any integers $m, n \ge 0$
\begin{equation}\label{comp 418}
    \norm{\partial_v^m\partial_{v'}^n \alpha_p^j(v,v', k)}_{L^\infty(\R \times \R)} + \norm{\partial_v^m\partial_{v'}^n \beta_p^j(v, v', k)}_{L^\infty(\R \times \R)} \lesssim k^{j + m  - p}, 
\end{equation}
and we can define $\mathcal{R}^j_{k,\eps}(v, w)$ as
\begin{equation}\label{comp 417}
\begin{split}
    \mathcal{R}^j_{k,\eps}(v,w) &= \int_\R \bigg\{e^{-k|\ib(v+w) - \ib(v'+w)|}\log(v'+i\eps) \sign(v - v') \\
    &\qquad \times \sum_{p=0}^1\big[\alpha^j_p(v+w, v'+w, k)\partial^p_{v'}h(v'+w)\big]\bigg\} dv' \\
   &\quad + \int_\R \bigg\{e^{-k|\ib(v+w) - \ib(v'+w)|}\log(v'+i\eps) \\
    &\qquad \times \sum_{p=0}^1\big[\beta^j_p(v+w, v'+w, k)\partial^p_{v'}h(v'+w)\big]\bigg\} dv'.
\end{split}
\end{equation}
We see from (\ref{computation1}) to (\ref{comp 413}) that we can define $\mathcal{R}^1_{k,\eps}(v,w)$ by setting
\begin{align}
    \alpha^1_0(v, v', k) &= -\frac{\partial_v\Psi_k(v)}{B(v')}, \quad \alpha^1_1(v, v', k) = \frac{\Psi_k(v)}{B(v)} \\
    \beta^{1}_0(v, v', k) &= -\frac{k\Psi_k(v)}{B(v)B(v')}, \quad \beta^{1}_1(v, v', k) = -\frac{\partial_v\Psi_k(v)}{k},
\end{align}
and (\ref{comp: 411}) holds with 
\begin{equation}\label{comp 420}
    A^1_k(v) = \frac{2\Psi_k(v)h(v)}{B(v)}.
\end{equation}
Noting that $\exp{-k|\ib(v+w) - \ib(v'+w)|}$ is smooth in $w$, the estimate \ref{est 415} follows from (\ref{comp 417}), (\ref{comp 420}) and the Cauchy-Schwarz inequality.

Assume for some $j_0 \ge 1$ we have constructed $\mathcal{R}^{j_0}_{k,\eps}(v,w)$. We have 
\begin{equation}\label{comp 419}
\begin{split}
    &\partial_v \mathcal{R}^{j_0}_{k,\eps}(v,w) \\
    = &~2\log(v+i\eps)\sum_{p=0}^1\big[\alpha^j_p(v+w, v+w, k) \partial^p_{v}h(v+w)\big] \\ 
    &+ \int_\R e^{-k|\ib(v+w) - \ib(v'+w)|}\log(v'+i\eps) \sign(v - v')S(v, v', w)dv'\\
   &+ \int_\R e^{-k|\ib(v+w) - \ib(v'+w)|}\log(v'+i\eps)T(v, v', w) dv' \\
   := &~A^{j_0 + 1}_k(v+w)\log(v+i\eps) + R^{j_0 + 1}_{k,\eps}(v, w),
\end{split}
\end{equation}
where
\begin{equation}
\begin{split}
        S(v, v', w) = &\sum_{p=0}^1\big[\partial_v\alpha^j_p(v+w, v'+w, k)\partial^p_{v'}h(v'+w)\big] \\
   &- \frac{k}{B(v+w)}\sum_{p=0}^1\big[\beta^j_p(v+w, v'+w, k)\partial^p_{v'}h(v'+w)\big]
\end{split}
\end{equation}
and
\begin{equation}
\begin{split}
      T(v, v', w) = &\sum_{p=0}^1\big[\partial_v\beta^j_p(v+w, v'+w, k)\partial^p_{v'}h(v'+w)\big] \\
   &- \frac{k}{B(v+w)}\sum_{p=0}^1\big[\alpha^j_p(v+w, v'+w, k) \partial^p_{v'}h(v'+w)\big].
\end{split}
\end{equation}
Hence we can define $\mathcal{R}^{j_0 + 1}_{k,\eps}(v,w)$ by setting for $p \in \{0, 1\}$
\begin{align}
    \alpha^{j_0 + 1}_p(v, v', k) &= \partial_v\alpha^{j_0}_p(v+w, v'+w, k) - \frac{k}{B(v+w)}\beta^{j_0}_p(v+w,v'+w, k), \label{comp 424}\\ 
    \beta^{j_0 + 1}_p(v, v', k) &= \partial_v\beta^{j_0}_p(v+w, v'+w, k) - \frac{k}{B(v+w)}\alpha^{j_0}_p(v+w,v'+w, k),\label{comp 425}
\end{align}
We define
\begin{equation}
    A^{j_0 + 1}_k(v) = 2\sum_{p=0}^1\big[\alpha^{j_0}_p(v,v, k)\partial^p_{v}h(v)\big].
\end{equation}
Hence (\ref{comp 415}) holds for $j = j_0$.
For $0 \le m \le N - j_0 - 2$, it follows that
\begin{equation}
    \norm{\partial_v^mA^{j_0 + 1}_k(v)}_{L^2_w H^1_{k,v}} \lesssim \sum_{q=0}^m\sum_{p = 0}^1 \binom{m}{q} \norm{\partial_v^q\alpha_p^{j_0}(v,v,k)\partial_v^ph(v)}_{H^1_k}.
\end{equation}
Since (\ref{comp 418}) holds for the case $j = j_0$, we have
\begin{equation}
    \norm{\partial_v^mA^{j_0 + 1}_k(v)}_{L^2_w H^1_{k,v}} \lesssim_m \sum_{q=0}^m\sum_{p = 0}^1 k^{j_0+q-p}\norm{h}_{H^{1+p}_k} \lesssim_m \norm{h}_{H^{j_0 + 1+m}_k}.
\end{equation}
The estimate of $\mathcal{R}^{j_0 + 1}_{k,\eps}(v, w)$ follows from (\ref{comp 418}) ,(\ref{comp 417}), (\ref{comp 424}), (\ref{comp 425}) and the Cauchy-Schwarz inequality. 
\end{proof}

The free component $\G^{fr}_k$ of the Green function generates singularities in $v$ when $v$ is close to 0. As a contrast, the boundary components $\G^b_k$ generate singularities in $w$ when $w$ is close to the boundary $b(0)$ or $b(1)$. We have the following lemma establishing the boundary singularity structure near $b(0)$. The proof is similar to that of Lemma \ref{lemma 41 fr}.

\begin{lemma}\label{lemma 42 b0}
        Let $N \ge 1$, $k \ge 1$, $0 < \eps < 1/2$. For any $h \in H^N_k(\R)$, define
    \begin{equation}
        \mathcal{S}^0_{k,\eps}(v,w) = \int_\R \G^{b0}(v+w, v'+w) \frac{h(v'+w)}{v'+i\eps}dv'.
    \end{equation}
    Then for $1 \le j \le N - 1$ there exist functions $B_k^{j}(v)$ and $\mathcal{S}^{j}_{k,\eps}(v,w)$ satisfying the following bounds (setting $B_k^0(v, w):= 0$)
    \begin{equation}
				\norm{\partial_v^mB^j_{k,\eps}(v)}_{H^1_{k}} + \norm{\partial_v^m\mathcal{S}^j_{k,\eps}(v,w)}_{L^2_w H^1_{k,v}}\lesssim_{m,j} \norm{h}_{H^{j+1+m}_k}
			\end{equation}
    such that for $0 \le j \le N - 2$,
			\begin{equation}
				\partial_w \mathcal{S}^j_{k,\eps}(v,w) = B_k^{j+1}(v+w)\log(b(0) - w +i\eps) + \mathcal{S}^{j+1}_{k,\eps}(v,w).
			\end{equation}
			
\end{lemma}
\begin{proof}
    The proof applies similar arguments to the proof of Lemma \ref{lemma 41 fr}. We thus omit the details.
\end{proof}

For the component $\G^{b1}$, it does not generate singularities if the variable $w$ is supported away from the boundary $b(1)$. We have the following lemma.
\begin{lemma}\label{lemma 43 b1}
        Let $N \ge 1$, $k \ge 1$, $0 < \eps < 1/2$. Assume for some $\delta_0 > 0$, $\varphi_{\delta_0}(w)$ is a smooth function such that $\varphi_{\delta_0}(w) = 0$ for $w \in (b(1 - \delta_0), b(1+\delta_0))$. For any $h \in H^N_k(\R)$, define
    \begin{equation}
        \mathcal{T}_{k,\eps}(v,w) = \int_\R \G^{b1}(v+w, v'+w) \frac{h(v'+w)\varphi_{\delta_0}(w)}{v'+i\eps}dv'.
    \end{equation}
    Then for integers $m, n\ge 0$ and $m + n \le N-1$, we have
    \begin{equation}\label{est 433}
        \norm{\partial_v^m\partial_w^n \mathcal{T}_{k,\eps}(v,w) }_{L^2_w H^1_{k,v}} \lesssim_{m,n}\norm{h}_{H^{1 + m + n}_k}.
    \end{equation}
\end{lemma}
\begin{proof}
    Define $$\mathcal{T}^0_{k,\eps}(v,w) = \int_\R \G^{b1}(v+w, v'+w) \frac{h(v'+w)}{v'+i\eps}dv'.$$
    By symmetry, a similar conclusion to Lemma \ref{lemma 42 b0} holds for $\mathcal{T}^0_{k,\eps}(v,w)$. More specifically, for $1 \le j \le N - 1$ there exist functions $B_k^{j}(v)$ and $\mathcal{T}^{j}_{k,\eps}(v,w)$, such that for $0 \le j \le N - 2$,
			\begin{equation}
				\partial_w \mathcal{T}^j_{k,\eps}(v,w) = B_k^{j+1}(v+w)\log(b(1) - w +i\eps) + \mathcal{T}^{j+1}_{k,\eps}(v,w).
			\end{equation}
		In addition, we have for $0 \le j \le N-1$ (setting $B_k^0(v, w):= 0$) and any $0 \le m \le N - j -1$,
			\begin{equation}
				\norm{\partial_v^mB^j_{k,\eps}(v)}_{H^1_{k}} + \norm{\partial_v^m\mathcal{T}^j_{k,\eps}(v,w)}_{L^2_w H^1_{k,v}}\lesssim_{m,j} \norm{h}_{H^{j+1+m}_k}.
			\end{equation}
   However, $\log(b(1) - w + i\eps)$ is smooth with bounded derivatives when $w$ is restricted on the support of $\varphi_{\delta_0}(w)$. Therefore, $\mathcal{T}_{k,\eps}(v,w) = \mathcal{T}^0_{k,\eps}(v,w)\varphi_{\delta_0}(w)$ has the estimate (\ref{est 433})
\end{proof}

\subsection{Singularity structures of the spectrum density function}\label{secSSSDF}
In order to study the singularity of $\Theta_{k, \eps}^\iota(v, w)$ more clearly, we split the domain of $w$ into three parts. Fix $0 < \delta_0 < 1/10$. There exist three non-negative smooth functions $\Upsilon_1(w)$, $\Upsilon_2(w)$ and $\Upsilon_3(w)$ such that
\begin{equation} \label{cutoff}
	\begin{split}
		&\Upsilon_1(w) = 1 \quad \text{on} \quad [b(-\delta_0), b(\delta_0)], \quad \quad \Upsilon_2(w) = 1 \quad \text{on} \quad [b(2\delta_0), b(1-2\delta_0)], \\ &\Upsilon_3(w) = 1 \quad \text{on} \quad [b(1-\delta_0), b(1+\delta_0)]
	\end{split}
\end{equation}
and
\begin{equation}
	\Upsilon_1(w) + \Upsilon_2(w) + \Upsilon_3(w) = 1 \quad \text{on}\quad \R.
\end{equation}
For $v, w \in \R$ and $j \in \{1, 2, 3\}$, set
\begin{equation} \label{eq 43}
	\Theta_{k, \eps}^{j, \iota}(v, w) := \Theta_{k, \eps}^\iota(v, w)\Upsilon_j(w).
\end{equation}

In order to save notation, we define a function $\tau(m)$ for an integer $m \ge 0$ as
\begin{equation}
\tau(m) = 
    \begin{cases}
        &1, \quad \text{if}\quad m = 0 \\
        &2, \quad \text{if}\quad m \ge 1.
    \end{cases}
\end{equation}

Before presenting the results on regularity structures of $\Theta^\iota_{k,\eps}(v,w)$, we first introduce the following definition of singularity structure when $w$ is close to the boundary $b(0)$. 

\begin{definition}\label{def b0}
    Assume $k \ge 1$, $m \ge 0$, $n \ge 0$, $N \ge 1$ and $1 \le \alpha \le N$ are five integers. For any $\eps \in (-1, 1)\setminus\{0\}$, $g \in H^N_k(\R)$, a function $\Gamma(v, w)$, $v, w \in \R$, is said to have the $\mathcal{F}^{b0}_{m, n, k, \eps}(g, \alpha, N)$ type singularity structure if there exists a constant $C$ which depends only on $\alpha$ and $N$ such that
    \begin{enumerate}
        \item $\Gamma(v, w)$ has the following estimate
        \begin{equation}
            \norm{\Gamma(v, w)}_{L^2_w H^1_{k,v}(\R^2)} \le C \norm{g}_{H^\alpha_k(\R)}.
        \end{equation}
        \item If $\alpha < N$, then for $1 \le p \le \tau(m)$ and $1 \le q \le \tau(n)$ such that there exist functions $A_p(v, w)$, $B_q(v, w)$, $R(v, w)$ and $S(v, w)$ such that
        \begin{equation}
        \begin{split}
            &\partial_v \Gamma(v, w) = \sum_{p = 1}^{\tau(m)}A_p(v, w)\log^p(v+i\eps) + R(v, w), \\
            &\partial_w \Gamma(v, w) = \sum_{q = 1}^{\tau(n)}B_q(v, w)\log^q(b(0) - w + i\eps) + S(v, w).
        \end{split}
        \end{equation}
        Furthermore, for $1 \le p \le \tau(m)$, $A_p(v, w)$ and $R(v, w)$ have the $\mathcal{F}^{b0}_{m+1, n, k,\eps}(g, \alpha + 1, N)$ type singularity structure.
        For $1 \le q \le \tau(n)$, $B_q(v, w)$ and $S(v, w)$ have the $\mathcal{F}^{b0}_{m, n+1, k,\eps}(g, \alpha + 1, N)$ type singularity structure.
    \end{enumerate}
\end{definition}

We also define the interior singularity structure when $w$ is away from the boundary points $b(0)$ and $b(1)$.
\begin{definition}
    Assume $k \ge 1$, $m \ge 0$, $N \ge 1$ and $1 \le \alpha \le N$ are four integers. For any $\eps \in (-1,1)\setminus\{0\}$, $g \in H^N_k(\R)$, a function $\Gamma(v, w)$, $v, w \in \R$, is said to have the $\mathcal{F}^{in}_{m, k,\eps}(g, \alpha, N)$ type singularity structure if there exists a constant $C$ which depends only on $\alpha$ and $N$ such that
    \begin{enumerate}
        \item $\Gamma(v, w)$ has the following estimate
        \begin{equation}
            \norm{\Gamma(v, w)}_{L^2_w H^1_{k,v}(\R^2)} \le C \norm{g}_{H^\alpha_k(\R)}.
        \end{equation}
        \item If $\alpha < N$, for $1 \le p \le \tau(m)$ there exist functions $A_p(v, w)$ and $R(v, w)$ such that
        \begin{equation}
        \begin{split}
            &\partial_v \Gamma(v, w) = \sum_{p = 1}^{\tau(m)}A_p(v, w)\log^p(v+i\eps) + R(v, w).
        \end{split}
        \end{equation}
        In addition, $A_p(v, w)$ and $R(v, w)$ have the $\mathcal{F}^{in}_{m+1, k,\eps}(g, \alpha + 1, N)$ type singularity structure.
        
         \item For any integer $0 \le q \le N - \alpha$, $\partial_w^q \Gamma(v, w)$ has the $\mathcal{F}^{in}_{m, k,\eps}(g, \alpha + q, N)$ type singularity structure.
    \end{enumerate}
\end{definition}

Using Definition \ref{def b0}, we see that the right-hand side of equation (\ref{mainequation1}) has the $\mathcal{F}^{b0}_{0, 0, k,\eps }(g, 1, N)$ type singularity structure assuming that $f_0^k \in H^N_k(\R)$. We have the following proposition.
\begin{proposition}\label{structure of RHS}
    Assume $g(v) \in H^N_k(\R)$ for some integer $N \ge 1$. Then 
    \begin{equation}
       \Gamma(v, w) := \int_\R\G_k(v+w, v'+w)\frac{g(v'+w)\Upsilon_1(w)}{v'+i\epsilon}dv'
   \end{equation}    
   has the $\mathcal{F}^{b0}_{0, 0, k,\eps }(g, 1, N)$ type singularity structure.
\end{proposition}
\begin{proof}
The proof follows from Lemma \ref{lemma 41 fr}, Lemma \ref{lemma 42 b0} and Lemma \ref{lemma 43 b1}.
\end{proof}

We also need the following technical lemmas regarding the singularity structures defined in Definition \ref{def b0}.
\begin{lemma}\label{aux lemma 2}
    Assume $g \in H^{N+j}_k(\R)$ for some $N \ge 1$ and $0 \le j \le N - 1$. For any $\Gamma(v, w) \in \mathcal{F}^{b0}_{0, 0, k,\eps}(g, N - j, N)$, we have $\Gamma(v, w) \in \mathcal{F}^{b0}_{0, 0, k,\eps}(g, N, N + j)$.
\end{lemma}
\begin{proof}
    The proof follows from the Definition \ref{def b0} and the fact that for $m \ge n$
    \begin{equation}
        H^{m}_k(\R) \subset  H^{n}_k(\R).
    \end{equation}
\end{proof}

Recall that for $k \ge 1$, $v \in \R$, the smooth function $\Phi^{b0}_k$ is defined as
\begin{equation}\label{phi b0}
    \Phi^{b0}_k(v) = \Psi_k(v) \frac{\sinh k(1-\ib(v))}{\sinh k}
\end{equation}
\begin{lemma}\label{aux lemma}
    Let  $\Phi^{b0}_k$ be defined as (\ref{phi b0}). For any integer $N \ge 1$ and any function $g \in H^M_k(\R)$ with $M \ge N$, assume $\alpha(v, w) \in \mathcal{F}^{b0}_{0, 0, k,\eps}(\Phi^{b0}_k, m, N)$ and $\beta(v, w) \in \mathcal{F}^{b0}_{0, 0, k,\eps}(g, n, N)$. If there exist a integer $j \ge 1$ that $N - m \ge j$, $N - n \ge j$ and $M \ge m+n+j$, then 
    \begin{equation}
        \alpha(v, w)\beta(b(0) - w, w) \in \mathcal{F}^{b0}_{0, 0, k,\eps}(g, m + n, m+n+j).
    \end{equation}
\end{lemma}
\begin{proof}
    We prove the lemma with an induction argument for $j\ge 0$. For $j = 0$, the lemma is reduced to prove that
    \begin{equation}
        \norm{\alpha(v, w)\beta(b(0) - w, w)}_{L^2_w H^{1}_{k,v}} \lesssim \norm{g}_{H^{m+n}_k}.
    \end{equation}
    As a matter of fact, 
    \begin{equation}
        \begin{split}
            \norm{\alpha(v, w)\beta(b(0) - w, w)}_{L^2_w H^{1}_{k,v}} & \lesssim \norm{\alpha(v, w)}_{L^\infty_w H^{1}_{k,v}}\norm{\beta(b(0) - w, w)}_{L^2_w} \\
            &\lesssim \norm{\alpha(v, w)}_{L^\infty_w H^{1}_{k,v}} \norm{\beta(v, w)}_{L^2_wL^\infty_v} \\
            &\lesssim \frac{1}{2}\norm{\alpha(v, w)}_{H^1_{k,w} H^{1}_{k,v}} \cdot \frac{1}{2} \norm{\beta(v, w)}_{L^2_w H^1_{k,v}} \\
            &\lesssim k^m \norm{g}_{H^n_k} \\
            &\lesssim \norm{g}_{H^{m+n}_k}.
        \end{split}
    \end{equation}
    Assume now the lemma holds for some integer $j_0 \ge 0$. We show that the case $j = j_0 + 1$ also holds. Taking derivative in $v$, there exist functions $\alpha_1, \alpha_2, \alpha_3$ such that
    \begin{equation}
        \partial_v \alpha(v, w) = \alpha_1(v, w)\log(v+i\eps) + \alpha_2(v, w)\log^2(v+i\eps) + \alpha_3(v,w),
    \end{equation}
    and for $1 \le j \le 3$, $\alpha_j(v, w) \in \mathcal{F}^{b0}_{k,\eps}(\Phi^{b0}_k, m + 1, N)$. By assumption $j = j_0 + 1$, we have $N - (m+1) \ge j_0$. Hence the induction assumption implies that $\alpha_j(v, w)\beta(b(0) - w, w) \in \mathcal{F}^{b0}_{k,\eps}(g, m+n+1, m+n+1+j_0)$. 

    We next take one derivative in $w$. The analysis is the same as above. The only difference is that the $w$ derivative can be taken on $\beta(b(0) - w, w)$, which can generate more terms with $\log(b(0) - w + i\eps)$. We omit the details here. In conclusion, we showed that for $j = j_0 + 1$, $\alpha(v, w)\beta(b(0) - w, w)$ has the $\mathcal{F}^{b0}_{k,\eps}(g, m + n, m+n+j_0 + 1)$ type singularity structure. The lemma is proved.
\end{proof}

We next show that different components of the Green's function $\G_k(v+w, v'+w)$ keep the singularity structures.
\begin{lemma}\label{aus lemma 3}
Assume $m \ge 0$, $n \ge 0$, $k \ge 1$, $N\ge 0$ and $\lambda \ge 0$ are five integers. For any $g \in H^{\lambda + N }_k(\R)$, $\eps \in (-1, 1)\setminus\{0\}$ and any $\Gamma(v, w) \in \mathcal{F}^{b0}_{m, n, k,\eps}(g, \lambda, \lambda + N)$. Let $a_{k}(v, v', w)$ be a smooth function such that for any integers $r, s, t \ge 0$,
    \begin{equation}
        |\partial_v^r\partial_{v'}^s\partial_w^t a_k(v, v', w)| \lesssim k^{ r+s+t}.
    \end{equation}
         For $1 \le p \le \tau(n)$, $v, w \in \R$, define
        \begin{equation}
            \Lambda^p(v, w) := \int_\R e^{-k|\ib(v'+w)|}a_k(v, v', w)\sign(v'+w -b(0))\log^p(v'+i\eps)\Gamma(v', w)dv',
        \end{equation}
        and
        \begin{equation}
            \Omega^p(v, w) := \int_\R e^{-k|\ib(v'+w)|}a_k(v, v', w)\log^p(v'+i\eps)\Gamma(v', w)dv'.
        \end{equation}
    Then $\Lambda^p(v, w)$ and $\Omega^p(v, w)$ have the $\mathcal{F}^{b0}_{m, n, k,\eps}(g, \lambda , \lambda  + N)$ type singularity structure.
\end{lemma}
\begin{proof}
We prove the lemma with an induction argument for $N \ge 0$. For the case $N = 0$, we need to show that 
\begin{equation}
    \norm{\Lambda^p(v, w)}_{L^2_wH^1_k} + \norm{\Omega^p(v, w)}_{L^2_wH^1_k} \lesssim \norm{g}_{H^\lambda_k},
\end{equation}
which follows from the Cauchy-Schwarz inequality.

Assume for some integer $N_0 \ge 0$ the lemma holds for all $N \le N_0$. We now consider the case $N = N_0 + 1$. By definition $\Lambda^p(v, w)$ and $\Omega^p(v,w)$ are smooth in $v$. We only need to consider their derivatives in $w$. For $\Lambda^p(v, w)$, we have
\begin{equation}
\begin{split}
        &\partial_w \Lambda^p(v, w) \\
        = ~&2a_k(v, b(0) - w, w)\Gamma(b(0) - w, w)\log^p(b(0) - w + i\eps) \\
&-k\int_\R e^{-k|\ib(v'+w)|}\frac{a_k(v, v', w)}{B(v'+w)}\log^p(v'+i\eps)\Gamma(v', w)dv' \\
&+\int_\R \bigg[e^{-k|\ib(v'+w)|} \partial_wa_k(v, v', w)\sign(v'+w -b(0)) \\ &\qquad \times\log^p(v'+i\eps)\Gamma(v', w)\bigg]dv' \\
& +\int_\R e^{-k|\ib(v'+w)|}a_k(v, v', w)\sign(v'+w -b(0))\\ &\qquad \times\log^p(v'+i\eps)\partial_w\Gamma(v', w)dv'\\
:= ~&T_1 + T_2 + T_3 + T_4.
\end{split}
\end{equation}

For $T_1$ we have 
\begin{equation}
    2a_k(v, b(0) - w, w)\Gamma(b(0) - w, w) \in \mathcal{F}^{b0}_{m, n, k,\eps}(g, \lambda, \lambda + N_0 + 1).
\end{equation}

For $T_2$ and $T_3$ it follows from the induction assumption that 
\begin{equation}
    T_2, T_3 \in \mathcal{F}^{b0}_{m, n, k,\eps}(g, \lambda + 1, \lambda + N_0 + 1).
\end{equation}

For $T_4$, since we assumed that $\Gamma(v, w)\in \mathcal{F}^{b0}_{m, n, k,\eps}(g, \lambda, \lambda + N_0 + 1)$, there exist functions $B(v, w), R(v, w) \in \mathcal{F}^{b0}_{m, n+1, k,\eps}(g, \lambda + 1, \lambda + N_0 + 1)$ such that
\begin{equation}
    \partial_w \Gamma(v, w) = B(v, w)\log(b(0) - w + i\eps) + R(v, w).
\end{equation}
Set 
\begin{equation}
\begin{split}
      \widetilde{B}(v, w) = &\int_\R e^{-k|\ib(v'+w)|}a_k(v, v', w)\sign(v'+w -b(0))\\ &\quad\times\log^p(v'+i\eps)B(v', w)dv',
\end{split}
\end{equation}
and 
\begin{equation}
\begin{split}
    \widetilde{R}(v, w) = &\int_\R e^{-k|\ib(v'+w)|}a_k(v, v', w)\sign(v'+w -b(0))\\ &\quad\times\log^p(v'+i\eps)R(v', w)dv'.
\end{split}
\end{equation}
Using the induction assumption for $N = N_0$, we have 
\begin{equation}
    \widetilde{B}(v, w), ~\widetilde{R}(v, w) \in \mathcal{F}^{b0}_{m, n+1, k,\eps}(g, \lambda + 1, \lambda + N_0 + 1),
\end{equation}
and
\begin{equation}
    T_4 = \widetilde{B}(v, w) \log(b(0) - w + i\eps ) + \widetilde{R}(v, w).
\end{equation}
Therefore, we have $\Lambda^p(v, w) \in \mathcal{F}^{b0}_{m, n, k,\eps}(g, \lambda, \lambda + N_0 + 1)$. The analysis of $\Omega^p(v, w)$ is similar.
\end{proof}

Using similar ideas, we can also prove the following two lemmas.
\begin{lemma}\label{lemma 48}
    Assume $m \ge 0$, $n \ge 0$, $k \ge 1$, $N\ge 0$ and $\lambda \ge 0$ are five integers. For any $g \in H^{\lambda + N }_k(\R)$, $\eps \in (-1, 1)\setminus\{0\}$ and any $\Gamma(v, w) \in \mathcal{F}^{b0}_{m, n, k,\eps}(g, \lambda, \lambda + N)$. Let $a_{k}(v, v', w)$ be a smooth function such that for any integers $r, s, t \ge 0$,
    \begin{equation}
        |\partial_v^r\partial_{v'}^s\partial_w^t a_k(v, v', w)| \lesssim k^{ r+s+t}.
    \end{equation}
         For $1 \le p \le \tau(n)$, $v, w \in \R$, define
        \begin{equation}
        \begin{split}
            \Lambda^p(v, w) := &\int_\R e^{-k|1- \ib(v'+w)|}a_k(v, v', w)\sign(v'+w -b(1))\\ &\quad\times\log^p(v'+i\eps)\Gamma(v', w)\Upsilon_1(w)dv',
        \end{split}
        \end{equation}
        and
        \begin{equation}
            \Omega^p(v, w) := \int_\R e^{-k|1- \ib(v'+w)|}a_k(v, v', w)\log^p(v'+i\eps)\Gamma(v', w)\Upsilon_1(w)dv'.
        \end{equation}
    Then $\Lambda^p(v, w)$ and $\Omega^p(v, w)$ have the $\mathcal{F}^{b0}_{m, n, k,\eps}(g, \lambda , \lambda  + N)$ type singularity structure.
\end{lemma}
\begin{proof}
    Note that on the support of $\Upsilon_1(w)$, $w$ is supported away from the boundary point $b(1)$. Hence the component
    $$e^{-k|1- \ib(v'+w)|}\sign(v'+w -b(1))$$
    actually does not generate any singularity. We can apply an induction argument similar to that in the proof of Lemma \ref{aus lemma 3} to prove this lemma.
   
\end{proof}

\begin{lemma}\label{lemma 49}
    Assume $m \ge 0$, $n \ge 0$, $k \ge 1$, $N\ge 0$ and $\lambda \ge 0$ are five integers. For any $g \in H^{\lambda + N }_k(\R)$, $\eps \in (-1, 1)\setminus\{0\}$ and any $\Gamma(v, w) \in \mathcal{F}^{b0}_{m, n, k,\eps}(g, \lambda, \lambda + N)$. Let $a_{k}(v, v', w)$ be a smooth function such that for any integers $r, s, t \ge 0$,
    \begin{equation}
        |\partial_v^r\partial_{v'}^s\partial_w^t a_k(v, v', w)| \lesssim k^{ r+s+t}.
    \end{equation}
         For $1 \le p \le \tau(n)$, $v, w \in \R$, define
        \begin{equation}
        \begin{split}
             \Lambda^p(v, w) := &\int_\R e^{-k|\ib(v+w) - \ib(v'+w)|}a_k(v, v', w)\sign(v - v')\\ &\quad\times\log^p(v'+i\eps)\Gamma(v', w)dv',
        \end{split}
        \end{equation}
        and
        \begin{equation}
            \Omega^p(v, w) := \int_\R e^{-k|\ib(v+w) - \ib(v'+w)|}a_k(v, v', w)\log^p(v'+i\eps)\Gamma(v', w)dv'.
        \end{equation}
    Then $\Lambda^p(v, w)$ and $\Omega^p(v, w)$ have the $\mathcal{F}^{b0}_{m, n, k,\eps}(g, \lambda , \lambda  + N)$ type singularity structure.
\end{lemma}

Our main result in this section is the characterization of the singularity structure of the spectrum density function $\Theta^1_{k,\eps}(v, w)$ with $w$ close to the boundary point $b(0)$, assuming that the right-hand side of the equation has the $\mathcal{F}^{b0}_{0,0, k,\eps}(g, \lambda, \lambda + N)$ type singularity structure.

\begin{theorem}\label{thm 49}
    For any integer $N \ge 0$, $k \ge 1$, $\eps \in (-1, 1)\setminus \{0\}$. For any integer $\lambda \ge 1$ and function $g \in H^{\lambda + N}_k(\R)$, assume $\Gamma(v, w)$ has the $\mathcal{F}^{b0}_{0,0, k,\eps}(g, \lambda, \lambda + N)$ type singularity structure. Let $\Theta^1_{k,\eps}(v, w) := \Theta_{k,\eps}(v, w)\Upsilon_1(w)$ solve the following equation
    \begin{equation}\label{eq 497}
        \begin{split}
			\Theta^1_{k,\epsilon}(v,w) + T_{k,\eps} \Theta^1_{k,\eps}(v,w) = \Gamma(v, w)
   \end{split}
    \end{equation}
    where $T_{k,\eps}$ is the operator defined in (\ref{Operator T}). Then $\Theta^1_{k,\eps}(v, w)$ also has the $\mathcal{F}^{b0}_{0,0,k,\eps}(g, \lambda, \lambda + N)$ type singularity structure.
\end{theorem}
\begin{proof}
If $N = 0$, the theorem follows directly from the limiting absorption principle. We apply induction arguments to show that Theorem \ref{thm 49} holds for any $N \ge 0$. Assume for some $N_0 \ge 0$ the theorem holds for all $0 \le N \le N_0$. We now consider the case $N = N_0 + 1$. We first analyze remainder terms.
In order to streamline the notation, we set 
\begin{equation}
    R^{0, 0}(v, w) := \Theta^1_{k,\eps}(v, w).
\end{equation}
 We have the following claim.

{\bf Claim:} For $0 \le n \le N_0 + 1$ and $0\le m \le N_0 + 1 - n$, there exist functions $R^{m,n}(v, w)$ satisfying the following properties.
\begin{enumerate}
    \item For $0 \le n \le N_0 + 1$ and $0\le m \le N_0 + 1 - n$,
    \begin{equation}\tag{Claim 1}\label{claim 1}
        \norm{R^{m, n}(v, w)}_{L^2_w H^1_{k,v}} \lesssim \norm{g}_{H^{m + n + \lambda}_k}
    \end{equation}
\item For $0 \le n \le N_0$, there exist a function $A^n(v, w, k) \in \mathcal{F}^{b0}_{1, n, k, \eps}(g, \lambda + n + 1, \lambda + N_0 + 1)$ and smooth functions $\alpha^n_j(v, w, k)$, $0 \le j \le n$, such that
\begin{equation}\tag{Claim 2}\label{claim 2}
\begin{split}
        \partial_v R^{0, n}(v, w) = ~&\sum_{j = 0}^{n} \alpha^n_j(v,w, k) R^{0, j}(v, w)\log(v+i\eps) \\
        &+ A^n(v, w, k)\log(v+i\eps) + R^{1, n}(v, w),
\end{split}
\end{equation}
and for any integers $p, q \ge 0$,
\begin{equation}\tag{Claim 3}\label{claim 3}
 |\partial_v^p\partial_w^q \alpha_j^n(v, w, k)| \lesssim k^{n-j + p + q}.
\end{equation}
\item For $0 \le n \le N_0$, $1 \le m \le N_0 - n$, there exist functions $A^{m, n}_{j}(v, w) \in \mathcal{F}^{b0}_{m, n, k,\eps}(g, \lambda + m+n+1, \lambda + N_0+1)$, $j \in \{1, 2\}$ and smooth functions $\tau^{m, n}_j(v, w)$, $\eta^{m, n}_j(v, w)$ such that
\begin{equation}\tag{Claim 4}\label{claim 4}
\begin{split}
    &\partial_v R^{m, n}(v, w) \\= &\sum_{j = 0}^n \big[\tau^{m, n}_j(v,, w)R^{0, j}(v, w) + \eta^{m, n}_j(v, w)\partial_vR^{0, j}(v, w)\big]\log(v+i\eps) \\
    &+\sum_{j\in\{1, 2\}} A^{m, n}_{j}(v, w)\log^j(v+i\eps) + R^{m+1, n}(v,w).
\end{split}
\end{equation}
Furthermore, for any integer $p, q\ge 0$, 
\begin{equation}\tag{Claim 5}\label{claim 5}
        |\partial_v^p\partial_w^q g^{m, n}_j(v, w)| \lesssim k^{m+n-j+p+q}, \quad  |\partial_v^p\partial_w^q h^{m, n}_j(v, w)| \lesssim k^{m+n-j+p+q-1}
    \end{equation}
\item For $0 \le n \le N_0$, there exist a function $B^{n}(v, w, k)\in \mathcal{F}^{b0}_{k, \eps}(g, \lambda + n + 1, \lambda + N_0 + 1)$ and functions $\beta^{n}_j(v, w, k)$ and $\gamma^{n}_j(v, w, k)$, $0 \le j \le n$ such that
\begin{equation}\tag{Claim 6}\label{claim 6}
\begin{split}
        &\partial_w R^{0, n }(v,w) \\= ~&\sum_{j = 0}^{n} \beta^{n}_j(v,w,k)R^{0,j}(b(0) - w, w)\log(b(0) - w + i\eps) \\
        &+ \sum_{j = 0}^{n-1} \gamma^{n}_j(v,w,k)\partial_vR^{0,j}(b(0) - w, w)\log(b(0) - w + i\eps) \\
        &+ B^{n}(v, w, k)\log(b(0) - w + i\eps) + R^{0, n+1}(v, w).
\end{split}
\end{equation}
Further more, $\beta^{n}_j(v, w, k)$ has the $\mathcal{F}^{b0}_{0, 0, k,\eps}(\Phi^{b0}_k, n-j+1, N_0 + 1)$ type singularity structure; $\gamma^{n}_j(v, w, k)$ has the $\mathcal{F}^{b0}_{0, 0, k,\eps}(\Phi^{b0}_k, n-j, N_0 + 1)$ type singularity structure. The smooth function $\Phi^{b0}_k$ is defined in (\ref{phi b0})
\item For $0 \le n \le N_0$, 
\begin{equation}\label{claim 7}\tag{Claim 7}
    R^{0, n}(v, w) \in \mathcal{F}^{b0}_{0, n, k,\eps}(g, \lambda + n, \lambda + N_0).
\end{equation}
\item For $0 \le n \le N_0 - 1$, 
\begin{equation}\label{claim 8}\tag{Claim 8}
    R^{1, n}(v, w) \in \mathcal{F}^{b0}_{1, n, k,\eps}(g, \lambda + n + 1, \lambda + N_0).
\end{equation}
\end{enumerate}
{\bf Step 1: proof of the case $N = N_0 + 1$ assuming (\ref{claim 1}) to (\ref{claim 8}) are true.} For the case $N = N_0 + 1$, we would like to show that for any $0 \le n \le N_0 + 1$, $R^{0, n}(v, w)$ has the $\mathcal{F}^{b0}_{0, n, k,\eps}(g, n + \lambda, \lambda + N_0 + 1)$ type singularity structure. We apply another induction argument to show this. 

For $n = N_0 + 1$, it follows from (\ref{claim 1}) that $R^{0, N_0+1}(v, w)$ has the $\mathcal{F}^{b0}_{0, N_0 + 1, k,\eps}(g, \lambda + N_0+1, \lambda + N_0+1)$ type singularity structure. Assume for some $j_0 \ge 0$, $R^{0, N_0 + 1 -j_0}(v, w)$ has the $\mathcal{F}^{b0}_{0, N_0 + 1 -j_0, k,\eps}(g, \lambda + N_0 - j_0 + 1, \lambda + N_0+1)$ type singularity structure. Now let us consider the singularity structure of $R^{0, N_0 - j_0}(v, w)$.

It follows from (\ref{claim 6}) and (\ref{claim 2}) that 
\begin{equation}
\begin{split}
        &\partial_w R^{0, N_0-j_0 }(v,w)\\ = ~&\sum_{j = 0}^{N_0-j_0} \beta^{N_0-j_0}_j(v,w,k)R^{0,j}(b(0) - w, w)\log(b(0) - w + i\eps) \\
        &+ \sum_{j = 0}^{N_0-j_0-1} \gamma^{N_0-j_0}_j(v,w,k)R^{0,j}(b(0) - w, w)\log^2(b(0) - w + i\eps) \\
        &+ \sum_{j = 0}^{N_0-j_0-1} \zeta^{N_0-j_0}_j(v,w,k)R^{1,j}(b(0) - w, w)\log(b(0) - w + i\eps) \\
        &+ B^{N_0-j_0}(v, w, k)\log(b(0) - w + i\eps) + R^{0, N_0-j_0+1}(v, w).
\end{split}
\end{equation}
In addition, $\beta^{N_0-j_0}_j(v, w, k)$ has the $\mathcal{F}^{b0}_{k,\eps}(\Phi^{b0}_k, N_0-j_0-j+1, N_0 + 1)$ type singularity structure; $\gamma^{N_0-j_0}_j(v, w, k)$ and $\zeta^{N_0-j_0}_j(v, w, k)$ have the $\mathcal{F}^{b0}_{k,\eps}(\Phi^{b0}_k, N_0-j_0-j, N_0 + 1)$ type singularity structure. By (\ref{claim 7}), $R^{0,j}(v, w) \in \mathcal{F}^{b0}_{k,\eps}(g, j+\lambda, N_0 + \lambda)$ for $j \le N_0$.

Note that for $j_0 \ge 0$ and $0 \le j \le N_0 - j_0$, 
\begin{equation}
    (N_0 + 1) - (N_0 - j_0 - j + 1) \ge j_0 -1, \quad (N_0 + \lambda) - (j + \lambda) \ge j_0.
\end{equation}
It follows from Lemma \ref{aux lemma} that for $0 \le j \le N_0 - j_0$, $$\beta^{N_0-j_0}_j(v,w,k)R^{0,j}(b(0) - w, w) \in \mathcal{F}^{b0}_{k,\eps}(g, N_0-j_0+\lambda + 1, N_0 + 1 + \lambda).$$ 
In view of (\ref{claim 8}), the same argument holds for $\zeta^{N_0-j_0}_j(v,w,k)R^{1,j}(b(0) - w, w)$. For $\gamma^{N_0-j_0}_j(v,w,k)R^{0,j}(b(0) - w, w)$, we can show that 
$$\gamma^{N_0-j_0}_j(v,w,k)R^{0,j}(b(0) - w, w) \in \mathcal{F}^{b0}_{k,\eps}(g, N_0-j_0+\lambda, N_0 + \lambda ) $$
which is the subset of $\mathcal{F}^{b0}_{k,\eps}(g, N_0-j_0+\lambda + 1, N_0 + \lambda + 1 )$ by Lemma \ref{aux lemma 2}. By the induction assumption, $R^{0, N_0 - j_0 + 1}(v, w)$ has a $\mathcal{F}^{b0}_{k,\eps}(g, N_0-j_0+\lambda + 1, N_0 + \lambda + 1 )$ type singularity structure. Hence we have finished the analysis of $\partial_w R^{0, N_0 - j_0}$.

Next, let us consider $\partial_v R^{0, N_0 - j_0}$. From (\ref{claim 2}), we have 
\begin{equation}\label{eqq: 472}
\begin{split}
        \partial_v R^{0,N_0 - j_0}(v, w) = ~&\sum_{j = 0}^{N_0 - j_0} \alpha^{N_0 - j_0}_j(v,w, k) R^{0, j}(v, w)\log(v+i\eps) \\
        &+ A^{N_0 - j_0}(v, w, k)\log(v+i\eps) + R^{1, N_0 - j_0}(v, w),
\end{split}
\end{equation}
and for any integers $p, q \ge 0$,
\begin{equation}
 |\partial_v^p\partial_w^q \alpha_j^{N_0 - j_0}(v, w, k)| \lesssim k^{N_0 - j_0-j + p + q}.
\end{equation}
We have for $0 \le j \le N_0 - j_0$,
$$\alpha^{N_0 - j_0}_j(v,w, k) R^{0, j}(v, w) \in \mathcal{F}^{b0}_{k,\eps}(g, N_0-j_0+\lambda, N_0 + \lambda ) \subset \mathcal{F}^{b0}_{k,\eps}(g, N_0-j_0+\lambda + 1, N_0 + \lambda + 1 ).$$
Hence we only need to analyze $R^{1, N_0-j_0}(v,w)$. Note that the singularity structure of $\partial_w R^{0, N_0 - j_0}$ implies the singularity structure of $\partial_w R^{1, N_0-j_0}(v,w)$, since from (\ref{eqq: 472}) we can write 
\begin{equation}
\begin{split}
        &\partial_w R^{1, N_0-j_0}(v,w) \\= &~\partial_w \bigg[ \partial_v R^{0,N_0 - j_0}(v, w) - \sum_{j = 0}^{N_0 - j_0} \alpha^{N_0 - j_0}_j(v,w, k) R^{0, j}(v, w)\log(v+i\eps) \\
        &- A^{N_0 - j_0}(v, w, k)\log(v+i\eps)\bigg].
\end{split}
\end{equation}
 Therefore, we only need to analyze $\partial_v R^{1, N_0-j_0}(v,w)$. It follows from (\ref{claim 4}) that
\begin{equation}
\begin{split}
    &\partial_v R^{1, N_0 - j_0}(v, w)\\ = &\sum_{j = 0}^{N_0 - j_0} \big[g^{1, N_0 - j_0}_j(v, w)R^{0, j}(v, w) + h^{1, N_0 - j_0}_j(v, w)\partial_vR^{0, j}(v, w)\big]\log(v+i\eps) \\
    &+\sum_{j\in\{1, 2\}} A^{1, N_0 - j_0}_{j}(v, w)\log^j(v+i\eps) + R^{2, N_0 - j_0}(v,w),
\end{split}
\end{equation}
and for any integer $p, q\ge 0$, 
\begin{equation}
        |\partial_v^p\partial_w^q g^{1, N_0 - j_0}_j(v, w)| \lesssim k^{1 + N_0 - j_0 -j+p+q}
    \end{equation}
and
\begin{equation}
      |\partial_v^p\partial_w^q h^{1, N_0 - j_0}_j(v, w)| \lesssim k^{N_0 - j_0-j+p+q}.
\end{equation}
It follows from Lemma \ref{aux lemma 2} that both $g^{1, N_0 - j_0}_j(v, w)R^{0, j}(v, w)$ and $h^{1, N_0 - j_0}_j(v, w)R^{1, j}(v, w)$ are in $\mathcal{F}^{b0}_{1, N_0 - j_0, k,\eps}(g, N_0-j_0+\lambda + 2, N_0 + \lambda + 1 )$. Hence we only need to analyze $\partial_v R^{2, N_0 - j_0}(v, w)$. This procedure only need to be repeated finite times, since we have $R^{j_0 + 1, N_0 - j_0}(v, w) \in \mathcal{F}^{b0}_{j_0 + 1, N_0 - j_0, k,\eps}(g, N_0 + \lambda + 1,  N_0 +  \lambda + 1 )$ by (\ref{claim 1}). 

Therefore, we have showed that $R^{0, N_0 - j_0}(v, w)$ has the $\mathcal{F}^{b0}_{k,\eps}(g, N_0-j_0+ \lambda, N_0 + \lambda +1 )$ type singularity structure, which is the case $j = j_0 + 1$ in the induction argument. Hence $\Theta^1_{k,\eps}(v, w) = R^{0,0}(v, w)$ has the $\mathcal{F}^{b0}_{k,\eps}(g, \lambda, N_0 + \lambda + 1 )$ type singularity structure which is the case $N = N_0 + 1$ for the theorem. We next show that (\ref{claim 1}) to (\ref{claim 6}) are correct for all $0 \le n \le N_0$.

For the rest of the proof, we show that for $0 \le n \le N_0 + 1$ and $0 \le m \le N_0 + 1 -m$, there exist functions $R^{m,n}(v, w)$ such that (\ref{claim 1}) to (\ref{claim 8}) hold.

{\bf Step 2: iterative construction of $R^{0, n}(v, w)$ for $0 \le n \le N_0 $}. Set $$R^{0,0}(v, w) = \Theta^1_{k,\eps}(v, w).$$ We claim that for $0 \le n \le N_0$ and $0 \le p \le n - 1$, there exist smooth functions $a^{n,p}_k(v,v',w)$, $b^{n,p}_k(v,v',w)$, $c^{n,p}_k(v,v',w)$ and $d^{n,p}_k(v,v',w)$ such that the solution $R^{0,n}(v, w)$ to the following system satisfying (\ref{claim 6}) and (\ref{claim 7}).

\begin{equation}\label{eqq 497}\tag{IA 1}
    R^{0, n}(v, w) + T_{k,\eps}(R^{0, n}(v, w)) = \mathcal{R}^{0, n}_1(v, w) + \mathcal{R}^{0, n}_2(v, w) + \mathcal{R}^{0, n}_3(v, w) + \mathcal{R}^{0, n}_4(v, w),
\end{equation}
and the right-hand side of (\ref{eqq 497}) has the following expression.
\begin{enumerate}
    \item $\mathcal{R}^{0, n}_1(v, w)$ is from the singularity structure of $\Gamma(v, w)$ and 
    \begin{equation}\tag{IA 2}
        \mathcal{R}^{0, n}_1(v, w) \in \mathcal{F}^{b0}_{0, n, k,\eps}(g, n+\lambda, N_0 + \lambda + 1 ).
    \end{equation}
\item $\mathcal{R}^{0, n}_2(v, w)$:
\begin{equation}\label{eqqq 479}\tag{IA 3}
    \mathcal{R}^{0, n}_2(v, w) = - \sum_{j = 1}^{n} \binom{n}{j} \int_\R \partial_w^j \big[(\G^{fr}_k + \G^{b1}_k)\partial_{v'}B(v'+w)\big]\frac{R^{0, n-j}(v', w)}{v'+ i\eps}dv'.
\end{equation}
\item $\mathcal{R}^{0, n}_3(v, w)$:
\begin{equation}\label{eqq 4100}\tag{IA 4}
    \begin{split}
        &\mathcal{R}^{0, n}_3(v, w) = \int_\R e^{-k|\ib(v'+w)|}\sign(v'+w - b(0))\log(v'+ i\eps) \times \\
    &\bigg[\sum_{p = 0}^{n-1}a_k^{n,p}(v, v', w)R^{0, p}(v', w) + \sum_{p = 0}^{n-1}b_k^{n, p}(v, v', w)\partial_{v'}R^{0,p}(v', w)\bigg]dv'
    \end{split}
\end{equation}
where the functions $a_{k}^{n, p}(v, v', w)$ and $b_k^{n, p}(v, v', w)$ in (\ref{eqq 4100}) are smooth functions depending only on the cutoff function $\Psi_k$ and background flow $B(v+w)$ such that for any integers $\alpha, \beta \ge 0$,
\begin{equation}\tag{IA 5}\label{IA 5}
    |\partial_v^\alpha\partial_w^\beta a_{k}^{n,p}(v, v', w)| \lesssim k^{n-p+ \alpha + \beta}, \quad |\partial_v^\alpha \partial_w^\beta b_{k}^{n, p}(v, v', w)| \lesssim k^{n-p-j - 1 + \alpha + \beta}.
\end{equation}
\item $\mathcal{R}^{0, n}_4(v, w)$:
\begin{equation}\label{eqq 4102}\tag{IA 6}
    \begin{split}
        \mathcal{R}^{0, n}_4(v, w) = &\int_\R e^{-k|\ib(v'+w)|}\log(v'+ i\eps) \times 
    \bigg[\sum_{p = 0}^{n-1}c_k^{n, p}(v, v', w)R^{0, p}(v', w) \\&+ \sum_{p = 0}^{n-1}d_k^{n, p}(v, v', w)\partial_{v'}R^{0,p}(v', w)\bigg]dv'
    \end{split}
\end{equation}
where the functions $c_{k}^p(v, v', w)$ and $d_k^{j, p}(v, v', w)$ in (\ref{eqq 4102}) are smooth functions such that for any integers $\alpha, \beta \ge 0$,
\begin{equation}\label{IA 7}\tag{IA 7}
    |\partial_v^\alpha \partial_w^\beta c_{k}^p(v, v', w)| \lesssim k^{n-p+\alpha + \beta} \quad \text{and} \quad |\partial_v^\alpha \partial_w^\beta d_{k}^{ p}(v)| \lesssim k^{n-p-1 + \alpha + \beta}.
\end{equation}
\end{enumerate}

For the case $n = 0$, since $R^{0, 0}(v, w)$ solves 
\begin{equation}\label{eq: 491}
    R^{0, 0}(v, w) + T_{k,\eps}R^{0, 0}(v, w) = \Gamma(v, w) \in \mathcal{F}^{b0}_{0, 0, k,\eps}(g, \lambda, N_0 + \lambda + 1 ).
\end{equation}
It suffices to take $a^{0,p}_k(v,v',w)$, $b^{0,p}_k(v,v',w)$, $c^{0,p}_k(v,v',w)$ and $d^{0,p}_k(v,v',w)$ to be all zero. We next study $\partial_w R^{0,0}(v, w)$. Taking a derivative in $w$ to the equation (\ref{eq: 491}), we get
\begin{equation}\label{eq: 492}
\begin{split}
    &(I + T_{k,\eps})\partial_w R^{0,0}(v, w) \\= &~\partial_w \Gamma(v, w) - \int_\R \partial_w\big[\G_k(v+w, v'+w)\partial_{v'}B(v'+w)\big]\frac{R^{0,0}(v', w)}{v'+i\eps}dv'\\
    := &~T_1 + T_2.
\end{split}
\end{equation}
For $T_2$, we have
\begin{equation}
    \begin{split}
        T_2 = &-\int_\R \partial_w\bigg\{(\G_k^{fr} + \G_k^{b1})\partial_{v'}B(v'+w)\bigg\}\frac{R^{0,0}(v',w)}{v'+i\epsilon}dv' \\
        &-\int_\R \partial_w\bigg\{\G^{b0}_k(v+w, v'+w)\partial_{v'}B(v'+w)\bigg\}\frac{R^{0,0}(v',w)}{v'+i\epsilon}dv' \\
        := &T_{21} + T_{22}.
    \end{split}
\end{equation}
For $T_{22}$, we deduce using integration by parts that
\begin{equation}
\begin{split}
        T_{22} = &\Phi^{b0}_k(v+w)\partial_v B(b(0))R^{0,0}(b(0) - w, w)\log(b(0) - w + i\eps) \\
        & + \Phi^{b0}_k(v+w) \int_\R \partial_{v'}\bigg\{e^{-k|\ib(v'+w)|}\partial_{v'}B(v'+w)R^{0,0}(v', w)\bigg\} \\
        & \qquad \qquad \qquad \qquad \qquad \times \sign(v'+w- b(0))\log(v'+i\eps) dv' \\
        &+\partial_w \big(\Phi^{b0}_k(v+w)\big) \int_\R \partial_{v'}\bigg\{ \frac{e^{-k|\ib(v'+w)|}}{k}R^{0,0}(v', w)\partial_{v'}B(v'+w)\bigg\}\\&\qquad \qquad \qquad \qquad \qquad\times\log(v'+i\eps)dv'\\
        := & T_{221} + T_{222} + T_{223}.
\end{split}
\end{equation}
By the assumption that $\Gamma(v, w) \in \mathcal{F}^{b0}_{0,0,k,\eps}(g, \lambda, N_0 + \lambda + 1)$, there exist functions $H^1(v, w)$, $S^1(v, w)$ in $\mathcal{F}^{b0}_{0,1,k,\eps}(g, \lambda + 1, N_0 + \lambda + 1)$ such that
\begin{equation}
    T_1 = \partial_w \Gamma(v, w) = H^1(v,w)\log(b(0) - w + i\eps) + S^1(v, w).
\end{equation}
Set 
\begin{equation}
    \mathcal{B}^1(v, w) = \Phi^{b0}_k(v+w)\partial_v B(b(0))R^{0,0}(b(0) - w, w) + H^1(v, w),
\end{equation}
and
\begin{equation}
    \mathcal{S}^1(v, w) = T_{21} + T_{222} + T_{223} + S^1(v, w).
\end{equation}
It follows from (\ref{eq: 492}) that $\partial_w R^{0,0}(v, w)$ solves 
\begin{equation}
    (I + T_{k,\eps})\partial_w R^{0,0}(v, w) = \mathcal{B}^1(v, w)\log(b(0) - w + i\eps) + \mathcal{S}^1(v, w).
\end{equation}
Take $R^{0, 1}(v, w)$ as 
\begin{equation}\label{eq: 499}
    R^{0, 1}(v, w) = (I + T_{k,\eps})^{-1}\mathcal{S}^1(v, w).
\end{equation}
We have
\begin{equation}
    \begin{split}
        &\partial_w R^{0,0}(v, w) \\= &(I + T_{k,\eps})^{-1}\Phi^{b0}_k(v+w)\partial_vB(b(0))R^{0,0}(b(0) - w, w)\log(b(0) - w +i\eps) \\
        &+ (I + T_{k,\eps})^{-1} H^{1}(v, w)\log(b(0) - w +i\eps) \\
        &+ R^{0,1}(v, w).
    \end{split}
\end{equation}
Take 
\begin{equation}
    \beta^0_0 (v, w, k) = (I + T_{k,\eps})^{-1}\Phi^{b0}_k(v+w)\partial_vB(b(0))
\end{equation}
and
\begin{equation}
    B^0(v, w, k) = (I+T_{k,\eps})^{-1}H^1(v, w).
\end{equation}
Since $\Phi^{b0}_k(v+w)$ is smooth and we have assumed that the theorem holds for $N \le N_0$, we have $\beta^0_0 (v, w, k) \in \mathcal{F}^{b0}_{0,0,k,\eps}(g, 1, N_0 + 1)$. Since $H^1 \in \mathcal{F}^{b0}_{0,1,k,\eps}(g, \lambda + 1, N_0 + \lambda + 1)$, we also have $B^0(v, w, k) \in \mathcal{F}^{b0}_{0,1,k,\eps}(g, \lambda + 1, N_0 + \lambda + 1)$. Therefore, (\ref{claim 6}) holds for $n = 0$.

For $R^{0, 1}(v, w)$ defined in (\ref{eq: 499}), we take
\begin{align}
    a^{1, 0}_k(v, v', w) &= \Phi^{b0}_k(v+w) \partial_{v'}^2B(v'+w) - \partial_w \Phi^{b0}_k(v+w) \frac{\partial_{v'}B(v'+w)}{B(v'+w)}, \\
    b^{1, 0}_k(v, v', w) &= \Phi^{b0}_k(v+w) \partial_{v'}B(v'+w),\\
    c^{1, 0}_k(v, v', w) &= - k\Phi^{b0}_k(v+w)\frac{\partial_{v'}B(v'+w)}{B(v'+w)} + \frac{\partial_w \Phi^{b0}_k(v+w) }{k}\partial_{v'}^2B(v'+w), \\
    d^{1,0}_k(v, v', w) &= \frac{\partial_w \Phi^{b0}_k(v+w) }{k}\partial_{v'}B(v'+w).
\end{align}
We see that (\ref{eqq 497}) to (\ref{IA 7}) holds for $n = 1$. From the induction assumption, we have $R^{0,0}(v, w) \in \mathcal{F}^{b0}_{0,0,k,\eps}(g, \lambda, \lambda + N_0)$. It follows from Lemma \ref{aus lemma 3}, Lemma \ref{lemma 48} and Lemma \ref{lemma 49} that $\mathcal{S}^1(v, w) \in \mathcal{F}^{b0}_{0,0,k,\eps}(g, \lambda + 1, \lambda + N_0)$. Then we have $R^{0, 1}(v, w) \in \mathcal{F}^{b0}_{0,0,k,\eps}(g, \lambda + 1, \lambda + N_0)$ which is (\ref{claim 7}) for the case $n = 1$.

Now assume for some integer $n_0 \ge 1$ we have constructed $R^{0, n}(v, w)$ for all $0\le n \le n_0$ that satisfy (\ref{eqq 497}) to (\ref{IA 7}), and (\ref{claim 6}) holds for all $0\le n \le n_0 - 1$, (\ref{claim 7}) holds for all $0 \le n \le n_0$. We take a derivative in $w$ to the equation (\ref{eqq 497}) for $n = n_0$, and get
\begin{equation}\label{eqqqq 477}
\begin{split}
        &(I + T_{k,\eps})\partial_w R^{0,n_0}(v, w) \\= ~&-\int_\R \partial_w \bigg[\G_k(v+w, v'+w)\partial_{v'}B(v'+w)\bigg] \frac{R^{0, n_0}(v', w)}{v'+i\eps}dv'\\
         & + \partial_w\bigg[ \mathcal{R}^{0, n_0}_1(v, w) + \mathcal{R}^{0, n_0}_2(v, w) + \mathcal{R}^{0, n_0}_3(v, w) + \mathcal{R}^{0, n_0}_4(v, w)\bigg] \\
         := ~&T_3 + T_4 + T_5 + T_6 + T_7.
\end{split}
\end{equation}

For $T_3$, we have
\begin{equation}
    \begin{split}
        T_3 = &-\int_\R \partial_w\bigg\{(\G_k^{fr} + \G_k^{b1})\partial_{v'}B(v'+w)\bigg\}\frac{R^{0,n_0}(v',w)}{v'+i\epsilon}dv' \\
        &-\int_\R \partial_w\bigg\{\G^{b0}_k(v+w, v'+w)\partial_{v'}B(v'+w)\bigg\}\frac{R^{0,n_0}(v',w)}{v'+i\epsilon}dv' \\
        := &T_{31} + T_{32}.
    \end{split}
\end{equation}
By induction assumption of (\ref{claim 7}), $R^{0, n_0} \in \mathcal{F}^{b_0}_{0,n_0,k,\eps}(g, \lambda + n_0, \lambda + N_0)$, we have by Lemma \ref{lemma 48} and \ref{lemma 49}
\begin{equation}
    T_{31} \in \mathcal{F}^{b_0}_{0,n_0,k,\eps}(g, \lambda + n_0 + 1, \lambda + N_0)
\end{equation}
For $T_{32}$, using integration by parts we get 
\begin{equation}\label{eqqq 490}
\begin{split}
       T_{32} = &\Phi^{b0}_k(v+w)\partial_v B(b(0))R^{0, n_0}(b(0) - w, w)\log(b(0) - w + i\eps) \\
        & + \Phi^{b0}_k(v+w) \int_\R \partial_{v'}\bigg\{e^{-k|\ib(v'+w)|}\partial_{v'}B(v'+w)R^{0, n_0}(v', w)\bigg\} \\
        & \qquad \qquad \qquad \qquad \qquad \times \sign(v'+w- b(0))\log(v'+i\eps) dv' \\
        &+\frac{\partial_w \big(\Phi^{b0}_k(v+w)\big)}{k} \int_\R \partial_{v'}\bigg[e^{-k|\ib(v'+w)|}R^{0, n_0}(v', w)\partial_{v'}B(v'+w)\bigg] \\&\qquad \qquad \qquad \qquad \qquad \times \log(v'+i\eps)dv'.
\end{split}
\end{equation}

The analysis of $T_4 = \partial_w \mathcal{R}^{0, n_0}_1(v, w)$ is straightforward since the induction assumption assumes that $\mathcal{R}^{0, n_0}_1(v, w)$ is from the singularity structure of $\Gamma(v, w)$ and has the $\mathcal{F}^{b0}_{k,\eps}(g, \lambda + n_0, \lambda + N_0 + 1 )$ type singularity structure. There exist functions $H^{n_0}(v, w, k)$, $S^{n_0}(v, w, k)$ in $\mathcal{F}^{b0}_{k,\eps}(g, \lambda + n_0 + 1, \lambda + N_0 + 1 )$ such that
\begin{equation}
    T_4 = H^{n_0}(v, w, k)\log(b(0) - w + i\eps) + S^{n_0}(v, w, k).
\end{equation}

For $T_5$, we have
\begin{equation}
\begin{split}
    &\partial_w \mathcal{R}^{0, n_0}_2(v, w) \\ = ~&- \sum_{j = 1}^{n_0} \binom{n_0}{j} \int_\R \partial_w^{j+1} \big[(\G^{fr}_k + \G^{b1}_k)\partial_{v'}B(v'+w)\big]\frac{R^{0, n_0-j}(v', w)}{v'+ i\eps}dv' \\ 
    & - \sum_{j = 1}^{n_0} \binom{n_0}{j} \int_\R \partial_w^{j+1} \big[(\G^{fr}_k + \G^{b1}_k)\partial_{v'}B(v'+w)\big]\frac{\partial_wR^{0, n_0-j}(v', w)}{v'+ i\eps}dv' \\
    := ~&T_{51} + T_{52}.
\end{split}
\end{equation}
From (\ref{claim 6}) for the case $n \le n_0 - 1$, we have
\begin{equation}
\begin{split}
    &\partial_w R^{0, n_0 - j}(v, w) \\= ~&\sum_{p = 0}^{n_0 - j}\beta^{n_0-j}_p(v, w, k)R^{0,p}(b(0) - w, w)\log(b(0) - w + i\eps) \\
    &+ \sum_{p = 0}^{n_0 - j - 1}\gamma^{n_0-j}_p(v, w, k)\partial_vR^{0,p}(b(0) - w, w)\log(b(0) - w + i\eps)\\
    &+ B^{n_0 - j}(v, w)\log(b(0) - w + i\eps) + R^{0, n_0 - j + 1}(v, w).
\end{split}
\end{equation}
In addition, $\beta^{n_0 - j}_p(v, w, k)$ has the $\mathcal{F}^{b0}_{0, 0, k,\eps}(\Phi^{b0}_k, n_0 -j -p + 1, N_0 + 1)$ type singularity structure; $\gamma^{n_0 - j}_p(v, w, k)$ has the $\mathcal{F}^{b0}_{0, 0, k,\eps}(\Phi^{b0}_k, n_0-j_p, N_0 + 1)$ type singularity structure. For $0 \le p \le n_0 - 1$, set
\begin{equation}
\begin{split}
   \widehat{\beta}^{n_0}_p(v, w, k) = &- \sum_{j = 1}^{n_0 - p} \binom{n_0}{j} \int_\R \partial_w^{j+1} \big[(\G^{fr}_k + \G^{b1}_k)\partial_{v'}B(v'+w)\big]\\ &\qquad \qquad \qquad  \times\frac{\beta^{n_0 - j}_p(v', w, k)}{v'+ i\eps}dv', 
\end{split}
\end{equation}
\begin{equation}
\begin{split}
    \widehat{\gamma}^{n_0}_p(v, w, k) = &- \sum_{j = 1}^{n_0 - 1- p} \binom{n_0}{j} \int_\R \partial_w^{j+1} \big[(\G^{fr}_k + \G^{b1}_k)\partial_{v'}B(v'+w)\big]\\ &\qquad \qquad \qquad \quad\times\frac{\gamma^{n_0 - j}_p(v', w, k)}{v'+ i\eps}dv',
\end{split}
\end{equation}
and 
\begin{equation}
\begin{split}
    \widehat{B}^{n_0}(v, w, k) = &- \sum_{j = 1}^{n_0} \binom{n_0}{j} \int_\R \partial_w^{j+1} \big[(\G^{fr}_k + \G^{b1}_k)\partial_{v'}B(v'+w)\big]\\ &\qquad \qquad \qquad  \times\frac{B^{n_0 - j}(v', w, k)}{v'+ i\eps}dv'.
\end{split}
\end{equation}
We have $\widehat{\beta}^{n_0}_p(v, w, k)$ has the $\mathcal{F}^{b0}_{0, 0, k,\eps}(\Phi^{b0}_k, n_0 -p + 1, N_0 + 1)$ type singularity structure; $\widehat{\gamma}^{n_0}_p(v, w, k)$ has the $\mathcal{F}^{b0}_{0, 0, k,\eps}(\Phi^{b0}_k, n_0 - p, N_0 + 1)$ type singularity structure; $\widehat{B}^{n_0}(v, w, k)$ has the $\mathcal{F}^{b0}_{0, 0, k,\eps}(g, \lambda + n_0 + 1, \lambda + N_0 + 1)$ type singularity structure. In addition,
\begin{equation}
\begin{split}
     T_{52} = &\sum_{p = 0}^{n_0 - 1}\widehat{\beta}^{n_0}_p(v, w, k)R^{0,p}(b(0) - w, w)\log(b(0) - w + i\eps) \\
     &+\sum_{p=0}^{n_0 - 2}\widehat{\gamma}^{n_0}_p(v, w, k)\partial_vR^{0,p}(b(0) - w, w)\log(b(0) - w + i\eps) \\
     &+ \widehat{B}^{n_0}(v, w, k) \log(b(0) - w + i\eps) \\
     &- \sum_{j = 1}^{n_0} \binom{n_0}{j} \int_\R \partial_w^{j+1} \big[(\G^{fr}_k + \G^{b1}_k)\partial_{v'}B(v'+w)\big]\frac{R^{0, n_0-j + 1}(v', w)}{v'+ i\eps}dv'\\
     := & T_{521} + T_{522} + T_{523} + T_{524}.
\end{split}
\end{equation}

Take 
\begin{equation}
\begin{split}
    \mathcal{R}^{0, n_0 + 1}_2(v, w) = T_{31} + T_{51} + T_{524}.
\end{split}
\end{equation}

For $T_6 = \partial_w \mathcal{R}^{0, n_0}_3(v, w)$, in view of (\ref{eqq 4100}), we have
\begin{equation}\label{eqqq 492}
\begin{split}
        T_6 = ~&2\log(b(0) - w + i\eps)\sum_{p = 0}^{n_0-1}a_k^{n_0, p}(v, b(0) - w, w)R^{0, p}(b(0) - w, w) \\
        &+ 2\log(b(0) - w + i\eps)\sum_{p = 0}^{n_0 - 1}b_k^{n_0, p}(v, b(0) - w, w)\partial_vR^{0,p}(b(0) - w, w) \\
        &-k \int_\R e^{-k|\ib(v'+w)|}\frac{\log(v'+ i\eps)}{B(v'+w)} \times \bigg[\sum_{p = 0}^{n_0-1}a_k^{n_0, p}(v, v', w)R^{0, p}(v', w) 
        \\ & \qquad \qquad \qquad + \sum_{p = 0}^{n_0 - 1}b_k^{n_0, p}(v, v', w)\partial_{v'}R^{0,p}(v', w)\bigg]dv' \\
        &+\int_\R e^{-k|\ib(v'+w)|}\sign(v'+w - b(0))\log(v'+ i\eps) \times \partial_w\bigg[\\
        &\sum_{p = 0}^{n_0-1}a_k^{n_0, p}(v, v', w)R^{0, p}(v', w)
        + \sum_{p = 0}^{n_0 - 1}b_k^{n_0, p}(v, v', w)\partial_{v'}R^{0,p}(v', w)\bigg]dv'\\
        := ~& T_{61} + T_{62} + T_{63} + T_{64}.
\end{split}
\end{equation}

For $T_{64}$, we have
\begin{equation}
\begin{split}
   T_{64} = &\int_\R e^{-k|\ib(v'+w)|}\sign(v'+w - b(0))\log(v'+ i\eps) \times\\
        &\sum_{p = 0}^{n_0-1}\bigg[\partial_wa_k^{n_0, p}(v, v', w) R^{0, p}(v', w)
        + \partial_w b_k^{n_0, p}(v, v', w)\partial_{v'}R^{0,p}(v', w)\bigg]dv'\\
        & +\int_\R e^{-k|\ib(v'+w)|}\sign(v'+w - b(0))\log(v'+ i\eps)\times\\
        &\sum_{p = 0}^{n_0-1}\bigg[a_k^{n_0, p}(v, v', w)\partial_w R^{0, p}(v', w)
        + b_k^{n_0, p}(v, v', w)\partial_{v'}\partial_wR^{0,p}(v', w)\bigg]dv'\\
        := &T_{641} + T_{642}.
\end{split}
\end{equation}
Since we assumed that (\ref{claim 6}) holds for all $n \le n_0 - 1$, we have for $p \le n_0 - 1$ 
\begin{equation}
    \begin{split}
        &\partial_v \partial_w R^{0, p}(v, w) \\
        = ~&\partial_v \bigg(~\sum_{j = 0}^{p} \beta^{p}_j(v,w,k)R^{0,j}(b(0) - w, w)\log(b(0) - w + i\eps) \\
        &+ \sum_{j = 0}^{p-1} \gamma^p_j(v,w,k)\partial_vR^{0,j}(b(0) - w, w)\log(b(0) - w + i\eps) \\
        &+ B^{p}(v, w, k)\log(b(0) - w + i\eps) + R^{0, p+1}(v, w)\bigg).
\end{split}
\end{equation}
In addition, for $0 \le j \le p$, $\beta^{p}_j(v,w,k) \in \mathcal{F}^{b0}_{0,0, k,\eps}(\Phi^{b0}_k, p-j+1, N_0 + 1)$ and $\gamma^{p}_j(v,w,k) \in \mathcal{F}^{b0}_{0,0, k,\eps}(\Phi^{b0}_k, p-j, N_0 + 1)$. $B^p(v, w, k)$ is from the singularity structure of $\Gamma(v, w)$ and $B^p(v, w, k) \in \mathcal{F}^{b0}_{0,p, k,\eps}(g, \lambda + p, \lambda+ N_0 + 1)$. For $0 \le j \le n_0 - 1$, define
\begin{equation}
\begin{split}
        \widetilde{\beta}^{n_0}_j(v,w,k) = &\sum_{p = j}^{n_0 - 1}\int_\R e^{-k|\ib(v'+w)|}\sign(v'+w - b(0))\log(v'+ i\eps)\times \\
        &\bigg[a^{n_0, p}_k(v, v', w)\beta^p_j(v', w, k) + b^{n_0, p}_k(v,v',w)\partial_{v'}\beta^p_j(v',w,k)\bigg]dv',
\end{split}
\end{equation}
\begin{equation}
\begin{split}
        \widetilde{\gamma}^{n_0}_j(v,w,k) = &\sum_{p = j+1}^{n_0 - 1}\int_\R e^{-k|\ib(v'+w)|}\sign(v'+w - b(0))\log(v'+ i\eps)\times \\
        &\bigg[a^{n_0, p}_k(v, v', w)\gamma^p_j(v', w, k) + b^{n_0, p}_k(v,v',w)\partial_{v'}\gamma^p_j(v',w,k)\bigg]dv',
\end{split}
\end{equation}
and
\begin{equation}
\begin{split}
        \widetilde{B}^{n_0}(v,w,k) = &\sum_{p = 0}^{n_0 - 1}\int_\R e^{-k|\ib(v'+w)|}\sign(v'+w - b(0))\log(v'+ i\eps)\times \\
        &\bigg[a^{n_0, p}_k(v, v', w)B^p(v', w, k) + b^{n_0, p}_k(v,v',w)\partial_{v'}B^p(v',w,k)\bigg]dv'.
\end{split}
\end{equation}

It then follows from Lemma \ref{aus lemma 3} that $\widetilde{\beta}^{n_0}_j(v, w, k) \in \mathcal{F}^{b0}_{0,0, k,\eps}(\Phi^{b0}_k, n_0 - j + 1, N_0 + 1)$, $\widetilde{\gamma}^{n_0}_j(v, w, k) \in \mathcal{F}^{b0}_{0,0, k,\eps}(\Phi^{b0}_k, n_0 - j, N_0 + 1)$ and $\widetilde{B}^{n_0}_j(v,w,k) \in \mathcal{F}^{b0}_{0,0, k,\eps}(g, n_0 + \lambda + 1, N_0 + \lambda + 1)$, and $T_{642}$ can be written as
\begin{equation}
    \begin{split}
        T_{642} = &\sum_{j = 0}^{n_0-1} \widetilde{\beta}^{n_0}_j(v,w,k)R^{0,j}(b(0) - w, w)\log(b(0) - w + i\eps) \\
        &+ \sum_{j = 0}^{n_0-2} \widetilde{\gamma}^{n_0}_j(v,w,k) \partial_vR^{0, j}(b(0) - w, w)\log(b(0) - w + i\eps) \\
        &+\widetilde{B}^{n_0}(v, w, k)\log(b(0) - w + i\eps)\\
        &+\int_\R e^{-k|\ib(v'+w)|}\sign(v'+w - b(0))\log(v'+ i\eps)\\
        &\sum_{p = 0}^{n_0-1}\bigg[a_k^{n_0, p}(v, v', w)R^{0, p+1}(v', w)
        + b_k^{n_0, p}(v, v', w)\partial_{v'}R^{0,p+1}(v', w)\bigg]dv'.
    \end{split}
\end{equation}

For $T_7$, in view of (\ref{eqq 4102}), we have
\begin{equation}\label{eqqq 493}
\begin{split}
        T_7 = ~&\partial_w \mathcal{R}^{0, n}_4(v, w) \\
= ~& -k\int_\R e^{-k|\ib(v'+w)|}\sign(v'+w-b(0))\frac{\log(v'+ i\eps)}{B(v'+w)} \times 
    \bigg[\\& \sum_{p = 0}^{n_0-1}c_k^{n_0, p}(v, v', w)R^{0, p}(v', w) + \sum_{p = 0}^{n_0 - 1}d_k^{n_0, p}(v, v', w)\partial_{v'}R^{0,p}(v', w)\bigg]dv'\\
    &+ \int_\R e^{-k|\ib(v'+w)|}\log(v'+ i\eps) \times 
    \partial_w\bigg[\sum_{p = 1}^{n_0-1}c_k^{n_0, p}(v, v', w)R^{0, p}(v', w) \\&+ \sum_{p = 0}^{n_0 - 1}d_k^{n_0, p}(v, v', w)\partial_{v'}R^{0,p}(v', w)\bigg]dv'.
\end{split}
\end{equation}
The analysis of $T_7$ is similar to $T_{63}$ and $T_{64}$ above.

Set 
\begin{equation}
\begin{split}
        \mathcal{B}^{n_0}(v, w, k) = &\Phi^{b0}_k(v+w)\partial_v B(b(0))R^{0, n_0}(b(0) - w, w) \\
        &+2\sum_{p = 0}^{n_0-1}a_k^{n_0, p}(v, b(0) - w, w)R^{0, p}(b(0) - w, w) \\
        &+ 2\sum_{p = 0}^{n_0 - 1}b_k^{n_0, p}(v, b(0) - w, w)\partial_vR^{0,p}(b(0) - w, w)\\
        &+\sum_{p = 0}^{n_0 - 1}\big[\widehat{\beta}^{n_0}_p(v, w, k) + \widetilde{\beta}^{n_0}_p(v, w, k)\big]R^{0,p}(b(0) - w, w) \\
     &+\sum_{p=0}^{n_0 - 2}\big[\widehat{\gamma}^{n_0}_p(v, w, k)+ \widetilde{\gamma}^{n_0}_p(v, w, k)\big]\partial_vR^{0,p}(b(0) - w, w) \\
     &+ \widehat{B}^{n_0}(v, w, k) + \widetilde{B}^{n_0}(v, w, k) + H^{n_0}(v, w, k),
\end{split}
\end{equation}
and
\begin{align}
    \mathcal{R}^{0, n_0+1}_1(v, w) &= \mathcal{S}^{n_0}(v,w,k) \\
    \mathcal{R}^{0, n_0+1}_2(v, w) &= T_{31} + T_{51} + T_{524}.
\end{align}
Furthermore, from the analysis of $T_{32}$, $T_6$ and $T_7$, we define $\mathcal{R}^{0, n_0+1}_3(v, w)$ and $\mathcal{R}^{0, n_0+1}_4(v, w)$ by setting for $0 \le p \le n_0 -1$
\begin{align*}
    a^{n_0 + 1, p}_k(v, v', w) &= \partial_w a_k^{n_0, p}(v, v', w) + a_k^{n_0, p-1}(v,v',w) - \frac{kc^{n_0, p}_k(v,v',w)}{B(v'+w)} \\
    b^{n_0 + 1, p}_k(v, v', w) &= \partial_w b_k^{n_0, p}(v, v', w) + b_k^{n_0, p-1}(v,v',w) - \frac{kd^{n_0, p}_k(v,v',w)}{B(v'+w)} \\
    c^{n_0 + 1, p}_k(v,v',w) &= \partial_w c_k^{n_0, p}(v, v', w) + c_k^{n_0, p-1}(v,v',w) - \frac{ka^{n_0, p}_k(v,v',w)}{B(v'+w)} \\
    d^{n_0 + 1, p}_k(v,v',w) &= \partial_w d_k^{n_0, p}(v, v', w) + d_k^{n_0, p-1}(v,v',w) - \frac{kb^{n_0, p}_k(v,v',w)}{B(v'+w)},
\end{align*}
and for $p = n_0$
\begin{align*}
    a^{n_0 + 1, n_0}_k(v, v', w) = &\Phi^{b0}_k(v+w)\partial_{v'}^2B(v'+w) - \partial_w \Phi^{b0}_k(v+w) \frac{\partial_{v'}B(v'+w)}{B(v'+w)} \\
    & +a^{n_0 , n_0 -1}_k(v, v', w), \notag\\
    b^{n_0 + 1, n_0}_k(v, v', w) = &\Phi^{b0}_k(v+w)\partial_{v'}B(v'+w) +b^{n_0, n_0 -1}_k(v, v', w), \\
    c^{n_0 + 1, n_0}_k(v, v', w) = &-k\Phi^{b0}_k(v+w)\frac{\partial_{v'}B(v'+w)}{B(v'+w)} +  \frac{\partial_w\Phi^{b0}_k(v+w)}{k}\partial^2_{v'}B(v'+w)\\
    &+c^{n_0, n_0 - 1}, \notag \\
    d^{n_0 + 1, n_0}_k(v, v', w) = &\frac{\partial_w\Phi^{b0}_k(v+w)}{k}\partial_{v'}B(v'+w) + d^{n_0, n_0 -1}_k(v, v', w).
\end{align*}
It follows from the induction assumption and Lemma \ref{aus lemma 3}, Lemma \ref{lemma 48} and Lemma \ref{lemma 49} that for $1 \le j \le 4$, $\mathcal{R}^{0,n_0+1}_j(v, w)\in \mathcal{F}^{b_0}_{0, n_0 + 1, k,\eps}(g, \lambda + n_0 + 1, \lambda + N_0)$.
We then define 
\begin{equation}
    B^{n_0}(v, w, k) = (I+T_{k,\eps})^{-1}\mathcal{B}^{n_0}(v, w, k)
\end{equation}
and
\begin{equation}
    R^{0, n_0 + 1}(v, w) = (I + T_{k,\eps})^{-1}\bigg(\sum_{j=1}^4\mathcal{R}^{0,n_0+1}_j(v, w)\bigg).
\end{equation}
We have 
\begin{equation}
    \partial_w R^{0,n_0}(v,w) = B^{n_0}(v, w, k) \log(b(0) - w + i\eps) + R^{0, n_0 + 1}(v, w).
\end{equation}
We apply the assumption that the theorem holds for $N \le N_0$ and get
\begin{equation}
    R^{0, n_0 + 1}(v, w) \in \mathcal{F}^{b_0}_{0, n_0+1, k,\eps}(g, \lambda + n_0 + 1, \lambda + N_0).
\end{equation}
Therefore, (\ref{claim 6}) holds for the case $n = n_0$ and (\ref{claim 7}) holds for the case $n = n_0 + 1$.

{\bf Step 3: construction of $R^{m,n}(v, w)$ for $1 \le m \le N_0 - n$}. Fix $0 \le n \le N_0$. For $1 \le m \le N_0 - n$, we construct $R^{m, n}(v, w)$ in the following way.
\begin{equation}\tag{IB 1}\label{IB 1}
    R^{m,n}(v, w) = \Omega^{m,n}_1(v, w) + \Omega^{m,n}_2(v, w) + \Omega^{m,n}_3(v, w) + \Omega^{m,n}_4(v, w)
\end{equation}
and the terms in (\ref{IB 1}) have the following expression.
\begin{enumerate}
    \item $\Omega^{m,n}_1(v, w)$ is from the singularity structure of $\Gamma(v, w)$, and $\Omega^{m,n}_1(v, w)$ has the $\mathcal{F}^{b0}_{k,\eps}(g, m+n+\lambda, N_0+\lambda + 1)$ type singularity structure.
    \item Recall that $\mathcal{R}^{0, n}_3(v, w)$ and $\mathcal{R}^{0,n}_4(v, w)$ are defined in (\ref{eqqq 479}) and (\ref{eqq 4100}). We have
    \begin{equation}\tag{IB 2}
    \begin{split}
         \Omega^{m,n}_2(v, w)= ~&\partial_v^m \big[\mathcal{R}^{0, n}_3(v, w) + \mathcal{R}^{0, n}_4(v, w)\big] \\
         & - \partial_v^m \int_\R \G^b_k(v+w, v'+w)\frac{\partial_{v'}B(v'+w)R^{0,n}(v', w)}{v'+i\eps}dv' \\
         &- \sum_{j = 1}^{n} \binom{n}{j} \int_\R \partial_v^m\partial_w^j \big[\G^{b1}_k(v+w, v'+w)\partial_{v'}B(v'+w)\big]\\ &\qquad \qquad \qquad \times\frac{R^{0, n-j}(v', w)}{v'+ i\eps}dv'.
    \end{split}
    \end{equation}
    \item For $\Omega^{m,n}_3(v, w)$, given smooth functions $g^{m, n,j}_k(v, v', w)$ and $h^{m, n,j}_k(v, v', w)$ ($0 \le j \le n$),
    \begin{equation}\tag{IB 3}
        \begin{split}
            &\Omega^{m,n}_3(v, w) \\
            = ~&\int_\R e^{-k|\ib(v+w) - \ib(v'+w)|}\sign(v-v')\log(v'+i\eps)\\
            &\times \sum_{j = 0}^n\big[g^{m, n,j}_k(v, v', w)R^{0, j}(v', w) + h^{m, n,j}_k(v, v', w)\partial_{v'}R^{0, j}(v', w)\big]dv'.
        \end{split}
    \end{equation}
    In addition, for any integers $p,q \ge 0$ the coefficient functions satisfy
    \begin{equation}\tag{IB 4}
    \begin{split}
        &|\partial_v^p\partial_w^q g^{m, n,j}_k(v, v', w)| \lesssim k^{m+n-j+p+q}, \\& |\partial_v^p\partial_w^q h^{m, n,j}_k(v, v', w)| \lesssim k^{m+n-j+p+q-1}.
    \end{split}
    \end{equation}
    \item For $\Omega^{m,n}_4(v, w)$, given smooth functions $\sigma^{m, n,j}_k(v, v', w)$ and $\rho^{m, n,j}_k(v, v', w)$ ($0 \le j \le n$),
    \begin{equation}\tag{IB 5}
        \begin{split}
            &\Omega^{m,n}_4(v, w) \\= ~&\int_\R e^{-k|\ib(v+w) - \ib(v'+w)|}\log(v'+i\eps)\\
            &\times \sum_{j = 0}^n\big[\sigma^{m, n,j}_k(v, v', w)R^{0, j}(v', w) + \rho^{m, n,j}_k(v, v', w)\partial_{v'}R^{0, j}(v', w)\big]dv'.
        \end{split}
    \end{equation}
    In addition, for any integers $p,q \ge 0$ the coefficient functions satisfy
    \begin{equation}\tag{IB 6}\label{IB 8}
    \begin{split}
        &|\partial_v^p\partial_w^q \sigma^{m, n,j}_k(v, v', w)| \lesssim k^{m+n-j+1+p+q}, \\&|\partial_v^p\partial_w^q \rho^{m,n,j}_k(v, v', w)| \lesssim k^{m+n-j+p+q}.
    \end{split}
    \end{equation}
\end{enumerate}

We claim that for $0 \le n \le N_0$ and $1 \le m \le N_0 - m$, there exist smooth functions $g^{m,n,j}_k(v,v',w)$, $h^{m,n,j}_k(v,v',w)$, $\sigma^{m,n,j}_k(v,v',w)$ and $\rho^{m,n,j}_k(v,v',w)$ such that the above system generates $R^{m,n}(v,w)$ and (\ref{claim 2}) to (\ref{claim 5}) holds.

For the case $m = 1$, we take a derivative in $v$ to the equation (\ref{eqq 497}) of $R^{0,n}$, and get
\begin{equation}
    \partial_v R^{0,n}(v, w) = -\partial_v T_{k,\eps}(R^{0, n}) + \partial_v \sum_{j = 1}^4 \mathcal{R}^{0,n}_j(v,w).
\end{equation}

For $-\partial_v T_{k,\eps}(R^{0, n})$, we have
\begin{equation}
\begin{split}
    -\partial_v T_{k,\eps}(R^{0, n}) = &- \partial_v  \int_\R\G^{fr}_k(v+w, v'+w)\frac{\partial_{v'}B(v'+w)R^{0, n}(v', w)}{v'+i\epsilon}dv' \\
            &- \partial_v \int_\R\G^{b}_k(v+w, v'+w)\frac{\partial_{v'}B(v'+w)R^{0,n}(v',w)}{v'+i\epsilon}dv'\\
            := & T_{a} + T_{b}.
\end{split}
\end{equation}
Using integration by parts, we get 
\begin{equation}
    \begin{split}
                &T_{a} \\= &-\frac{2\Psi_k(v+w)\partial_vB(v+w)}{B(v+w)}R^{0,n}(v,w)\log(v+i\eps) \\
                & + \frac{\Psi_k(v+w)}{B(v+w)}\int_\R \bigg\{\partial_{v'}\big[ e^{-k|\ib(v+w) - \ib(v'+w)|}\partial_{v'}B(v'+w)R^{0,n}(v',w)\big] \\
        &\quad \quad \quad \times \sign(v-v')\log(v'+i\eps)\bigg\}dv' \\
        &+\partial_v \Psi_k(v+w)\int_\R \bigg\{\partial_{v'}\big[\frac{ e^{-k|\ib(v+w) - \ib(v'+w)|}}{k}\partial_{v'}B(v'+w)R^{0, n}(v',w)\big] \\
        & \quad \quad \quad \times \log(v'+i\eps)\bigg\}dv'\\
        := &T_{a1} + T_{a2} + T_{a3}.
    \end{split}
    \end{equation}

Since $\mathcal{R}^{0,n}_1(v, w) \in \mathcal{F}^{b0}_{0,n, k,\eps}(g, \lambda + n, \lambda + N_0 + 1)$, there exist functions $X^{0, n}(v, w), Y^{0,n}(v, w) \in \mathcal{F}^{b0}_{1,n, k,\eps}(g, \lambda + n + 1, \lambda + N_0 + 1)$ such that
\begin{equation}
    \partial_v \mathcal{R}_1^{0,n}(v, w) = X^{0, n}(v, w) \log(v+i\eps) + Y^{0,n}(v, w).
\end{equation}

For $\partial_v \mathcal{R}^{0,n}_2(v, w)$, we have
\begin{equation}
\begin{split}
    &\partial_v \mathcal{R}^{0,n}_2(v, w) \\= &- \sum_{j = 1}^{n} \binom{n}{j} \int_\R \partial_w^j \partial_v\big[\G^{fr}_k(v+w, v'+w)\partial_{v'}B(v'+w)\big]\frac{R^{0, n-j}(v', w)}{v'+ i\eps}dv' \\
    &- \sum_{j = 1}^{n} \binom{n}{j} \int_\R \partial_w^j \partial_v\big[\G^{b1}_k(v+w, v'+w)\partial_{v'}B(v'+w)\big]\frac{R^{0, n-j}(v', w)}{v'+ i\eps}dv'\\
    :=& T_{c} + T_{d}.
\end{split}
\end{equation}
Using integration by parts, we have 
\begin{equation}
    \begin{split}
        &T_{c1} \\= &-\sum_{j = 1}^n \binom{n}{j}\partial_w^j\bigg[\frac{2\Psi_k(v+w)\partial_vB(v+w)}{B(v+w)}\bigg]R^{0,n-j}(v,w)\log(v+i\eps) \\
                & + \sum_{j = 1}^n \binom{n}{j}\int_\R \partial_{v'}\bigg[\partial_w^j\bigg(\frac{\Psi_k(v+w)}{B(v+w)} e^{-k|\ib(v+w) - \ib(v'+w)|}\partial_{v'}B(v'+w)\bigg)\\ 
                & \quad \quad \quad \times R^{0,n-j}(v',w)\bigg]\sign(v-v')\log(v'+i\eps)dv' \\
        &+ \sum_{j = 1}^n \binom{n}{j} \int_\R\partial_{v'}\bigg[ \partial_w^j \bigg(\partial_v \Psi_k(v+w)\frac{ e^{-k|\ib(v+w) - \ib(v'+w)|}}{k}\partial_{v'}B(v'+w)\bigg) \\
        & \quad \quad \quad \times R^{0, n-j}(v',w)\bigg] \log(v'+i\eps)dv'\\
        := &T_{c1} + T_{c2} + T_{c3}.
    \end{split}
\end{equation}

Take 
\begin{align}
    &\Omega^{1, n}_1(v, w) = Y^{0,n}(v,w) \\
    &\Omega^{1, n}_2(v, w) = \partial_v(\mathcal{R}^{0, n}_3(v, w) + \mathcal{R}^{0, n}_3(v, w)) + T_b + T_d.
\end{align}
From $T_{a2}$, $T_{a3}$, $T_{c2}$ and $T_{c3}$, we define $\Omega^{1, n}_3(v, w)$ and $\Omega^{1, n}_4(v, w)$ by setting for $0 \le j \le n$
\begin{align*}
    g^{1,n, j}_k(v, w) = &\bigg\{\partial_w^{n-j}\bigg(\frac{\Psi_k(v+w)}{B(v+w)}e^{-k|\ib(v+w) - \ib(v'+w)|}\partial_{v'}^2 B(v'+w)\bigg) \notag\\
    & - \partial_w^{n-j}\bigg(\partial_v \Psi_k(v+w) e^{-k|\ib(v+w) - \ib(v'+w)|} \frac{\partial_{v'}B(v'+w)}{B(v'+w)}\bigg)\bigg\}\notag\\
    &\times \binom{n}{j}e^{k|\ib(v+w) - \ib(v'+w)|},\\
    h^{1,n,j}_k(v, w) = & \partial_w^{n-j}\bigg( \frac{\Psi_k(v+w)}{B(v+w)}e^{-k|\ib(v+w) - \ib(v'+w)|}\partial_{v'} B(v'+w)\bigg)\notag\\
    &\times \binom{n}{j}e^{k|\ib(v+w) - \ib(v'+w)|}, \\
    \sigma^{1, n, j}_k(v, w) = &\bigg\{-k\partial_w^{n-j}\bigg( \frac{\Psi_k(v+w)}{B(v+w)}e^{-k|\ib(v+w) - \ib(v'+w)|}\frac{\partial_{v'} B(v'+w)}{B(v'+w)}\bigg)\notag\\
     &+ \partial_w^{n-j}\bigg(\partial_v \Psi_k(v+w) \frac{e^{-k|\ib(v+w) - \ib(v'+w)|}}{k} \partial^2_{v'}B(v'+w)\bigg)\bigg\}\notag\\
    &\times \binom{n}{j}e^{k|\ib(v+w) - \ib(v'+w)|},\\
    \rho^{1, n, j}_k(v, w) = &\partial_w^{n-j}\bigg(\partial_v \Psi_k(v+w) \frac{e^{-k|\ib(v+w) - \ib(v'+w)|}}{k} \partial_{v'}B(v'+w)\bigg)\notag \\
    &\times \binom{n}{j}e^{k|\ib(v+w) - \ib(v'+w)|}.
\end{align*}

Set $R^{1,n}(v, w)$ as
\begin{equation}
    R^{1,n}(v,w) = \sum_{j = 1}^4 \Omega_j^{1,n}(v,w).
\end{equation}
We see that (\ref{claim 2}) holds for all $0 \le n \le N_0$ with 
\begin{equation}
    \alpha_j^n(v, w, k) = \binom{n}{j}\partial_w^{n - j}\bigg[\frac{2\Psi_k(v+w)\partial_vB(v+w)}{B(v+w)}\bigg].
\end{equation}
for $0 \le j \le n.$

Assume for some integer $m_0 \ge 1$, we have constructed functions $R^{m, n}(v, w)$ for $m \le m_0$, and (\ref{claim 2}) to (\ref{claim 5}) holds for $m = m_0 - 1$. Taking a derivative in $v$ to $R^{m_0, n}$, we get
\begin{equation}
    \partial_v R^{m_0, n} = \partial_v \big[\Omega^{m_0,n}_1(v, w) + \Omega^{m_0,n}_2(v, w) + \Omega^{m_0,n}_3(v, w) + \Omega^{m_0,n}_4(v, w)\big]
\end{equation}
We only need to analyze $\partial_v \Omega^{m_0,n}_3(v, w)$ and $\partial_v \Omega^{m_0,n}_3(v, w)$. It turns out the we can define $R^{m_0+1, n}(v, w)$ by setting
\begin{align}
    g^{m_0+1, n,j}_k(v,v',w) &= \partial_v g^{m_0,n,j}_k(v,v',w) -k\frac{\sigma_k^{m_0, n, j}(v, v', w)}{B(v+w)}, \\
    h^{m_0+1, n,j}_k(v,v',w) &= \partial_v h^{m_0, n, j}_{k}(v,v',w),\\
    \sigma^{m_0+1, n,j}_k(v,v',w) &= -k\frac{g^{m_0,n,j}_k(v,v',w)}{B(v+w)} + \partial_v\sigma_k^{m_0, n, j}(v, v', w), \\
    \rho^{m_0+1, n,j}_k(v,v',w) &= \partial_v\rho^{m_0, n,j}_k(v,v',w).
\end{align}
Then (\ref{claim 4}) and (\ref{claim 5}) holds with 
\begin{equation}
    \tau^{m_0, n}_j(v, w) = 2g^{m_0, n,j}_k(v, w), \quad \eta^{m_0, n}_j(v, w) = 2h^{m_0, n, j}_k(v,w).
\end{equation}

{\bf Step 4: estimate of $R^{m,n}$ for $0\le n \le N_0 + 1$, $m = N_0+1-m$}. In order to complete the proof, we only need to prove (\ref{claim 1}) for $0\le n \le N_0 + 1$, $m = N_0+1-m$ under the assumption that $\Gamma(v, w) \in \mathcal{F}_{0,0,k,\eps}^{b0}(g, \lambda, \lambda + N_0 + 1)$. We first show that 
\begin{equation}
    \norm{R^{1, N_0}(v,w)}_{L^2_wH^1_{k,v}} \lesssim \norm{g}_{H^{\lambda + N_0 + 1}_k}.
\end{equation}
By (\ref{claim 7}), for $0 \le n \le N_0$ we have
\begin{equation}
    \norm{R^{0, n}(v,w)}_{L^2_wH^1_{k,v}} \lesssim \norm{g}_{H^{\lambda + n}_k}.
\end{equation}
It follows from (\ref{claim 2}) that $R^{1, N_0} \in L^2_vL^2_w$ is defined and is given by (\ref{IB 1}) for $m = 1$ and $N = N_0$. We have 
\begin{equation}
    \norm{R^{1, N_0}(v,w)}_{L^2_wL^2_v} \lesssim \sum_{j = 1}^4 \norm{\Omega^{1,N_0}_j(v,w)}_{L^2_wL^2_v} \lesssim \frac{1}{k}\norm{g}_{H^{\lambda + N_0+1}}.
\end{equation}
Differentiate $R^{1, N_0}(v, w)$ in $v$ and we get
\begin{equation}\label{eq: 4165}
    \partial_v R^{1, N_0}(v, w) = 2h^{1,N_0, N_0}_k(v,v,w) \partial_v R^{0, N_0}\log(v+i\eps) + R(v, w),
\end{equation}
where
\begin{equation}
    \norm{R(v,w)}_{L^2_v L^2_w} \lesssim \norm{g}_{H^{\lambda + N_0 + 1}_k}.
\end{equation}
By (\ref{claim 2}), there exists a function $M(v, w) \in L^2_w H^{1}_{k, v}$ such that
\begin{equation}
    \norm{M(v, w)}_{L^2_w H^1_{k,v}} \lesssim \norm{g}_{H^{\lambda + N_0 + 1}_k},
\end{equation}
and 
\begin{equation}
    \partial_v R^{0, N_0}(v, w) = M(v, w)\log(v+i\eps) + R^{1, N_0}(v, w).
\end{equation}
From (\ref{eq: 4165}), we deduce that $R^{1, N_0}$ sovles the following equation
\begin{equation}
\begin{split}
    \partial_v R^{1, N_0}(v, w) = &\big[2h^{1,N_0, N_0}_k(v,v,w) \log(v+i\eps)\big]R^{1, N_0}(v, w) \\
    &+ 2h^{1,N_0, N_0}_k(v,v,w)M(v, w) \log^2(v+i\eps) + R(v, w).
\end{split}
\end{equation}
Since $|h^{1,N_0, N_0}_k| \lesssim 1$, we have
\begin{equation}
\begin{split}
    &\norm{\partial_v R^{1, N_0}(v, w)}_{L^2_wL^2_v} \\ \lesssim &\norm{2h^{1,N_0, N_0}_k(v,v,w)M(v, w) \log^2(v+i\eps) + R(v, w)}_{L^2_wL^2_v} \\
    \lesssim &\norm{g}_{H^{\lambda + N_0 + 1}_k}.
\end{split}
\end{equation}
Therefore, 
\begin{equation}
    \norm{R^{1, N_0}(v, w)}_{L^2_wH^{1}_{k,v}} \lesssim \norm{g}_{H^{\lambda + N_0 + 1}_k}.
\end{equation}
Then by (\ref{eqq 497}) to (\ref{IA 7}), $R^{0, N_0 + 1}(v,w)$ is also defined and 
\begin{equation}
    \norm{R^{0, N_0 + 1}(v,w)}_{L^2_wH^1_{k,v}} \lesssim \norm{g}_{H^{\lambda + N_0 + 1}_k}.
\end{equation}
The estimate of $R^{m,n}(v,w)$ for $m = N_0 + 1 - n$ and $0 \le n \le N_0 - 1$ follows directly from (\ref{IB 1}) to (\ref{IB 8}) and estimates of $R^{0, p}(v, w)$ and $R^{1, p}(v, w)$, $0 \le p \le n$.
\end{proof}

\begin{theorem}\label{thm 410}
    For any $k \ge 1$, $N \ge 1$, $0 < \eps < 1$, $\iota \in \{+, -\}$, assume $f_0^k \in H^N_k(\R)$ and $\Theta^\iota_{k, \eps}(v, w)$ solves
    \begin{equation}
		\begin{split}
			&\Theta^\iota_{k,\epsilon}(v,w) + \int_\R\G_k(v+w, v'+w)\frac{\partial_{v'}B(v'+w)\Theta^\iota_{k,\epsilon}(v',w)}{v'+i\iota\epsilon}dv' \\
			= &\int_\R\G_k(v+w, v'+w)\frac{1}{B(v'+w)}\frac{f^k_0(v'+w)}{v'+i\iota\epsilon}dv'.
		\end{split}
	\end{equation}
 Then $\Theta^{2, \iota}_{k,\epsilon}(v,w) = \Theta^\iota_{k,\epsilon}(v,w) \Upsilon_2(w)$ has the $\mathcal{F}^{in}_{0, k,\eps}(f_0^k, 1, N)$ type singularity structure.
\end{theorem}
\begin{proof}
    It follows from Proposition \ref{structure of RHS} and Theorem \ref{thm 49} that $\Theta^{2, \iota}_{k,\epsilon}(v,w)$ has the $\mathcal{F}^{b0}_{0, 0, k,\eps}(f_0^k, 1, N)$ type regularity structure. Noting that on the support of $\Upsilon_2(w)$, $w - b(0) > \delta$
    for some $\delta > 0$. Hence $\Upsilon_2(w)\log(b(0) - w + i\eps)$ is a Schwarz function. The definition of $\mathcal{F}^{b0}_{0, 0, k,\eps}(f_0^k, 1, N)$ type singularity structure implies that $\Theta^{2, \iota}_{k,\epsilon}(v,w)$ has the $\mathcal{F}^{in}_{0, k,\eps}(f_0^k, 1, N)$ type singularity structure.
\end{proof}

\begin{remark}
    The major terms in the singularity structure of $\Theta_{k, \eps}^{\iota}$ are explicitly computable. As a matter of fact, one can verify with integration by parts that 
    \begin{equation}
        \partial_v \Theta_{k, \eps}^{\iota} = A^\iota_{k,\eps}(v, w)\log(v+i\eps) + R_{k,\eps}^\iota(v,w)
    \end{equation}
    where $A^\iota_{k,\eps}$ is given explicitly by
    \begin{equation}\label{major term in v}
        \begin{split}
            A^\iota_{k,\eps}(v, w) = \frac{2\partial_vB(v+w)}{B(v+w)}\Theta^\iota_{k,\eps}(v,w) - \frac{2}{B^2(v+w)}f_0^k(v+w).
        \end{split}
    \end{equation}

Taking a derivative in $w$, we have
    \begin{equation}
    \begin{split}
        &\partial_w \Theta_{k, \eps}^{\iota} 
        \\= &D^{0,\iota}_{k,\eps}(v,w)\log(b(0) - w + i\eps) + D^{1,\iota}_{k,\eps}(v,w)\log(b(1) - w + i\eps) + S_{k,\eps}^\iota(v,w), 
    \end{split}
    \end{equation}
    where for $j \ in \{0, 1\}$, 
    \begin{equation}\label{major term in w}
        D^{j,\iota}_{k,\eps}(v,w)=\Phi^{j}_{k,\iota\eps}(v, w)\frac{2f_0^k(b(j))}{(b'(j))^2} - \Phi^{j}_{k,\iota\eps}(v, w)\frac{2b''(j)\Theta_{k,\eps}^\iota(b(j)-w, w)}{b'(j)},
    \end{equation}
    and $\Phi^{j}_{k, \eps}(v, w)$ is defined in Proposition \ref{solve Phi}.
\end{remark}

\begin{proposition}\label{vanish}
    Under the same condition of Theorem \ref{thm 49}, if we further assume that $N \ge 2$ and  $f_0^k(b(0)) = 0$, then $\partial_w \Theta^{1,\iota}_{k,\eps}(v, w)$ has the $\mathcal{F}^{b0}_{k,\eps}(f_0^k, 2, N)$ type singularity structure.
\end{proposition}
\begin{proof}
Note that under the change of variable $v = b(y)$, $w = b(y_0)$,
$$\Theta_{k,\eps}^\iota(b(0) - w, w) = \psi_{k,\eps}^{\iota}(0, y_0) = 0.$$
Hence when $f_0^k(b(0)) = 0$, we see from (\ref{major term in w}) that the major term $D^{0,\iota}_{k,\eps}(v, w) \equiv 0$. 
Therefore, it follows from Theorem \ref{thm 49} that $\partial_w \Theta^{1,\iota}_{k,\eps}(v, w)$ has the $\mathcal{F}^{b0}_{k,\eps}(f_0^k, 2, N)$ type singularity structure.
\end{proof}

\section{Proof of the main theorem}

\subsection{Technical lemmas}
In order to capture the regularity of different components of the stream function, we need the following lemmas.
\begin{lemma} \label{lemma 54}
	Suppose $h(v, w)$ is supported on $(-L, L)\times (-L, L)$ for some $L > 0$ and $h(v, w)$ is $H^{\frac{1}{2}+\delta_0}$ in $v$ and $H^s$ in $w$ for some $s \ge 0$ and $\delta_0 > 0$. For any non-negative integer $p'$, $\eps > 0$ and $k\in \mathbb{Z}$, set
	\begin{equation}
		g^{p'}_{k, \eps}(t, v) = \int_\R e^{-ikwt}h(v-w, w)\frac{\log^{p'}(v-w+i\eps)}{v-w+i\eps}dw.
	\end{equation}
	Then there exists a function $\alpha_k^{p'}(t, v)$ such that $g^{p'}_{k, \eps}(t, v)$ converges to
	\begin{equation}
		g^{p'}_k(t, v) := e^{-ikvt}\alpha_k^{p'}(t, v).
	\end{equation}
 in distribution as $\eps \to 0+$. In addition, following estimates hold.

 \begin{enumerate}
     \item If $p' = 0$, then 
     \begin{equation}\label{ineqqq 55}
		\norm{\alpha^{0}_{k}(t, \cdot)}_{H^{s}} \lesssim_{\delta, s, \delta_0}\norm{h(v,w)}_{H^{1/2 + \delta_0}_vH^{s}_w}.
	\end{equation}
\item If $p' > 0$, then for any $0 < \delta < \delta_0$ we have
	\begin{equation}
		\norm{\alpha^{p'}_{k}(t, \cdot)}_{H^{s-\delta}} \lesssim_{\delta, s, p', \delta_0}\big|1+\log^{p'}\la kt \ra\big|\norm{h(v,w)}_{H^{1/2 + \delta_0}_vH^{s-\delta/2}_w}.
	\end{equation}
 \end{enumerate}

\end{lemma}
\begin{proof}
	Set 
	\begin{equation}
		\mu(v, \xi) = \int_\R h(v-w, w)e^{-iw\xi}dw.
	\end{equation}
	Take \begin{equation}
		z^{p'}_\eps(w)  = \varphi(w)\frac{\log^{p'}(w+i\eps)}{w + i\eps}
	\end{equation}
where $\varphi(w)$ is a smooth cutoff function with $\varphi(w) = 1$ on $(-L, L)$. We have 
\begin{equation}
	z^{p'}_\eps(w) = \frac{d}{dw}\big[\varphi(w)\log^{p'+1}(w + i\eps)\big] - \frac{d\varphi(w)}{dw}\log^{p'+1}(w+i\eps)
\end{equation}
It follows from Lemma \ref{Fourier estimate with log} in the Appendix that 
\begin{equation}
	|\mathcal{F}(z^{p'}_\eps)(\xi)| \lesssim 1 + \log^{p'}\la\xi\ra.
\end{equation}
We have 
	\begin{equation}
		g^{p'}_{k,\eps}(t, v) = e^{-ikvt}\int_\R \mu(v, \xi)\widehat{z^{p'}_\eps}(kt-\xi)e^{iv\xi}d\xi.
	\end{equation}
Define 
	\begin{equation}
		\alpha_{k,\eps}^{p'}(t, v) = \int_\R \mu(v, \xi)\widehat{z^{p'}_\eps}(kt-\xi)e^{iv\xi}d\xi.
	\end{equation}
	We have 
	\begin{equation}\label{eq 513}
		\widehat{\alpha^{p'}_{k,\eps}}(t, \eta) = \int_\R \widetilde{h}(\eta-\xi, \eta)\widehat{z^{p'}_\eps}(kt-\xi)d\xi.
	\end{equation}
	Here $\widetilde{h}(\xi, \eta)$ denotes the Fourier transform of $h(v, w)$ in both variables. If $p' > 0$, for any $\delta > 0$ we apply Cauchy-Schwarz inequality and get
	\begin{equation} \label{eq 515}
		\begin{split}
			&\int_\R \la\eta\ra^{2s-2\delta}\big|\widehat{\alpha^{p'}_{k,\eps}}(t, \eta)\big|^2d\eta \\ \lesssim_{s, \delta, \delta_0, p'}& \int_\R\bigg[\la\eta\ra^{2s-2\delta}\int_\R\frac{1}{\la\eta - \xi\ra^{1+2\delta_0}}\bigg|1+ \log^{2p'}\la kt-\xi\ra\bigg|d\xi\\
			&\qquad \quad \quad \times\int_\R \la\eta - \xi\ra^{1+2\delta_0} |\widetilde{h}(\eta-\xi, \eta)|^2d\xi\bigg]d\eta \\
			\lesssim_{s, \delta, \delta_0, p'}& \int_\R\int_R \la\eta\ra^{2s-\delta}\la\xi\ra^{1+2\delta_0} \bigg|1 + \log^{2p'}\la kt\ra\bigg||\widetilde{h}(\xi, \eta)|^2 d\xi d\eta \\
			\lesssim_{s, \delta, \delta_0, p'}& \big|1 + \log^{2p'}\la kt\ra \big|\norm{h(v,w)}^2_{H^{1/2+\delta_0}_vH^{s-\delta/2}_w}.
		\end{split}
	\end{equation}
 If $p' = 0$, we have for $\xi \in \R$
 \begin{equation}
     |\widehat{z^{p'}_\eps}(\xi)| \lesssim 1.
 \end{equation}
 Then (\ref{ineqqq 55}) follows directly from (\ref{eq 513}) and the Cauchy-Schwarz inequality. 
\end{proof}

Similarly, we also have the following Lemma. The idea of proof is similar to Lemma \ref{lemma 54}
\begin{lemma}\label{lemma 55}
    	Suppose $h(v, w)$ is compactly supported on $(-L, L)\times (-L, L)$ for some $L > 0$ and $h(v, w)$ is $H^{\frac{1}{2}+\delta_0}$ in $w$ and $H^s$ in $v$ for some $s \ge 0$ and $\delta_0 > 0$. For any non-negative integer $p'$, $\eps > 0$ and $k \ge 1$, set
	\begin{equation}
		g^{p'}_{k, \eps}(t, v) = \int_\R e^{-ikwt}h(v-w, w)\frac{\log^{p'}(b(0)-w+i\eps)}{b(0)-w+i\eps}dw.
	\end{equation}
	Then there exists a function $\beta_k^{p'}(t, v)$ such that $g^{p'}_{k, \eps}(t, v)$ converges to
	\begin{equation}
		g^{p'}_k(t, v) := e^{-ikb(0)t}\beta_k^{p'}(t, v).
	\end{equation}
 in distribution as $\eps \to 0+$. In addition, following estimates hold.

 \begin{enumerate}
     \item If $p' = 0$, then 
     \begin{equation}\label{ineq 55}
		\norm{\beta^{0}_{k}(t, \cdot)}_{H^{s}} \lesssim_{\delta, s, \delta_0}\norm{h(v,w)}_{H^{1/2 + \delta_0}_wH^{s}_v}.
	\end{equation}
\item If $p' > 0$, then for any $0 < \delta < \delta_0$ we have
	\begin{equation}
		\norm{\beta^{p'}_{k}(t, \cdot)}_{H^{s-\delta}} \lesssim_{\delta, s, p', \delta_0}\big|1+\log^{p'}\la kt \ra\big|\norm{h(v,w)}_{H^{1/2 + \delta_0}_wH^{s-\delta/2}_v}.
	\end{equation}
 \end{enumerate}
\end{lemma}

\subsection{Proof of main theorems}
By (\ref{F5}) and the change of variables (\ref{change of variables}) $v = b(y)$, we have for $t \ge 0$, $v \in \R$
\begin{equation} \label{eq 50}
\begin{split}
	\phi_k(t, v) &:= \psi_k(t, \ib(v)) \\
 &= -\frac{1}{2\pi i}\lim_{\epsilon\to0+}\int_{b(0)}^{b(1)} e^{-ikwt}\bigg[\Theta^-_{k, \eps}(v-w, w) - \Theta^+_{k, \eps}(v-w, w)\bigg]dw.
\end{split}
\end{equation}
Since $\Theta^\iota_{k, \eps}(v, w)$ is smooth for $w$ outside of $[b(0), b(1)]$, we can safely rewrite (\ref{eq 50}) as
\begin{equation} \label{eq 51}
	\phi_k(t, v) = -\frac{1}{2\pi i}\lim_{\epsilon\to0+}\int_{-\infty}^\infty e^{-ikwt}\bigg[\Theta^-_{k, \eps}(v-w, w) - \Theta^+_{k, \eps}(v-w, w)\bigg]dw.
\end{equation}
By (\ref{eq 43}) we can seperate the range of $w$ into three parts to capture different type of singularities and rewrite (\ref{eq 51}) as 
\begin{equation}\label{eqqqq 522}
	\begin{split}
		\phi_k(t, v) &= -\frac{1}{2\pi i} \sum_{j = 1}^3 \lim_{\eps \to 0+}\int_{-\infty}^\infty e^{-ikwt}\bigg[\Theta^{j,-}_{k, \eps}(v-w, w) - \Theta^{j,+}_{k, \eps}(v-w, w)\bigg]dw \\
		&:= \sum_{j=1}^3\phi^j_k(t, v).
	\end{split}
\end{equation}

Using integration by parts, we get for $j\in\{1, 2, 3\}$
\begin{equation} \label{eq 510}
	\begin{split}
		&\phi^j_k(t, v) \\=& -\frac{1}{2\pi i}\frac{1}{(-ikt)^2}\lim_{\eps \to 0+}\int_\R e^{-ikwt} \frac{d^2}{dw^2}\bigg[\Theta^{j,-}_{k, \eps}(v-w, w) - \Theta^{j,+}_{k, \eps}(v-w, w)\bigg]dw \\
		=& -\frac{1}{2\pi i}\frac{1}{(-ikt)^2}\bigg\{ \lim_{\eps \to 0+}\int_\R e^{-ikwt} \bigg[\partial^2_v\Theta^{j, -}_{k, \eps}(v-w, w) - \partial^2_v \Theta^{j, +}_{k, \eps}(v-w, w)\bigg] dw \\
		&+  \lim_{\eps \to 0+}\int_\R e^{-ikwt} \bigg[2\partial_v\partial_w\Theta^{j,-}_{k, \eps}(v-w, w) - 2\partial_v\partial_w\Theta^{j,+}_{k, \eps}(v-w, w)\bigg] \\ &-\lim_{\eps \to 0+}\int_\R e^{-ikwt} \bigg[\partial^2_w\Theta^{j, -}_{k, \eps}(v-w, w) - \partial^2_w \Theta^{j,+}_{k, \eps}(v-w, w)\bigg] dw\bigg\} \\
		:=& T_1^j + T_2^j + T_3^j.
	\end{split}
\end{equation}

\subsubsection{Proof of Theorem \ref{interior regularity}}
For the new variable $v = b(y)$, we define the cutoff function $\mu^{in}(v) = \chi^{in}(v)$. As in (\ref{eqqqq 522}), we write 
\begin{equation}
    \phi_k(t, v)\mu^{in}(v) = \phi^1_k(t, v)\mu^{in}(v) + \phi^2_k(t, v)\mu^{in}(v) + \phi^3_k(t, v)\mu^{in}(v).
\end{equation}

{\bf Step 1: analysis of $T_1^2$.}  We first analyze $\phi_k^2(t, v)\mu^{in}(v)$. For $T_1^2$ defined in (\ref{eq 510}), it follows from Theorem \ref{thm 410} that there exist functions $A^{j, \iota}_{k,\eps}(v, w)$, $1 \le j \le 4$, and a remaining term $R^\iota_{k,\eps}(v, w)$ such that 
\begin{equation}
	\begin{split}
		\partial_v^2\Theta_{k,\eps}^{2, \iota}(v, w) = \frac{A^{1,\iota}_{k, \eps}(v, w)}{v+ i\iota\eps} + \sum_{j = 2}^4A^{j,\iota}_{k, \eps}(v, w) \log^{j-1}(v+i\iota\eps) + R^\iota_{k,\eps}(v, w),
	\end{split}
\end{equation}
and for any integer $m \ge 0$
\begin{equation}\label{eq 523}
    \norm{\partial_w^mA^{1,\iota}_{k, \eps}(v, w)}_{L^2_wH^{1}_{k,v}} \lesssim \norm{f_0^k}_{H^{2+m}_k},
\end{equation}
\begin{equation}\label{eq 524}
    \sum_{j=2}^4\norm{\partial_w^mA^{j,\iota}_{k, \eps}(v, w)}_{L^2_wH^1_{k,v}} + \norm{\partial_w^mR^\iota_{k,\eps}(v, w)}_{L^2_wH^1_{k,v}} \lesssim \norm{f_0^k}_{H^{3+m}_k}.
\end{equation}

By Lemma \ref{lemma 54}, there exists a function $\varpi_{2,1}^{\iota}(t,v, k)$ such that
\begin{equation}
    \lim_{\eps \to 0+} \int_{\R}e^{-ikwt}A^{1, \iota}_{k,\eps}(v-w, w)\frac{1}{v- w +i\iota\eps}dw = e^{-ikvt}\varpi_{2,1}^{\iota}(t,v, k),
\end{equation}
and for any $\delta > 0$
\begin{equation}
    \norm{\partial_v^m \varpi_2^\iota(t, \cdot, k)}_{L^2} \lesssim_{m, \delta} \norm{A^\iota_{k,\eps}(v, w)}_{H^{1/2+\delta}_vH^m_w}\lesssim_{m}\norm{f_0^k}_{H^{m+2}_k}.
\end{equation}
Set 
\begin{equation}
    \widetilde{R}^\iota_{k,\eps}(v, w) = \sum_{j = 2}^4A^{j,\iota}_{k, \eps}(v, w) \log^{j-1}(v+i\iota\eps) + R^\iota_{k,\eps}(v, w).
\end{equation}
It follows from (\ref{eq 523}) and (\ref{eq 524}) that
\begin{equation}
    \norm{\partial_w^m \widetilde{R}^\iota_{k,\eps}(v, w)}_{L^2_wL^2_v} \lesssim_{m} \norm{f_0^k}_{H^{m+3}_k}.
\end{equation}
Therefore, by change of variable $\tilde{w} = v - w$, we have 
\begin{equation}
\begin{split}
        \lim_{\eps \to 0+} \int_\R e^{-ikwt} \widetilde{R}^\iota_{k,\eps}(v - w, w)dw  &= e^{-ikvt}  \lim_{\eps \to 0+} \int_\R e^{ik\tilde{w}t} \widetilde{R}^\iota_{k,\eps}(\tilde{w}, v - \tilde{w})d\tilde{w} \\
        &:= e^{-ikvt} \varpi_{2,2}^\iota(t, v, k),
\end{split}
\end{equation}
and for any $m \ge 0$,
\begin{equation}
    \norm{\partial_v^m \varpi_{2,2}^\iota(t, \cdot, k)}_{L^2(\R)} \lesssim_{m} \norm{f_0^k}_{H^{3+m}_k(\R)}.
\end{equation}

{\bf Step 2: analysis of $T_3^2$}. For $T_3^2$, $\Theta_{k,\eps}^{2, \iota}(v, w)$ is smooth in $w$ since $w$ is supported away from the boundary. We have by change of variable $\tilde{w} = v - w$
 \begin{equation}
 \begin{split}
     &T_3^2\mu^{in}(v) \\= &\frac{1}{2\pi i (kt)^2} \lim_{\eps \to 0+}\int_\R e^{-ikwt} \bigg[\partial^2_w\Theta^{2, -}_{k, \eps}(v-w, w) - \partial^2_w \Theta^{2,+}_{k, \eps}(v-w, w)\bigg]\mu^{in}(v) dw \\
     = &\frac{e^{-ikvt}}{2\pi i(kt)^2}\lim_{\eps \to 0+} \int_\R e^{ik\tilde{w}t} \bigg[\partial^2_w\Theta^{2, -}_{k, \eps}(\tilde{w}, v - \tilde{w}) - \partial^2_w \Theta^{2,+}_{k, \eps}(\tilde{w}, v - \tilde{w})\bigg]\mu^{in}(v) d\tilde{w}.
\end{split}
 \end{equation}
 Now $v - \bar{w}$ is bounded away from $b(0)$ and $b(1)$, thus $\partial_w^2\Theta_{k,\eps}^{2, \iota}(\bar{w}, v - \bar{w})$ is smooth in $v$. It follows from Theorem \ref{thm 410} that there exists a function $\varpi_{2, 3}^\iota(t,v, k)$ such that
\begin{equation}
T^2_3 \mu^{in}(v)= \frac{1}{2\pi i}\frac{e^{-ikvt}}{k^2 t^2}\varpi_{2,3}^\iota(t,v, k)
\end{equation}
and for any integer $m \ge 0$
\begin{equation}
    \norm{\partial_v^m\varpi_{2,3}^\iota(t, \cdot, k)}_{L^2} \lesssim_{m} \norm{f_0^k}_{H^{m+3}_{k}}.
\end{equation}

{\bf Step 3: analysis of $T_2^2$}. For $T^2_2$, the singularity for $\partial_v\partial_w \Theta_{k,\eps}^{2, \iota}(v, w)$ is of $\log(v + i\iota\eps)$ type. Hence we can repeat the analysis of $T_3^2$ and deduce that there exist a function $\varpi_{2,4}^\iota(t,v, k)$ such that
\begin{equation}
T^2_2 \mu^{in}(v)= \frac{1}{2\pi i}\frac{e^{-ikvt}}{k^2 t^2}\varpi_{2,4}^\iota(t,v, k)
\end{equation}
and
\begin{equation}
    \norm{\partial_v^m\varpi_{2,4}^\iota(t, \cdot, k)}_{L^2} \lesssim_{m}\norm{f_0^k}_{H^{m+3}_{k}}.
\end{equation}

Set 
\begin{equation}
    \varpi_2(t, v, k) := \frac{1}{2\pi i}\sum_{j = 1}^4 \big[\varpi_{2,j}^-(t, v, k) - \varpi_{2,j}^+(t, v, k)\big],
\end{equation}
and 
\begin{equation}
    \alpha^{in}(t,y, k) := \varpi_2(t, b(y), k).
\end{equation}
We have 
\begin{equation}
    \psi_k^2(t, y)\chi^{in}(y) = \frac{e^{-ikb(y)t}}{k^2 t^2}\alpha^{in}(t,y, k),
\end{equation}
and for any integer $m \ge 0$,
\begin{equation}
    \norm{\partial_y^m\alpha^{in}_k(t,\cdot, k)}_{L^2} \lesssim_{m}\norm{f_0^k}_{H^{m+3}_k(\R)}.
\end{equation}

{\bf Step 4: analysis of $\psi_k^1(t, y)\chi^{in}(y)$ and $\psi_k^3(t, y)\chi^{in}(y)$}. By symmetry we only need to analyze $\psi_k^1(t, y)\chi^{in}(y)$.  For $\Theta^{1, \iota}_{k, \eps}(v-w, w)$, $w$ is supported close to $b(0)$ while $v$ is supported away from the boundary. Hence $|v - w| > \delta_0$ for some $\delta_0 > 0$. It follows from Theorem \ref{thm 49} that $\Theta^{1, \iota}_{k, \eps}(v-w, w)$ is smooth in $v$, and
\begin{equation}
    \psi_k^1(t, y)\chi^{in}(y) = \frac{e^{-ikb(0)t}}{k^2 t^2}\beta^{in}(t,y, k),
\end{equation}
with 
\begin{equation}
    \norm{\partial_y^m\beta^{in}(t, y, k)}_{L^2} \lesssim_{m} \norm{f_0^k}_{H^{m+3}_k}
\end{equation}
for any integer $m \ge 0$. This finishes the proof of Theorem \ref{interior regularity}.

\subsubsection{Proof of Theorem \ref{regularity near b(0)}}

For the variable $v = b(y)$, $y \in \R$, we define the cutoff function $$\mu^{b0}(v) := \chi^{b0}(y).$$ Similar to the proof of Theorem \ref{interior regularity}, we write for $t \ge 0, v \in \R$
\begin{equation}
    \phi_k(t, v)\mu^{b0}(v) = \big(\phi_k^1(t, v)+\phi_k^2(t, v)+\phi_k^3(t, v)\big)\mu^{b0}(v).
\end{equation}
For $\phi_k^1(t, v)$, we apply integration by parts and get
\begin{equation}
\begin{split}
    &\phi_k^1(t, v) \\=& -\frac{1}{2\pi i}\lim_{\epsilon\to0+}\int_{-\infty}^\infty e^{-ikwt}\bigg[\Theta^{1,-}_{k, \eps}(v-w, w) - \Theta^{1, +}_{k, \eps}(v-w, w)\bigg]dw \\
    =& \frac{1}{2\pi(kt)^2 i}\lim_{\epsilon\to0+}\int_{-\infty}^\infty e^{-ikwt}\frac{d^2}{dw^2}\bigg[\Theta^{1,-}_{k, \eps}(v-w, w) - \Theta^{1,+}_{k, \eps}(v-w, w)\bigg]dw
\end{split}
\end{equation}
It follows from Theorem \ref{thm 49} that
\begin{equation}\label{eq 540}
\begin{split}
    &\frac{d^2}{dw^2}\big[\Theta_{k, \eps}^{1, \iota}(v -w, w)\big] \\=~& A_{k,\eps}^\iota(v-w, w)\frac{\Upsilon_1(w)}{v-w + i\iota\eps} + B_{k,\eps}^\iota(v-w, w)\frac{\Upsilon_1(w)}{b(0) - w + i\iota\eps} \\
    &\quad+ \mathcal{R}_{2, k, \eps}^\iota(v-w, w).
\end{split}
\end{equation}
Here $A_{k,\eps}^\iota(v, w)$ is given by (\ref{major term in v}) and $B^\iota_{k, \eps}(v, w)$ is given by (\ref{major term in w}). The cutoff function $\Upsilon_1(w)$ is defined in (\ref{cutoff}). $\mathcal{R}_{2, k, \eps}^\iota(v, w)$ can be written as
\begin{equation}\label{eqqqq 548}
\begin{split}
    \mathcal{R}_{2, k, \eps}^\iota(v, w) =~& \sum_{1 \le p' + q' \le 3} M^{p', q', \iota}_{2, k,\eps}(v, w)\log^{p'}(v+i\iota\eps)\log^{q'}(b(0) - w +i\iota\eps) \\
    &+ \widetilde{R}^\iota_{2, k,\eps}(v, w),
\end{split}
\end{equation}
where 
\begin{equation}
    \sum_{1 \le p' + q' \le 3}\norm{M^{p', q', \iota}_{2, k,\eps}(v, w)}_{L^2_wH^1_{k,v}} + \norm{\widetilde{R}_{2, k, \eps}^\iota(v, w)}_{L^2_wH^1_{k,v}} \lesssim \norm{f_0^k}_{H^{3}_k}.
\end{equation}

By Theorem \ref{thm 49}, $A_{k,\eps}^\iota(v, w)$ is in $H^1_k$ for both $v$ and $w$. It follows from Lemma \ref{lemma 54} that there exist a function $\alpha_1^{\iota}(t, v, k)$ such that
\begin{equation}
    \lim_{\eps\to 0+} \int_\R e^{-ikwt}A^\iota_{k,\eps}(v-w, w)\frac{\Upsilon_1(w)\mu^{b0}(v)}{v-w + i\iota\eps}dw = e^{-ikvt}\alpha_1^{\iota}(t, v, k),
\end{equation}
and
\begin{equation}
    \norm{\alpha_1^\iota(t,\cdot,k)}_{H^1_k}\lesssim \norm{f_0^k}_{H^3_k}.
\end{equation}
Similarly, we deduce from Lemma \ref{lemma 55} that there exist a function $\beta_1^{\iota}(t, v, k)$ such that
\begin{equation}
     \lim_{\eps\to 0+} \int_\R e^{-ikwt}B^\iota_{k,\eps}(v-w, w)\frac{\Upsilon_1(w)\xi^{b0}(v)}{b(0)-w + i\iota\eps}dw = e^{-ikb(0)t}\beta_1^{\iota}(t, v, k),
\end{equation}
and,
\begin{equation}
    \norm{\beta_1^\iota(t,\cdot,k)}_{H^1_k}\lesssim \norm{f_0^k}_{H^3_k}.
\end{equation}
Set 
\begin{equation}
    \alpha_1(t, y, k) = \frac{1}{2\pi i}(\alpha_1^-(t, v, k) - \alpha_1^+(t, v, k))
\end{equation}
and
\begin{equation}
\beta_1(t, y,k) = \frac{1}{2\pi i}(\beta_1^-(t, v, k) - \beta_1^+(t, v, k)).
\end{equation}
 We have, with one more integration by part,
\begin{equation}\label{eq 545}
\begin{split}
    &\phi_k^1(t,v)\mu^{b0}(v) \\=& ~\frac{1}{(kt)^2}\big[e^{-ikvt}\alpha_1(t,v,k) + e^{-ikb(0)t}\beta_1(t,v,k)\big] \\
    &- \frac{1}{2\pi(kt)^3}\lim_{\eps \to 0+}\int_\R e^{-ikwt}\frac{d}{dw}\bigg[\mathcal{R}^-_{2,k,\eps}(v-w, w) - \mathcal{R}^+_{2,k,\eps}(v-w, w)\bigg]dw.
\end{split}
\end{equation}
For the remaining term $\mathcal{R}_{2, k, \eps}^\iota(v, w)$, it follows from (\ref{eqqqq 548}) and Theorem \ref{thm 49} that
\begin{equation}
\begin{split}
    &\frac{d}{dw} \mathcal{R}_{2, k, \eps}^\iota(v-w, w) \\
    =& \sum_{0\le p' + q' \le 2} A_{k, \eps}^{\iota, 2, p', q'}(v- w, w)\frac{\log^{p'}(v-w+i\iota\eps)\log^{q'}(b(0)-w+i\iota\eps)}{v-w+i\iota\eps} \\
    &+ \sum_{0\le p' + q' \le 2} B_{k, \eps}^{\iota, 2, p', q'}(v- w, w)\frac{\log^{p'}(v-w+i\iota\eps)\log^{q'}(b(0)-w+i\iota\eps)}{b(0)-w+i\iota\eps}\\
    &+ \mathcal{R}_{3, k, \eps}^\iota(v-w, w).
\end{split}
\end{equation}
In addition, we have 
\begin{equation}
    \sum_{0\le p' + q' \le 2}\bigg[\norm{A_{k, \eps}^{\iota, 2, p', q'}(v, w)}_{L^2_wH^{1}_{k,v}} + \norm{B_{k, \eps}^{\iota, 2, p', q'}(v, w)}_{L^2_wH^{1}_{k,v}}\bigg] \lesssim \norm{f_0^k}_{H^3_{k}} .
\end{equation}
The remaining term $\mathcal{R}_{3, k, \eps}^\iota(v-w, w)$ has singularities in the form of $\log^{p'}(v - w+i\iota\eps)\log^{q'}(b(0) - w + i\iota\eps)$ $(1 \le p'+q' \le 5)$.

In general, we can repeat the process in (\ref{eq 545}) and take one more derivative for the remaining term for each step.
We have for the $n$-th remaining term, 
\begin{equation}\label{eq 547}
\begin{split}
    &\frac{d}{dw} \mathcal{R}_{n, k, \eps}^\iota(v-w, w) \\
    =& \sum_{0\le p' + q' \le 2(n-1)} A_{k, \eps}^{\iota, n, p', q'}(v- w, w)\frac{\log^{p'}(v-w+i\iota\eps)\log^{q'}(b(0)-w+i\iota\eps)}{v-w+i\iota\eps} \\
    &+ \sum_{0\le p' + q' \le 2(n-1)} B_{k, \eps}^{\iota, n, p', q'}(v- w, w)\frac{\log^{p'}(v-w+i\iota\eps)\log^{q'}(b(0)-w+i\iota\eps)}{b(0)-w+i\iota\eps}\\
    &+ \mathcal{R}_{n+1, k, \eps}^\iota(v-w, w),
\end{split}
\end{equation}
and for $0\le p' + q' \le 2(n-1)$ the following estimate holds
\begin{equation}
    \norm{A_{k, \eps}^{\iota, n, p', q'}(v, w)}_{L^2_wH^{1}_{k,v}} + \norm{B_{k, \eps}^{\iota, n, p', q'}(v, w)}_{L^2_wH^{1}_{k,v}} \lesssim \norm{f_0^k}_{H^{n+1}_{k}} .
\end{equation}

The appearance of $\log$ terms in the numerator makes the regularity different from (\ref{eq 540}). When $q' > 0$, the term $$\frac{\log^{p'}(v-w+i\iota\eps)\log^{q'}(b(0)-w+i\iota\eps)}{v-w+i\iota\eps}$$ indicates the interaction of interior singularities and boundary singularities. Using Lemma \ref{lemma 54}, we deduce that there exists a function $\alpha_n^{p', q', \iota}(t, v, k)$ such that
\begin{equation}
\begin{split}
    &\lim_{\eps \to 0+} \int_\R e^{-ikwt}A_{k, \eps}^{\iota, n, p', q'}(v- w, w)\frac{\log^{p'}(v-w+i\iota\eps)\log^{q'}(b(0)-w+i\iota\eps)}{v-w+i\iota\eps} \\&\qquad \quad \times \xi^{b0}(v)\Upsilon_1(w)dw \\
    =~& e^{-ikvt}\alpha_n^{p', q', \iota}(t, v, k).
\end{split}
\end{equation}
To study the regularity of $\alpha_n^{p', q', \iota}(t, v, k)$, we consider the following two cases. 
\begin{enumerate}
    \item There is no interaction between interior singularities and boundary singularities. In this case, $q' = 0$. Then it follows from Lemma \ref{lemma 54} and Theorem \ref{thm 49} that for any $\delta > 0$
\begin{equation}
\begin{split}
    \norm{\alpha_n^{p', 0, \iota}(t, \cdot, k)}_{H^{1-\delta}} &\lesssim_{\delta, p',n} (1 + \log^{p'}\la t\ra) \norm{A_{k, \eps}^{\iota, n, p', 0}(v, w)}_{H^1_{k,v}H^{1}_{k,w}} \\
    &\lesssim_{\delta, p', n} (1 + \log^{p'}\la t\ra)\norm{f_0^k}_{H^{n+2}_k}.
\end{split}
\end{equation}
    \item There exists interaction between interior singularities and boundary singularities. In this case, $q' > 0$ and it follows from Lemma \ref{Fourier estimate with log} that $A_{k, \eps}^{\iota, n, p', q'}(v, w)\log^{q'}(b(0)-w+i\iota\eps)$ is only $H^{1/2-}$ in the variable $w$. Therefore, Lemma \ref{lemma 54} and Theorem \ref{thm 49} implies that
    \begin{equation}
\begin{split}
    \norm{\alpha_n^{p', q', \iota}(t, \cdot, k)}_{L^2} &\lesssim_{ p',n}(1 + \log^{p'}\la t\ra) \norm{A_{k, \eps}^{\iota, n, p', q'}(v, w)}_{H^1_{k,v}H^{1}_{k,w}} \\
    &\lesssim_{ p', q', n} (1 + \log^{p'}\la t\ra)\norm{f_0^k}_{H^{n+2}_k}.
\end{split}
\end{equation}
\end{enumerate}
Similarly, we have 
\begin{equation}
\begin{split}
    &\lim_{\eps \to 0+}\int_\R e^{-ikwt}B_{k, \eps}^{\iota, n, p', q'}(v- w, w)\frac{\log^{p'}(v-w+i\iota\eps)\log^{q'}(b(0)-w+i\iota\eps)}{b(0)-w+i\iota\eps}dw \\
    =& e^{-ikb(0)t}\beta_n^{p', q', \iota}(t, v, k).
\end{split}
\end{equation}
When $p' = 0$, we have 
\begin{equation}
    \norm{\beta_n^{0,q',\iota}(t, \cdot, k)}_{H^{1-\delta}} \lesssim_{\delta, q', n} (1 + \log^{q'}\la t\ra)\norm{f_0^k}_{H^{n+2}_k}.
\end{equation}
When $p' > 0$, we have
\begin{equation}
    \norm{\beta_n^{p',q',\iota}(t, \cdot, k)}_{L^2} \lesssim_{ p', q', n} (1 + \log^{q'}\la t\ra)\norm{f_0^k}_{H^{n+2}_k}.
\end{equation}
For $2 \le n \le N-2$, set
\begin{equation}
    \alpha_n(t,y,k) = \frac{1}{2\pi i^{n+2}}\sum_{0 \le p'+q' \le n - 1}\bigg[\alpha_n^{p', q', -}(t, b(y), k) - \alpha_n^{p', q', +}(t, b(y), k)\bigg]
\end{equation}
and
\begin{equation}
    \beta_n(t,y,k) = \frac{1}{2\pi i^{n+2}}\sum_{0 \le p'+q' \le n - 1}\bigg[\beta_n^{p', q', -}(t, b(y), k) - \beta_n^{p', q', +}(t, b(y), k)\bigg].
\end{equation}
We have
\begin{equation}
    \norm{\alpha_n(t, \cdot ,k)}_{L^2} + \norm{\beta_n(t, \cdot ,k)}_{L^2} \lesssim(1+\log^{2(n-1)}\la t\ra) \norm{f_0^k}_{H^{n+2}_k}
\end{equation}
and
\begin{equation}\label{eq 565}
\begin{split}
    &\psi_k^1(t, y)\chi^{b0}(y) \\= ~&e^{-ikb(y)t}\sum_{n = 1}^{N-2}\frac{\alpha_n(t,y,k)}{(kt)^{n+1}} + e^{-ikb(0)t}\sum_{n = 1}^{N-2}\frac{\beta_n(t,y,k)}{(kt)^{n+1}} + \frac{R_{N-2}(t,y,k)}{(kt)^{N-1}}.
\end{split}
\end{equation}
Here 
\begin{equation}
    R_{N-2}(t,y,k) = \lim_{\eps\to 0+}\frac{1}{2\pi i^{N}}\int_\R e^{-ikwt}\big[\mathcal{R}^-_{N-1, k, \eps}(b(y)-w, w) - \mathcal{R}^+_{N-1, k, \eps}(b(y)-w, w)\big] dw.
\end{equation}Noting that $\mathcal{R}^\iota_{N-1, k, \eps}(v-w, w)$ only has log-type singularity, we have 
\begin{equation}
    \norm{R_{N-2}(t, \cdot, k)}_{H^1_k}\lesssim_{N}(1 + \log^{2(N-3)}\la t\ra)\norm{f_0^k}_{H^N_k}.
\end{equation}

For $\psi_k^2(t, y)\chi^{b0}(y)$ and $\psi_k^3(t, y)\chi^{b0}(y)$, decomposition (\ref{eq 565}) still holds since we can keep taking integration by parts and applying Lemma \ref{lemma 54} and Lemma \ref{lemma 55}. But in this case, $v$ and $w$ are supported on non-overlapping intervals. From the proof of Theorem \ref{interior regularity}, we know that $\psi_k^2(t, y)\chi^{b0}(y)$ and $\psi_k^3(t, y)\chi^{b0}(y)$ are smooth in $y$. Therefore, when $y$ is near the boundary $0$, only $\psi^1_k(t,y)$ contributes singularity in $y$. 

If we further assume that $\omega_0(0) = 0$, that is, $f_0^k(b(0)) = 0$ after change of variables, then from Proposition \ref{vanish} we know that $\partial_w \Theta^{1, \iota}_{k, \eps}(v, w)$ is in $H^1_k$ for both variables $v$ and $w$. In addition, the major term $B^\iota_{k,\eps} = 0$ in (\ref{eq 540}). In this situation, we have
\begin{equation}\label{eq: 572}
\begin{split}
     &\frac{d}{dw} \big[\Theta^{1, \iota}_{k, \eps}(v - w, w)\big] \\
     = &~2\Psi_k(v)\frac{\partial_vB(v)}{B(v)}\Theta^{1, \iota}_{k, \eps}(v - w, w) \log(v-w+i\eps) + R_{k,\eps}(v-w, w),
\end{split}
\end{equation}
and
$R_{k,\eps}(v, w)$ has the $\mathcal{F}^{b_0}_{1,1,k,\eps}(f_0^k,2,N)$ type singularity structure.
Now $\alpha_1^{0,0,\iota}(t, v, k)$ is given by
\begin{equation}
    \alpha_1^{0,0,\iota}(t, v, k) = \lim_{\eps \to 0 +}2\Psi(v)\frac{\partial_vB(v)}{B(v)} \int_\R e^{-ikwt}\frac{\Theta^{1, \iota}_{k, \eps}(v - w, w)}{v - w +i\iota\eps} dw.
\end{equation}
Now $\Theta^{1, \iota}_{k, \eps}(v, w)$ is $H^1_k$ in $v$ and $H^2_k$ in $w$. It follows from Lemma \ref{lemma 54} that 
\begin{equation}
    \norm{\alpha_1^{0,0,\iota}(t, \cdot, k)}_{H^2_k} \lesssim \norm{f_0^k}_{H^4_k}.
\end{equation}
$\beta_1^{0,0,\iota}(t, v, k) = 0$ since the major term $B^\iota_{k,\eps} = 0$. 

If we take two more derivatives in $w$ to the equation \eqref{eq: 572}, we observe that there will be no interaction of $\log$ singularities. Therefore, $\alpha_2$ and $\beta_2$ are in $H^1_k$.

\bibliographystyle{plain}
\bibliography{thesis}

\newpage
\appendix
\section{Technical lemmas and proof}

\begin{lemma} \label{Fourier estimate with log}
	Assume $f(x) \in H^1(\R)$ and $f$ is supported on $(-L, L)$ for some $L > 0$. For integer $m \ge 1$ and $\eps \in (-1, 1)\setminus \{0\}$, let 
	\begin{equation}
		h_{m, \eps}(x) = f(x)\log^m(x + i\eps).
	\end{equation}
	We have the following decay estimate of the Fourier transform of $h_{m, \eps}$: for $\xi \in \R$,
	\begin{equation} \label{estimate 10}
		|\widehat{h}_{m, \eps}(\xi)| \lesssim_L \norm{f}_{H^1}\frac{1 + \log^{m-1}\la \xi\ra}{\la \xi \ra}.
	\end{equation}
\end{lemma}
\begin{proof}
Without loss of generality, assume $\eps > 0$. To make notations simple, set
\begin{equation}
    u_\eps(x) := \log(x + i\eps).
\end{equation}
Applying the residue theorem, we get 
\begin{equation}\label{eq a1}
   i \xi \mathcal{F}(u_\eps)(\xi) =   \int_\R \frac{1}{x+i\eps}e^{-ix\xi}dx = -2\pi i e^{-\eps \xi} \mathbf{1}_{\xi > 0}.
\end{equation}
Hence for $|\xi| \ge 1$,
\begin{equation}
    |\mathcal{F}(u_\eps)(\xi)| \lesssim \frac{1}{|\xi|}.
\end{equation}

Let $\varphi(x)$ be a smooth function supported on $(-2L, 2L)$ such that $\varphi(x) = 1$ on $(-L, L)$. We first show that 
\begin{equation}\label{base estimate}
	|\mathcal{F}(\varphi u_\eps)(\xi)| \lesssim_L \frac{1}{\la \xi\ra}.
\end{equation}

Let $\chi(x)$ be a smooth cutoff function supported on $(-2, 2)$ such that $\chi(x) = 1$ on $(-1, 1)$. We have 
\begin{equation}
    \begin{split}
        \mathcal{F}(\varphi u_\eps) &= \int_\R \widehat{\varphi}(\xi - \eta)\chi(\eta) \widehat{u_\eps}(\eta)d\eta + \int_\R \widehat{\varphi}(\xi - \eta)(1 - \chi(\eta)) \widehat{u_\eps}(\eta)d\eta \\
        &:= T_1 + T_2.
    \end{split}
\end{equation}

Assume first $|\xi| \ge 10$. For $T_1$, since $\eta$ is supported on $(-2, 2)$, $\varphi$ is a Schwartz function and $\widehat{u_\eps}$ is a tempered distribution, we have for any integer $N \ge 1$,
\begin{equation}
    |T_1(\xi)| \lesssim_{N, L} |\xi|^{-N}.
\end{equation}
For $T_2$, we have 
\begin{equation}
\begin{split}
        |T_2(\xi)| &\lesssim \int_{|\eta|\ge 1}|\widehat{\varphi}(\xi - \eta)||\widehat{u_\eps}(\eta)|d\eta \\
        &\lesssim \int_{|\eta|\ge 1}|\widehat{\varphi}(\xi - \eta)| \frac{1}{|\eta|}d\eta \\
        &= \bigg(\int_{1 \le |\eta| \le |\xi|/2} + \int_{|\eta| \ge |\xi|/2}\bigg) \frac{|\widehat{\varphi}(\xi - \eta)|}{|\eta|}d\eta \\
        &:= T_{21} + T_{22}.
\end{split}
\end{equation}
For $T_{21}$ we have 
\begin{equation}
\begin{split}
        |T_{21}(\xi)| &\lesssim |\widehat{\varphi}(\xi / 2)|\int_{1 \le |\eta| \le |\xi|/2} \frac{1}{|\eta|}d\eta \\
        &\lesssim \frac{1}{\la \xi \ra}.
\end{split}
\end{equation}
For $T_{22}$ we have
\begin{equation}
    \begin{split}
        |T_{22}(\xi)| &\lesssim \frac{1}{|\xi|} \int_{\R}|\widehat{\varphi}(\eta)|d\eta \\
        &\lesssim \frac{1}{|\xi|}.
    \end{split}
\end{equation}

For $|\xi|\le 10$, we note that 
\begin{equation}
    |\mathcal{F}(\varphi u_\eps)(\xi)| \lesssim \norm{\varphi u_\eps}_{L^1} \lesssim 1.
\end{equation}
Hence for $\xi \in \R$, we have
\begin{equation}
    |\mathcal{F}(\varphi u_\eps)(\xi)| \lesssim \frac{1}{\la \xi\ra}.
\end{equation}

We next show that if $f \in H^1(\R)$ and is supported on $(-L, L)$, we have
\begin{equation}
    |\mathcal{F}(f u_\eps)(\xi)| \lesssim \frac{\norm{f}_{H^1}}{\la \xi\ra}.
\end{equation}
In fact, 
\begin{equation}
\begin{split}
        |\mathcal{F}(f u_\eps))(\xi)| &= |\mathcal{F}(f \varphi u_\eps)(\xi)| \\
        &= |\int_\R \widehat{f}(\xi - \eta) \widehat{\varphi u_\eps}(\eta) d\eta|\\
        &\lesssim \norm{f}_{H^1}\bigg(\int_\R  \frac{1}{\la \xi - \eta\ra^2}\frac{1}{\la \eta\ra^2} d\eta\bigg)^{\frac{1}{2}}\\
        &\lesssim \frac{\norm{f}_{H^1}}{\la \xi\ra}.
\end{split}
\end{equation}

Assume now for some $m_0 \ge 1$, the estimate \eqref{estimate 10} holds for $m = m_0$. For the case $m = m_0 + 1$, we have
\begin{equation}
\begin{split}
        &~~~~|\mathcal{F}(f(x) \log^{m_0+1}(x+i\eps))(\xi)| \\
        &= |\mathcal{F}(f(x) \log^{m_0}(x+i\eps)u_\eps(x))(\xi)| \\
        &= |\int_\R \mathcal{F}\big(f(x)\log^{m_0}(x+i\eps)\big)(\xi -\eta) \cdot \widehat{u_\eps}(\eta)d\eta| \\
        &\lesssim_L \norm{f}_{H^1}\int_\R \frac{1 + \log^{m_0 - 1}\la \xi - \eta\ra}{\la \xi - \eta\ra} \frac{1}{\la \eta\ra}d\eta\\
        &= \bigg(\int_{|\eta|<|\xi| / 2} + \int_{|\xi|/2 \le |\eta| \le 2 |\xi|} + \int_{|\eta| > 2|\xi|}\bigg)\frac{1 + \log^{m_0 - 1}\la \xi - \eta\ra}{\la \xi - \eta\ra} \frac{1}{\la \eta\ra}d\eta \\
        &:= T_3 + T_4 + T_5.
\end{split}
\end{equation}

For $T_3$, we have
\begin{equation}
\begin{split}
        |T_3| &\lesssim_L \norm{f}_{H^1}\frac{1 + \log^{m_0 - 1}\la \xi \ra}{\la \xi \ra} \int_{|\eta|<|\xi| / 2} \frac{1}{\la \eta\ra} d\eta \\
        &\lesssim_L \norm{f}_{H^1}\frac{1 + \log^{m_0}\la \xi \ra}{\la \xi \ra}.
\end{split}
\end{equation}
For $T_4$, we have
\begin{equation}
    \begin{split}
        |T_4| &\lesssim_L \frac{\norm{f}_{H^1}}{\la \xi \ra}\int_{\frac{|\xi|}{2} < |\eta|\le 2|\xi|} \frac{1 + \log^{m_0 - 1}\la \xi - \eta \ra}{\la \xi - \eta\ra}d\eta \\
			&\lesssim_L\frac{\norm{f}_{H^1}}{\la \xi \ra} \int_{|\eta| \le 3|\xi|} \frac{1 + \log^{m_0-1}\la \eta \ra}{\la \eta\ra}d\eta \\
			&\lesssim_L \norm{f}_{H^1}\frac{1 + \log^{m_0}\la\xi\ra}{\la\xi \ra}.
    \end{split}
\end{equation}
For $T_5$, we have
\begin{equation}
    \begin{split}
        |T_5| &\lesssim \norm{f}_{H^1}\int_{|\eta| \ge 2|\xi|} \frac{1 + \log^{m_0-1}\la \eta \ra}{\la \eta \ra^2} d\eta \\
		&\lesssim \norm{f}_{H^1}\frac{1 + \log^{m_0-1}\la \xi\ra}{\la \xi\ra}.
    \end{split}
\end{equation}

Therefore, the estimate \eqref{estimate 10} holds for the case $m = m_0 + 1$. This finishes the proof of the lemma.
\end{proof}

\end{document}